\def\elsart{0} 
\if \elsart1
\documentclass{elsarticle}
\else
\documentclass{amsart}
\newcommand{\newdefinition}{\theoremstyle{definition}\newtheorem}
\newcommand{\newproof}[2]{}
\newcommand{\ead}{\email}
\fi
\usepackage[OT2,T1]{fontenc}
\usepackage{amsmath,amssymb,microtype}
\usepackage{pifont,MnSymbol}
\usepackage{tikz}
\usepackage[colorlinks=true]{hyperref} 
\usepackage[toc,title]{appendix}

\usetikzlibrary{matrix,arrows,decorations}

\DeclareMathOperator{\Mod}{Mod}
\DeclareMathOperator{\Nil}{Nil}
\DeclareMathOperator{\Fin}{Fin}
\newcommand{\C}{\mathcal C}
\newcommand{\D}{\mathcal D}
\newcommand{\V}{\mathcal V}

\newcommand{\Z}{\mathbf Z}
\newcommand{\op}{{\operatorname{op}}}
\newcommand{\map}{\operatorname{map}}
\newcommand{\id}{\operatorname{id}}
\newcommand{\pow}[1]{[\![#1]\!]}

\numberwithin{equation}{section}

\newtheorem{lemma}[equation]{Lemma}
\newtheorem{thm} [equation]{Theorem}
\newtheorem{corollary} [equation]{Corollary}
\newtheorem{prop} [equation]{Proposition}

\newdefinition{example}[equation]{Example}
\newdefinition{defn}[equation]{Definition} 
\newdefinition{remark}[equation]{Remark}
\newdefinition{notation}[equation]{Notation}
\newproof{proof}{Proof}

\DeclareMathOperator{\Ext}{Ext}

\DeclareMathOperator{\eq}{eq}
\DeclareMathOperator{\coeq}{coeq}
\DeclareMathOperator{\Hom}{Hom}

\DeclareMathOperator{\tor}{tor}

\newcommand{\colim}{\operatornamewithlimits{colim}}
\newcommand{\RKan}{\operatornamewithlimits{RKan}}
\DeclareMathOperator{\const}{const}
\DeclareMathOperator{\Spec}{Spec}

\def\smashop#1_#2{%
\displaystyle{#1_{%
\hbox to 0pt{\hss$\scriptstyle{#2}$\hss}}\;}}
\makeatother

\DeclareMathOperator{\Set}{Set}
\DeclareMathOperator{\Alg}{Alg}
\DeclareMathOperator{\Coalg}{Coalg}

\DeclareMathOperator{\Ind}{Ind}

\newcommand{\N}{\mathbf{N}}

\DeclareMathOperator{\Pro}{Pro}
\DeclareMathOperator{\Fr}{Fr}

\makeatletter
\DeclareRobustCommand{\doubleindex}[1]{%
    {
     \edef\resetfontdimens{\noexpand%
         \fontdimen16\textfont2=\the\fontdimen16\textfont2
         \fontdimen17\textfont2=\the\fontdimen17\textfont2\relax}%
     \fontdimen16\textfont2=2.7pt \fontdimen17\textfont2=2.7pt
     #1
     \resetfontdimens}}
\makeatother

\DeclareMathOperator{\Sym}{Sym}

\def\citep#1#2{\cite[{#1}]{#2}} 

\def\widehat {\mathaccent"03C2 }

\newcommand{\btensor}[2]{\,{}_{#1}\!\otimes_{#2}}
\newcommand{\multtensor}[2]{\mathbin{\sideset{^{#2}}{_{#1}}{\mathop{\boxtimes}}}}

\newcommand{\Sch}[1]{\operatorname{\widehat{Sch}}_{#1}}
\newcommand{\Schf}[1]{\operatorname{\smash{\widehat{\operatorname{Sch}}}}^f_{#1}}
\newcommand{\ModSch}[2]{\sideset{_{#2}}{_{#1}}{\operatorname{Hopf}}}
\newcommand{\fModSch}[2]{\sideset{_{#1}}{_{#2}}{\operatorname{\widehat{\operatorname{Hopf}}}}}
\newcommand{\fModSchf}[2]{\sideset{_{#1}}{_{#2}^f}{\operatorname{\smash{\widehat{\operatorname{Hopf}}}}}}

\newcommand{\fAlgSch}[2]{\sideset{_{#1}}{_{#2}}{\operatorname{\widehat{\Alg}}}}
\newcommand{\fAlgSchf}[2]{\sideset{_{#1}}{_{#2}^f}{\operatorname{\smash{\widehat{\Alg}}}}}

\newcommand{\Bimod}[2]{\sideset{_{#1}}{_{#2}}{\Mod}}
\newcommand{\fBimod}[2]{\sideset{_{#1}}{_{#2}}{\operatorname{\smash{\widehat{\Mod}}}}}
\newcommand{\fBimodf}[2]{\sideset{_{#1}}{_{#2}^f}{\operatorname{\widehat{\Mod}}}}

\newcommand{\indprotensor}{\mathbin{\overline{\otimes}}}
\newcommand{\predual}[1]{{#1}^\vee}
\newcommand{\dual}[1]{{}^{\vee}{#1}}

\newcommand{\tensunit}[2]{\sideset{_{#2}}{_{#1}}{\operatorname{I}}}

\renewcommand{\ell}{\mathpzc{l}}

\newcommand{\opnSm}{\operatorname{Small}}
\newcommand{\Sm}[1]{\opnSm(#1)}
\newcommand{\Lat}{\operatorname{Lat}}
\newcommand{\sm}{\wedge}

\newcommand{\modtensor}[2]{\rtimes}
\newcommand{\modcotensor}[2]{\hom}
\newcommand{\Spf}[1]{\operatorname{Spf}\left({#1}\right)}
\newcommand{\Reg}[1]{\mathcal{O}_{#1}}
\newcommand{\Prim}[1]{\mathcal{P}(#1)}
\newcommand{\Indec}[1]{\mathcal{Q}(#1)}
\newcommand{\Primnoarg}{\mathcal{P}}
\newcommand{\Indecnoarg}{\mathcal{Q}}
\DeclareMathOperator{\Cof}{Cof}
\DeclareMathOperator{\Gr}{Gr}
\DeclareMathOperator{\Top}{Top}

\newcommand{\G}{\mathcal G}
\renewcommand{\S}{\mathbf{S}}

\title{Formal plethories}
\author{Tilman Bauer}
\address{Department of mathematics, Kungliga Tekniska H\"ogskolan\\
Lindstedtsv\"agen 25, 10044 Stockholm, Sweden}
\ead{tilmanb@kth.se}

\date{December 21, 2013}

\begin{document}

\begin{abstract}
Unstable operations in a generalized cohomology theory $E$ give rise to a functor from the category of algebras over $E_*$ to itself which is a colimit of representable functors and a comonoid with respect to composition of such functors. In this paper I set up a framework for studying the algebra of such functors, which I call formal plethories, in the case where $E_*$ is a Pr\"ufer ring. I show that the ``logarithmic'' functors of primitives and indecomposables give linear approximations of formal plethories by bimonoids in the $2$-monoidal category of bimodules over a ring.
\end{abstract}

\if \elsart1
\else
\fi
\maketitle

\tableofcontents

\newpage

\section{Introduction}

Let $k$ be a commutative ring. From an algebro-geometric point of view, the category of representable endofunctors of commutative $k$-algebras can be considered as affine schemes over $k$ with a structure of a $k$-algebra on them. Composition of such representable endofunctors constitutes a non-symmetric monoidal structure $\circ$. A \emph{plethory} is such a representable endofunctor $F$ of $k$-algebras which is a comonoid with respect to $\circ$, i.e. which is equipped with natural transformations $F \to \id$ and $F \to F \circ F$ such that coassociativity and counitality conditions are satisfied. The algebra of plethories was first studied by Tall and Wraith \cite{tall-wraith} and then extended by Borger and Wieland \cite{borger-wieland:plethystic}. The aim of this paper is to extend the theory of plethories to the setting of graded \emph{formal} schemes and to study linearizations of them. The motivation for doing this comes from topology. 

\bigskip
Let $E$ be a homotopy commutative ring spectrum representing a cohomology theory $E^*$. For any space $X$, $E^*(X)$ is naturally an algebra over the ring of coefficients $E_*$ of $E$; furthermore, there is an action
\[
E^n(\underline E_m) \times E^m(X) \to E^n(X)
\]
by unstable operations. Here $\underline E_m$ denotes the $m$th space in the $\Omega$-spectrum associated to $E$. The bigraded $E_*$-algebra $E^*(\underline E_*)$ almost qualifies as the representing object of a plethory, but not quite. In order for $E^*(\underline E_*)$ to have the required structure maps (the ring structure on the spectrum of this ring must come from a coaddition and a comultiplication, for instance), one would need a K\"unneth isomorphism $E^*(\underline E_n) \otimes_{E_*} E^*(\underline E_m) \to E^*(\underline E_n \times \underline E_m)$. This happens almost never in cohomology; one would need a condition such as that $E^*(\underline E_n)$ is a finitely presented flat $E_*$-module. A solution to this is to pass to the category of pro-$E_*$-algebras. We denote by $\Alg_{E_*}$ the category of $\Z$-graded commutative algebras over $E_*$. If $X$ is a CW-complex, we define $\hat E^*(X) \in \Pro-\Alg_{E_*}$ to be the system $\{E^*(F)\}_{F \subseteq X}$ indexed by all finite sub-CW-complexes $F$ of $X$. We then assume:
\begin{equation}\label{flatnessassump}
\hat E^*(\underline E_n) \text{ is pro-flat for all $n$.}
\end{equation}
Note that we do not require that $E^*(F)$ be flat for all finite sub-CW-complexes $F \subseteq \underline E_n$, or even for any such $F$. We merely require that $\hat E^*(\underline E_n)$ is pro-isomorphic to a system that consists of flat $E_*$-modules.

This passing to pro-objects gives the theory of plethories a whole new flavor.

\begin{defn} \label{def:fschemefalgscheme}
Denote by $\Set$ the category of sets and by $\Set^\Z$ the category of $\Z$-graded sets.
Let $k$, $l$ be graded commutative rings. A \emph{(graded) formal scheme over $k$} is a functor $\hat X\colon \Alg_k \to \Set^\Z$ on the category of graded commutative $k$-algebras which is a filtered colimit of representable functors. A \emph{formal $l$-algebra scheme over $k$} is a functor $\hat A\colon \Alg_k \to \Alg_l$ whose underlying functor to $\Set^\Z$ is a formal scheme. Denote the category of formal schemes over $k$ by $\Sch{k}$ and of formal $l$-algebra schemes over $k$ by $\fAlgSch{k}{l}$. Denote by $\Schf{k}$ and $\fAlgSchf{k}{l}$ the full subcategories of functors represented by pro-flat $k$-modules.
\end{defn}

For our main theorems, we restrict ourselves to the case where $k$ is a (graded) Pr\"ufer ring, i.~e. a ring where submodules of flat modules are flat. We hope that this restriction can ultimately be removed, but it is necessary at the moment.

\begin{thm}\label{thm:falgschstructure}
Let $k$ be a graded Pr\"ufer ring. Then the category $\fAlgSchf{k}{k}$ is complete and cocomplete and has a monoidal structure $\circ_k$ given by composition of functors. The category $\Schf{k}$ has an action of $\fAlgSchf{k}{k}$.
\end{thm}

In contrast, the category of ordinary $k$-algebra schemes over $k$ is not complete -- it does not have an initial object, for example.

\begin{defn} \label{def:fplethory}
A \emph{formal plethory} is a comonoid in $\fAlgSchf{k}{k}$ with respect to $\circ_k$. A \emph{comodule} over a formal plethory $\hat P$ is a formal scheme $\hat X \in \Schf{k}$ with a coaction $\hat X \to \hat P \circ_k \hat X$.
\end{defn}

\begin{thm} \label{thm:topplethory}
Let $E$ be a ring spectrum such that $E_*$ is a Pr\"ufer domain and such that \eqref{flatnessassump} holds. Then $\hat E^*(\underline E_n)$ represents a formal plethory, and $\hat E^*(X)$ represents a comodule over this plethory for any space $X$.
\end{thm}

Examples of spectra satisfying these conditions include Morava $K$-theories, integral Morava $K$-theories including complex $K$-theory, and more generally any theories $E$ satisfying \eqref{flatnessassump} such that $E_*$ is a graded PID or Dedekind comain. Theories like $MU$ or $E(n)$ do not satisfy the hypothesis, but Landweber exact theories are treated in \cite{bendersky-curtis-miller,bendersky-thompson}.

Thus formal plethories provide an algebraic framework for studying unstable cohomology operations. Algebraic descriptions of unstable cohomology operations are not new: in \cite{boardman-johnson-wilson}, Boardman, Johnson, and Wilson studied unstable algebras (and unstable modules) in great depth. Their starting point is a comonadic description of unstable algebras: the unstable operations of a cohomology theory $E$ are described by a comonad
\[
U(R) = \mathsf{FAlg}(E^*(\underline E_*),R),
\]
and unstable algebras are coalgebras over this comonad \cite[Chapter 8]{boardman-johnson-wilson}. Here $\mathsf{FAlg}$ \label{falgdefn} denotes the category of complete filtered $E_*$-algebras, where the filtration of $E^*(X)$ for a space $X$ is given by the kernels of the projection maps $E^*(X) \to E^*(F)$ to finite sub-CW-complexes.

It requires work to make such a comonadic description algebraically accessible, i.~e. to represent it as a set with a number of algebraic operations. If one disregards composition of operations as part of the structure on $E^*(\underline E_*)$, one arrives (dually) at the notion of a \emph{Hopf ring} or \emph{coalgebraic ring} \cite{wilson:hopf-rings,ravenel-wilson:morava,ravenel-wilson:hopf-ring-mu}. Boardman, Johnson, and Wilson made an attempt at incorporating the composition structure into the definition of a Hopf ring and called it an \emph{enriched Hopf ring} \cite[Chapter~10]{boardman-johnson-wilson}, but stopped short of describing the full algebraic structure.

For this, one needs the language of \emph{plethories}. Plethories were first introduced in \cite{tall-wraith} and were extensively studied in \cite{borger-wieland:plethystic} from an algebraic point of view. These plethories, however, do not carry filtrations or pro-structures such as needed for topological applications. A framework for dealing with this, similar to our formal plethories, was developed by Stacey and Whitehouse \cite{stacey-whitehouse:hunting} using algebras with a complete filtration instead of pro-objects. They prove a version of Thm.~\ref{thm:topplethory} under different assumptions (namely, requiring $E_*(\underline E_k)$ to be a free $E_*$-module) and with a weaker conclusion (they show that the \emph{profinitely completed} cohomology ring $E^*(X)$ is a comodule over the plethory). This completion process destroys phantom maps \cite[Chapters~3,~4]{boardman:stable}. Thus in situations where $E^*(X)$ contains phantoms (cohomology classes that are zero on any finite subcomplex), our description is capable of describing the unstable operations on them. (Phantoms are seen by the pro-system; they live in $\lim^i \hat E^*(X)$ for some $i>0$.) If $E_*$ is a graded field, there are never any phantoms in $E^*(X)$ \cite[Thm.~4.14]{boardman:stable}, and therefore our comodule category and that of Stacey-Whitehouse are equivalent.

\begin{example}
Let $E=K_{(p)}$ be $p$-local complex $K$-theory. Then $E_*=\Z_{(p)}[u,u^{-1}]$ is a Pr\"ufer domain. 
Since $\underline E_{2n} = \Z \times BU_{(p)}$ and $BU$ can be built from only even cells, $\hat E^*(\underline E_{2n}) = \hat E^*(\Z \times BU)$ is pro-free. Similarly, $\underline E_{2n+1} = U_{(p)}$ is the union of $U(n)_{(p)}$, which have free $E$-cohomology, so also $\hat E^*(\underline E_{2n+1})$ is pro-free. Thus $E$ satisfies \eqref{flatnessassump}. There are many spaces $X$ such that $E^*(X)$ has phantoms. Take for example $X=K(\Z[\frac1p],1)$, whose $p$-localization is $K(\mathbf Q,1)$. Then
\[
X = \operatorname{hocolim}(\S^1 \xrightarrow{p} \S^1 \xrightarrow{p} \cdots)
\]
is a fitration by finite subcomplexes and hence $\hat E^1(X,1) = \{\cdots \to Z_{(p)} \xrightarrow{p} Z_{(p)}\}$. The inverse limit of this system is trivial, so all classes are phantoms with $\tilde E^2(X) \cong \lim^1 \hat E^1(X) \cong \mathbf Q_p/\mathbf Q$.
\end{example}

The importance of having an algebraic description for unstable cohomology operations, including the tricky composition pairing, comes to a great extent from the existence of an unstable Adams-Novikov spectral sequence
\begin{equation}\label{unstableanss}
\Ext_{\G}^s(\hat E^*X,E^*(\S^t)) \Longrightarrow \pi_{t-s}(X\hat{{}_E})
\end{equation}
converging conditionally to the homotopy groups of the unstable $E$-completion of $X$. The $\Ext$ group on the left is the a nonlinear comonad-derived functor of homomorphisms of comodules over a comonad $\G$, a category which is closely related to plethory comodules. Note that if $E_*$ is not a graded field, the plethory module category of \cite{stacey-whitehouse:hunting} will in general be different from ours and may therefore not produce the correct $E^2$-term for the unstable Adams spectral sequence.

Since this $E^2$-term is generally very difficult to compute, it is desirabe to find good linear approximations to the formal plethory represented by $\hat E^*(\underline E_n)$ and its comodules, such that the $\Ext$ computation takes place in an abelian category. For this, it is useful to introduce the concept of a $2$-monoidal category.

\begin{defn}[{\cite{aguiar-mahajan:monoidal-functors}}]
A \emph{$2$-monoidal category} is a category $\C$ with two monoidal structures $(\otimes,I)$ and $(\circ,J)$ and natural transformations
\[
\zeta\colon (A \circ B) \otimes (C \circ D) \to (A \otimes C) \circ (B \otimes D)
\]
and
\[
\Delta_I\colon I \to I \circ I,\quad \mu_J\colon J \otimes J \to J, \quad \iota_J = \epsilon_I\colon I \to J,
\]
satisfying various compatibility conditions explicated in Section~\ref{sec:fringsfplethories}.

A \emph{bilax monoidal functor} $F\colon \C \to \D$ between $2$-monoidal categories is functor which is lax monoidal with respect to $\otimes$, oplax monoidal with respect to $\circ$, and whose lax and oplax monoidal structures satisfy certain compatibility conditions (cf. Section~\ref{sec:fringsfplethories}).
\end{defn}

Loosely speaking, a $2$-monoidal category is the most general setting in which one can define a bimonoid (an object with a multiplication and a comultiplication that are compatible with each other). A bilax monoidal functor is the most general notion of functor which sends bimonoids to bimonoids.

\begin{example}\label{exa:coprodtwomonoidal}
If $(\C,\otimes,I)$ is a cocomplete monoidal category then $(\C,\otimes,I,\sqcup,\emptyset)$ becomes a $2$-monoidal category (the other monoidal structure being the categorical coproduct). In particular, the category of formal algebra schemes is $2$-monoidal with the composition product and the coproduct.
\end{example}
\begin{example}[{\cite[Example~6.18]{aguiar-mahajan:monoidal-functors}}]\label{exa:bimodtwomonoidal}
For graded rings $k$, $l$, let $\Bimod{k}{l}$ be the category of $k$-$l$-bimodules (with a $\Z \times \Z$-grading).
Then the category $\Bimod{k}{k}$ is $2$-monoidal with respect to the two-sided tensor product $\btensor{k}{k}$ and the tensor product $\circ_k$ over $k$ given by $M \circ_k N = M \otimes_k N$, using the right $k$-module structure on $M$ and the left $k$-module structure on $N$.
\end{example}
\begin{example}\label{exa:fbimodtwomonoidal}
Let $k$ be a Pr\"ufer domain. Then the category $\fBimodf{k}{k}$ of flat \emph{formal bimodules}, i.e. ind-representable \emph{additive} functors $\Mod_k \to \Mod_k$ represented by pro-flat $k$-modules, is $2$-monoidal.
\end{example}

The main result of this paper about linearization of formal plethories concerns the functor of primitives $P\colon \Coalg_k^+ \to \Mod_k$ from co-augmented $k$-coalgebras to $k$-modules and the functor of indecomposables $Q\colon \Alg_k^+ \to \Mod_k$ from augmented $k$-algebras to $k$-modules.

\begin{thm} \label{thm:primindecbimodulefunctors}
The functors of primitives and indecomposables extend to bilax monoidal functors
\[
\Primnoarg,\Indecnoarg\colon \fAlgSchf{k}{k} \to \fBimodf{k}{k}.
\]
\end{thm}

This means that a formal plethory $\hat P$ gives rise to two bimonoids $\Prim{\hat P}$, $\Indec{\hat P}$. Their module categories are approximations to the category of $\hat P$-comodules and can be used to study the $E_2$-term of \eqref{unstableanss}, but I will defer these topological applications and computations to a forthcoming paper.

Finally, if one desires to leave the worlds of pro-categories, one can do so after dualization because of the following theorem.

\begin{thm}\label{thm:dualization}
Let $k$ be a Pr\"ufer domain. Then the full subcategory of $\fBimod{k}{k}$ consisting of pro-finitely generated free $k$-modules is equivalent, as a $2$-monoidal category, to the full subcategory of right $k$-flat modules in $\Bimod{k}{k}$ (cf. Example~\ref{exa:bimodtwomonoidal}).
\end{thm}

\begin{corollary}
Let $E$ be a multiplicative homology theory such that $E_*$ is a Pr\"ufer domain and $E_*(\underline E_n)$ is a flat $E_*$-modules for all $n$. Then $PE_*\underline E_*$ is a bimonoid in $\Bimod{k}{k}$, and for any pointed space $X$, $PE_*(X)$ is a comodule over it.
\end{corollary}
\begin{proof}
By the flatness condition and the Lazard-Govorov theorem \cite{lazard:platitude,govorov:flat}, $\{E_*(F)\}_{F \subseteq \underline E_n}$, where $F$ runs through all finite sub-CW-complexes of $\underline E_n$, is ind-finitely generated free. By the universal coefficient theorem \cite[Thm 4.14]{boardman:stable}, this implies that condition \eqref{flatnessassump} is also satisfied. Thus $\hat E^*(\underline E_*)$ represents a formal plethory by Theorem~\ref{thm:topplethory}, and applying the functor $\Indecnoarg$ gives a bimonoid $Q\hat E^*(\underline E_*)$ by Thm.~\ref{thm:primindecbimodulefunctors}. Under the additional flatness assumption on $PE_*(\underline E_*)$, $Q\hat E^*(\underline E_*)$ is in fact pro-finitely generated free, and Thm.~\ref{thm:dualization} yields that $PE_*(\underline E_*)$ is a bimonoid in $\Bimod{k}{k}$. With a similar reasoning, $PE_*(X)$ is a comodule over this bimonoid.
\end{proof}

\subsection{Outline of the paper}
In Section~\ref{sec:colimofrep}, we set up notation and terminology to deal with the kind of monoidal categories that come up in this context, and with pro-objects and ind-representable functors. In Sections~\ref{sec:formalbimod}--\ref{sec:falgschstructure}, we study formal bimodules and formal plethories and the structure of their categories and prove the first part of Thm.~\ref{thm:falgschstructure}. In Section~\ref{sec:fringsfplethories}, we recall the definition of $2$-monoidal categories and functors between them, define the objects of the title of this paper, along with their linearizations, formal rings and bimonoid, and prove the second part of Thm.~\ref{thm:falgschstructure}. Thm.~\ref{thm:topplethory} is proved in Subsection~\ref{sec:cohopnfalgsch} (Cor.~\ref{cor:cohomologyalgsch}) and Thm.~\ref{thm:topplethorypf}. The long Section~\ref{section:primindec} is devoted to the study of the linearizing functors of primitives and indecomposables and culminates in a proof of Thm.~\ref{thm:primindecbimodulefunctors}. Section~\ref{section:dualization} deals with dualization and the proof of Thm.~\ref{thm:dualization}. The final, short Section~\ref{section:anss} connects the algebra of formal plethories to the unstable Adams-Novikov spectral sequence. There are two appendices: in Appendix~\ref{colimitsapp}, we review some background on ind- and pro-categories and how the indexing categories can be simplified. This is needed for Appendix~\ref{indproapp}, which contains an exposition of how enrichments of categories lift to pro- and ind-categories.

\subsection*{Acknowledgements}
I am grateful to the anonymous referee for pointing out a serious error in an earlier version of the free module scheme construction, and to David Rydh for teaching me about the divided power construction for non-flat rings which one might use to generalize the free module scheme, but which did not make it into this paper.

\section*{List of notations} \addcontentsline{toc}{section}{List of notations}

\begin{list}{}{\labelwidth 8ex \leftmargin 10ex \labelsep 2ex \itemsep 2pt \parsep 0pt}
  \item [{$\boxtimes $}] generic monoidal structure of a $2$-algebra, Def.~\ref {def:twoalgebra}
  \item [{$\circ_k$}] monoidal structure of $\fAlgSch{k}{k}$ by composition of functors (Thm.~\ref{thm:falgschstructure}); tensor product of a $m$-$k$-bimodule with a $k$-$l$-bimodule over $k$ (Ex.~\ref{exa:bimodtwobimod})
  \item [$(\otimes,I)$] generic symmetric monoidal structure, Not.~\ref{not:inthoms}
  \item [$\otimes_l$] regular (graded) tensor product over $l$; tensor product of formal bimodules (Def.~\ref{def:sweedlerprod}); tensor product of formal module schemes (Def.~\ref{def:fmodschtensor})
  \item[$\btensor{k}{l}$] tensor product of two $k$-$l$-bimodules over $k \otimes l$, Ex.~\ref{exa:bimodtwomonoidal}
  \item [{$\indprotensor$}] left $2$-module structure (tensoring) of $\C$ over $\V$, Not.~\ref{not:inthoms}; in particular of $\Pro-\Alg_k^\Z$ over $\Set^\Z$ (page~\pageref{proalgsettensor})
  \item [{$\indprotensor_\Z$}] left $2$-module structure of $\Pro-\Mod_k^\Z$ over $\Mod_\Z^\Z$, \eqref{modtwobimodule}
  \item [$\indprotensor_l$] left $2$-module structure of $\fBimod{k}{l}$ over $\Mod_l$, Lemma~\ref{fbimodtwomodule}; left $2$-module structure of $\fModSch{k}{l}$ over $\Mod_l$, Lemma~\ref{fmodschztensoradj}, page~\pageref{fmodschtensoring}
  \item[$(\circ,J)$] second monoidal structure (besides $(\otimes,I)$) in a $2$-monoidal category, Def.~\ref{def:twomonoidalcat}
  \item[$\circ_l$] composition monoidal structure, Def.~\ref{def:circprod}
  \item[$\hat 0$] terminal formal module or algebra scheme, Ex.~\ref{exa:terminalfmodsch}
  
  \item [$\hat A$] formal algebra scheme, Def.~\ref{def:fschemefalgscheme}
  \item [$\hat A_+$] augmented algebra scheme, \eqref{ahatplus}
  \item [{$\Alg_k$}] category of graded commutative $k$-algebras, Def.~\ref{def:fschemefalgscheme}
  \item [{$\Alg_k^+$}] category of augmented $k$-algebras
  \item [{$\fAlgSch {k}{l}$, $\fAlgSchf {k}{l}$}] category of (flat) formal $l$-algebra schemes over $k$, Def.~\ref {def:fschemefalgscheme}
  \item [{$\alpha $}] structure map of a morphism of $2$-bimodules, Def.~\ref {def:morphtwobimod}

  \item [{$\beta $}] structure map of a morphism of $2$-bimodules, Def.~\ref {def:morphtwobimod}

  \item [{$\C $}] generic category, Def.~\ref {def:tworingtwomod}
  \item [{$\C(X,Y)$}] enrichment of $\C$ over $\V$, Not.~\ref{not:inthoms}
  \item[$\Cof$] cofree formal module scheme on a bimodule, page~\pageref{cofdefn}
  
  \item[$\Delta_I$] comonoid structure on $I$ in a $2$-monoidal category, Def.~\ref{def:twomonoidalcat}
\item[$E$, $F$] ring spectra
\item [{$\underline E_m$}] $m$th space in the $\Omega$-spectrum associated to a spectrum $E$
\item [$\epsilon_I$] counit of the comonoid structure on $I$ in a $2$-monoidal category, equal to $\iota_J$, Def.~\ref{def:twomonoidalcat}
  \item[$\mathfrak{F}_l$] category of all functors $\Alg_k \to \Mod_l$, Lemma~\ref{constantfmod}; category of all additive functors $\Mod_k \to \Mod_l$, Lemma~\ref{freebimodadjunction}
  \item[$\Fr_\N$] free formal commutative monoid scheme on a formal scheme, \eqref{diag:cofree}
  \item[$\Fr_\Z^\N$] Grothendieck construction on a formal monoid scheme, \eqref{diag:cofree}
  \item[$\Fr_\Z$] free abelian group on a set; free formal abelian group on a formal scheme, Thm.~\ref{cofreehopfalgebra}
  \item[$\Fr_l(X)$] free $l$-module on a set $X$; free formal $l$-module scheme on a formal scheme $X$, Cor.~\ref{fmodschstructure}
  \item [{$\hom$}] right $2$-module structure (cotensoring) of $\C$ over $\V$, Not.~\ref{not:inthoms}
  \item [$\hom_l$] right $2$-module structure of $\fBimod{k}{l}$ over $\Mod_l$, Lemma~\ref{fbimodtwomodule}; right module structure of $\fModSch{k}{l}$ over $\Mod_l$, Lemma~\ref{fmodschztensoradj}
  \item [$\hom_l^\times$] pairing of $\Alg_l$ with $\fAlgSch{k}{l}$, Section~\ref{sec:falgschstructure}
  \item [{$\hom_\Z$}] right $2$-module structure of $\Pro-\Mod_k^\Z$ over $\Mod_\Z^\Z$, \eqref{modtwobimodule}
  \item[$\fModSch{k}{l}, \fModSchf{k}{l}$] category of (flat) formal $l$-module schemes over $k$, Def.~\ref{def:fmodsch}
  \item[$\fModSch{k}{\N}$] category of formal monoid schemes over $k$, \eqref{diag:cofree}
  \item [$\tensunit{M}{N}$, $\tensunit{k}{l}$] formal bimodule defined by $M \indprotensor_\Z \Spf{N}$, Ex.~\ref{exa:bimodhom}
  \item [$\tensunit{k}{M}$, $\tensunit{k}{l}$] discrete formal module scheme, Def.~\ref{def:tensunit}
  \item [{$\Ind-\C$}] ind-category of a category $\C$, App.~\ref{colimitsapp}
  \item [$\iota_J$] unit of the monoid structure on $J$ in a $2$-monoidal category, equal to $\epsilon_I$, Def.~\ref{def:twomonoidalcat}

  \item [$J_k$] identity functor $\Mod_k \to \Mod_k$, considered as a formal bimodule (Ex.~\ref{exa:circunitmod}); identity functor $\Alg_k \to \Alg_k$, considered as a formal algebra scheme (Ex.~\ref{exa:circunit}); unit for the $\circ_k$-product (Lemma~\ref{circmonoidal})
  \item [{$k$, $l$, $m$}] graded commutative rings; $k$ often a Pr\"ufer domain
  
  \item [$\hat M$, $\hat N$] formal module schemes (Def.~\ref{def:fmodsch}) or formal bimodules (Def.~\ref{def:formalbimodule})
  \item[$\Mod_\Z$] ungraded abelian groups
  \item [{$\Mod_k$}] (graded) modules over a graded commutative ring $k$
  \item [{$\Mod_k^f$}] flat (graded) $k$-modules
  \item[$\ModSch{l}{\Z}$] category of Hopf algebras over $l$, Prop.~\ref{Podotmonoidal}
  \item [{$\Bimod{k}{l}$}] category of ($\Z\times\Z$-graded) $k$-$l$-bimodules, Ex.~\ref{exa:bimodtwomonoidal}, Ex.~\ref{exa:bimodtwobimod}
  \item [{$\fBimod {k}{l}$,$\fBimodf {k}{l}$}] category of (flat) formal $k$-$l$-bimodules, Def.~\ref {def:formalbimodule}
  \item [{$\mu$, $\mu'$}] unit multiplication maps in a $2$-algebra, Lemma~\ref{tensorjuggle}
  \item[$\mu_\times$] multiplication map of a formal algebra scheme, Lemma~\ref{comultintotensor}
  \item[$\mu_J$] monoid structure on $J$ in a $2$-monoidal category, Def.~\ref{def:twomonoidalcat}
  \item [{$\Reg{F}$}] pro-$k$-module associated to a formal bimodule (page~\pageref{regspfmod}), pro-$k$-algebra associated to a formal scheme (page~\pageref{regspfalg})


  \item [{$P$}] functor of primitives, Def.~\ref{def:prim}
  \item [{$\Primnoarg$}] functor of primitives from $\fAlgSch{k}{k}$ to $\fBimod{k}{k}$, Def.~\ref{def:prim}
  \item [{$\Pro-\C$}] pro-category of a category $\C$, App.~\ref{colimitsapp}
  \item [{$\phi $, $\phi _0$}] structure maps of a lax morphism, Def.~\ref {def:laxmorphism}
  \item [{$\psi $, $\psi _0$}] structure maps of a colax morphism, Def.~\ref {def:laxmorphism}
 
  \item [{$Q$}] functor of indecomposables, Def.~\ref{def:indec}
  \item [{$\Indecnoarg$}] functor of indecomposables from $\fModSch{k}{l}$ to $\fBimod{k}{l}$, \eqref{fmodschQ}
 
  \item [{$\Sch{k}$, $\Schf{k}$}] category of (flat) formal schemes over $k$, Def.~\ref {def:fschemefalgscheme}
  \item [{$\Set $}] category of sets, Def.~\ref {def:fschemefalgscheme}
  \item [{$\Set^\Z$}] category of $\Z$-graded sets, Def.~\ref{def:fschemefalgscheme}, Ex.~\ref{exa:twomodoverset}
  \item [{$\opnSm$}] functor sending an object to the directed system of small subobjects, Lemma~\ref{basiclimcolimadj}
  \item [{$\operatorname{Spf}$}] formal bimodule associated to a pro-$k$-module (page~\pageref{regspfmod}), formal  scheme associated to a  pro-$k$-algebra (page~\pageref{regspfalg})
  \item [$\Sym$] symmetric algebra on a module, page~\pageref{cofdefn}
  \item [$\Sigma$] shift functor, Ex.~\ref{exa:twomodoverset}

  \item[$U^\times$] forgetful functor $\fAlgSch{k}{l} \to \fModSch{k}{l}$, page~\ref{forgetalgmod}, Lemma~\ref{odotleftfree}
  \item[$U^l$] forgetful functor $\fModSch{k}{l} \to \Sch{k}$, Thm.~\ref{cofreehopfalgebra}, Cor.~\ref{fmodschstructure}
  \item[$U^\N$] forgetful functor $\fModSch{k}{\N} \to \Sch{k}$, \eqref{diag:cofree}
  \item[$U^\Z_\N$] forgetful functor $\fModSch{k}{\Z} \to \fModSch{k}{\N}$, \eqref{diag:cofree}
  \item[$U^{l'}_l$] forgetful functor $\fModSch{k}{l'} \to \fModSch{k}{l}$, Cor.~\ref{fmodschzinduction}, Cor.~\ref{fmodschinduction}
  \item [{$(\V ,\otimes ,I)$}] A generic closed symmetric monoidal category or $2$-ring, Def.~\ref {def:tworingtwomod}
  \item [{$\V(X,Y)$}] internal hom object of the $2$-ring $\V$, Not.~\ref{not:inthoms}
  
\item [$\hat X$, $\hat Y$] formal schemes, Def.~\ref{def:fschemefalgscheme}
\item [{$\zeta$, $\zeta'$}] tensor juggle maps in a $2$-algebra, Lemma~\ref{tensorjuggle}; structure map in a $2$-monoidal category, Def.~\ref{def:twomonoidalcat}

\end{list}

\section{Functors which are filtered colimits of representable functors} \label{sec:colimofrep}
 
In this section we will set up some category-theoretic terminology to talk about enrichments, monoidal structures, and ind-representable functors.
  
\subsection{Two-algebra}
 
 Throughout this paper, we deal with closed symmetric monoidal categories, categories enriched over them, possibly tensored and cotensored, and categories that have a monoidal structure on top of that. To keep language under control, I will give these notions $2$-names that I find quite mnemonic.
 
\begin{defn}\label{def:tworingtwomod}
A \emph{$2$-ring} is a bicomplete closed symmetric monoidal category. A category $\C$ enriched over a $2$-ring $(\V,\otimes,I)$ is a \emph{left $2$-module} if it is tensored over $\V$, a \emph{right $2$-module} if it is cotensored over $\V$, and a \emph{$2$-bimodule} if it is both.
\end{defn}

\begin{notation}\label{not:inthoms}
We will typically denote the internal hom object of a $2$-ring by $\V(X,Y)$ and the enrichment of a $2$-module by $\C(X,Y)$. We will use the symbol $\otimes$ for the symmetric monoidal structure of $\V$, $\indprotensor$ for the left module structure of $\C$ over $\V$, and $\hom(-,-)$ for the right module structure of $\C$ over $\V$.
\end{notation}

\begin{remark} \label{oppositetwobimod}
If $\V$ is a $2$-ring and $\C$ is a left $2$-module over $\V$ then the opposite category $\C^\op$ is a right $2$-module over $\V$ and vice versa.
\end{remark}

\begin{example}
Any category with arbitrary coproducts (coproducts) is a left (right) $2$-module over the category of sets: the left or right $2$-module structures are given by 
\[
S \indprotensor X = \coprod_{s \in S} X \quad \text{and}\quad \hom(S,X) = \prod_{s \in S} X.
\]
\end{example}

\begin{example}\label{exa:twomodoverset}
The category $\Set^\Z$ of $\Z$-graded sets is a $2$-ring with enrichment
\[
\Set^\Z(X,Y)(n) = \prod_{i \in \Z} \Set(X(i),Y(i+n))
\]
and symmetric monoidal structure
\[
(X \times Y)(n) = \coprod_{i+j=n} X(i) \times Y(j).
\]
The unit object is the singleton in degree $0$.

A left (right) $2$-module over $\Set^\Z$ is a precisely a category $\C$ with coproducts (products) and a $\Z$-action on objects by a shift functor $\Sigma^n\colon \C \to \C$ ($n \in \Z$). The $\Z$-grading on its morphism sets is determined by the shift functor: $\C(X,Y)(n) = \C_0(X,\Sigma^n Y)$, where the right hand side denotes the unenriched homomorphism sets.
\end{example}

\begin{defn}\label{def:morphtwobimod}
A \emph{morphism of $2$-bimodules} $F\colon\C \to \D$ over a fixed $2$-ring $\V$ is an enriched functor $\C \to \D$. This implies the existence of canonical morphisms $\alpha\colon L \indprotensor F(X) \to F(L \indprotensor X)$ and $\beta\colon F(\hom(L,X)) \to \hom(L,F(X))$ given by the adjoints of
\[
L \to \C(X,L \indprotensor X) \xrightarrow{F} \D(F(X),F(L \indprotensor X))
\]
and
\[
L \to \C(\hom(L,X),X) \xrightarrow{F} \D(F(\hom(L,X)),F(X)),
\]
respectively. We call $F$ \emph{left strict} if $\alpha$ is a natural isomorphism, and \emph{right strict} if $\beta$ is a natural isomorphism.
\end{defn}

\begin{example}
Let $\C$, $\D$ be two categories with all coproducts and products and a $\Z$-action by a shift functor $\Sigma^n$. By Example~\ref{exa:twomodoverset}, this is equivalent with being a $2$-bimodule over $\Set^\Z$. Then a functor $F\colon \C \to \D$ is a morphism of $2$-bimodules iff $F$ commutes with the shift functor, i.e. if the maps $\alpha$ and $\beta$ give mutually inverse maps between $\Sigma^nF(X)$ and $F(\Sigma^n X)$.
\end{example}

\begin{defn}\label{def:twoalgebra}
A \emph{$2$-algebra} $\C$ over a $2$-ring $(\V,\otimes,I_\V)$ is a $2$-bimodule $\C$ with a  monoidal structure $(\boxtimes,I_\C)$ such that the functors $- \boxtimes X$ and $X \boxtimes -\colon \C \to \C$ are enriched functors for all $X \in \C$.
\end{defn}

Although $\C \times \C$ is a $\V$-category by the diagonal enrichment, we do not require the functor $\boxtimes\colon \C \times \C \to \C$ to be thus enriched.

To make the structure maps more explicit, the enrichment gives
\begin{align}
\text{a natural map $\alpha\colon L \indprotensor (X \boxtimes Y) \to (L \indprotensor X) \boxtimes Y$} \label{alphamap}\\
\text{a natural map $\beta\colon \hom(L,X) \boxtimes Y \to \hom(L,X \boxtimes Y)$} \label{betamap}
\end{align}

Note that a $2$-algebra, even if it is symmetric, is not required to be closed monoidal ($\boxtimes$ need not have a right adjoint). Note also that the definition of a $2$-algebra is symmetric: if $\C$ is a $2$-algebra over $\V$ then so is $\C^\op$.

\begin{example}\label{twoalgebraoverset}
A $2$-algebra over the category of (ungraded) sets is simply a monoidal category with all coproducts and products. A $2$-algebra over $\Set^\Z$ is a $\Z$-graded category with all (co)products and a monoidal structure $\boxtimes$ which is equivariant under grading shifts in either variable.
\end{example}

\begin{example}\label{exa:bimodtwobimod}
Let $k$ be a (not necessarily commutative) ring, $\V$ the category of abelian groups, and $\C=\Bimod{k}{k}$ the category of $k$-bimodules. Then $\C$ is a $2$-bimodule over $\V$, and the tensor product $\circ_k$ of Ex.~\ref{exa:bimodtwomonoidal} makes $\C$ into a $2$-algebra.
\end{example}

\begin{lemma}\label{tensorjuggle}
Let $(\C,\boxtimes,I_\C)$ be a $2$-algebra over $(\V,\otimes,I_\V)$.
Let $X, Y \in \C$ and $K,L \in \V$. There are natural maps
\[
\mu\colon I_\V \indprotensor I_\C \to I_\C, \quad \zeta\colon (K \otimes L) \indprotensor (M \boxtimes N) \to (K \indprotensor M) \boxtimes (L \indprotensor N)
\]
and
\[
\mu'\colon I_\C \to \hom(I_\V,I_\C), \quad \zeta'\colon \hom(K,M) \boxtimes \hom(L,N) \to \hom(K \otimes L, M \boxtimes N)
\]
which make $\indprotensor$ and $\hom$ monoidal functors $\V \times \C \to \C$.
\end{lemma}

\begin{proof}
The map $\mu$ is adjoint to the map $I_\V \to \D(I_\C,I_\C)$ classifying the unit map of $\D$. The map $\zeta$ is the composite
\begin{align}\label{zetadef}
(K \otimes L) \indprotensor (M \boxtimes N) \cong & K \indprotensor (L \indprotensor (M \boxtimes N))\\
\xrightarrow{K \indprotensor \alpha} & K \indprotensor (M \boxtimes (L \indprotensor N))
\xrightarrow{\alpha}   (K \indprotensor M) \boxtimes (L \indprotensor N)). \notag
\end{align}
The assertion about $\zeta'$ follows from passing to the $2$-algebra $\C^\op$. 
\end{proof}

\begin{defn}\label{def:laxmorphism}
A \emph{lax morphism of $2$-algebras} $F\colon \C \to \D$ over a fixed $2$-ring $\V$ is a morphism of $2$-modules with a natural transformation $\phi\colon F(X) \boxtimes F(Y) \to F(X \boxtimes Y)$ and a morphism $\phi_0\colon I_\D \to F(I_\C)$ which make $F$ into a lax monoidal functor, and which is compatible with the enrichment in the sense that the following diagrams commute:
\[
\begin{tikzpicture}
	\matrix (m) [matrix of math nodes, row sep=2em, column sep=3em]
	{
		I_\V & \C(I_\C,I_\C) & \D(F(I_\C),F(I_\C))\\
		\D(I_\D,I_\D) & \D(I_\D,F(I_\C))\\
	};
	\path[->,font=\scriptsize]
	(m-1-1)	edge node[auto]{$\id_{I_\D}$} (m-2-1)
			edge node[auto]{$\id_{I_\C}$} (m-1-2)
	(m-1-2)	edge node[auto]{$F$}	(m-1-3)
	(m-2-1)	edgenode[auto]{$(\phi_0)_*$} (m-2-2)
	(m-1-3)	edge node[auto]{$\phi_0^*$} (m-2-2);
\end{tikzpicture}
\]
{\small
\[
\begin{tikzpicture}
	\matrix (m) [matrix of math nodes, row sep=2em, column sep=2.0em]
	{
		\C(X,X') & \D(FX,FX') & \D(FX \boxtimes FY, FX' \boxtimes FY)\\
		\C(X \boxtimes Y, X' \boxtimes Y) & \D(F(X \boxtimes Y),F(X' \boxtimes Y)) & \D(FX \boxtimes FY, F(X' \boxtimes Y)).\\
	};
	\path[->,font=\scriptsize]
	(m-1-1)	edge node[auto]{$- \boxtimes Y$} (m-2-1)
			edge node[auto]{$F$} (m-1-2)
	(m-1-2)	edge node[auto]{$- \boxtimes FY$}	(m-1-3)
	(m-2-1)	edgenode[auto]{$F$} (m-2-2)
	(m-2-2)	edge node[auto]{$\phi^*$} (m-2-3)
	(m-1-3)	edge node[auto]{$\phi_*$} (m-2-3);
\end{tikzpicture}
\]
}
Similarly, the previous diagram with appliction of $Y \boxtimes -$ and $FY \boxtimes -$ instead of $- \boxtimes Y$ and $- \boxtimes FY$ is required to commute.

An \emph{oplax morphism of $2$-algebras} is a morphism $F\colon \C \to \D$ such that $F^\op\colon \C^\op \to \D^\op$ is a lax morphism of $2$-algebras. More explicitly, it is a morphism of $2$-modules with a natural transformation $\psi\colon F(X \boxtimes Y) \to F(X) \boxtimes (Y)$ and a morphism $\psi_0\colon F(I_\C) \to I_\D$ making $F$ into an oplax monoidal functor, and which is compatible with the enrichment in a similar sense to the above.

A \emph{strict morphism of $2$-algebras} is a morphism $F$ as above which is both lax and oplax with $\phi=\psi^{-1}$ and $\phi_0 = \psi_0^{-1}$.
\end{defn}

It will be useful later to express the conditions for being a lax/oplax morphism of $2$-algebras in terms of the maps $\mu$, $\zeta$ of Lemma~\ref{tensorjuggle} and $\alpha$, $\beta$ from Def.~\ref{def:morphtwobimod}.

\begin{lemma} \label{laxmorphproperties}
Let $F\colon \C \to \D$ be a lax morphism of $2$-algebras. Then the following diagrams commute:
\[
\begin{tikzpicture}
	\matrix (m) [matrix of math nodes, row sep=2em, column sep=3em]
	{
		I_\V \indprotensor I_\D & I_\V \indprotensor F(I_\C) & F(I_\V \indprotensor I_\C)\\
		I_\D & F(I_\C)\\
	};
	\path[->,font=\scriptsize]
	(m-1-1)	edge node[auto]{$\mu$} (m-2-1)
			edge node[auto]{$\id \indprotensor \phi_0$} (m-1-2)
	(m-1-2)	edge node[auto]{$\alpha$}	(m-1-3)
	(m-2-1)	edgenode[auto]{$\phi_0$} (m-2-2)
	(m-1-3)	edge node[auto]{$F(\mu)$} (m-2-2);
\end{tikzpicture}
\]
{\small
\[
\begin{tikzpicture}
	\matrix (m) [matrix of math nodes, row sep=2em, column sep=1.75em]
	{
		(K \otimes L) \indprotensor (F(X) \boxtimes F(Y)) & (K \otimes L) \indprotensor F(X \boxtimes Y) & F((K \otimes L) \indprotensor (X \boxtimes Y))\\
		(K \indprotensor F(X)) \boxtimes (L \indprotensor F(Y)) & F(K \indprotensor X) \boxtimes F(L \indprotensor Y) & F((K \indprotensor X) \boxtimes (L \indprotensor Y))\\
	};
	\path[->,font=\scriptsize]
	(m-1-1)	edge node[auto]{$\zeta_\D$} (m-2-1)
			edge node[auto]{$\id \indprotensor \phi$} (m-1-2)
	(m-1-2)	edge node[auto]{$\alpha$}	(m-1-3)
	(m-2-1)	edgenode[auto]{$\alpha \boxtimes \alpha$} (m-2-2)
	(m-2-2)	edge node[auto]{$\phi$} (m-2-3)
	(m-1-3)	edge node[auto]{$F(\zeta_\C)$} (m-2-3);
\end{tikzpicture}
\]
}
\end{lemma}
We leave the formulation of the analogous other three assertions (for oplax morphisms, and for $\zeta'$ and $\beta$, and $\hom$ for lax and oplax morphisms) to the reader, along with the proofs, which are standard exercises in adjunctions. \qed

\subsection{Ind-representable functors}

For any category $\C$, denote by $\Ind-\C$ its ind-category, whose objects are diagrams $I \to \C$ with $I$ a small filtering category, and by $\Pro-\C$ its pro-category, whose objects are diagrams $J \to \C$ with $J$ a small cofiltering category. See App.~\ref{colimitsapp} for the definition of morphisms, along with possible simplifications on the type of indexing categories we need to allow.

We collect some easy limit and colimit preservation properties and adjoints in the following lemma.
\begin{lemma} \label{basiclimcolimadj}
Let $\V$ be a $2$-ring (Def.~\ref{def:tworingtwomod}) with a set of small generators.
\begin{enumerate}
\item Filtered colimits commute with finite limits in $\V$. \label{vlimcolimcommute}
\item Let $\opnSm\colon \V \to \Ind-\V$ be the functor that sends an object $V \in \V$ to the ind-object consisting of all small subobjects of $\V$. Then there are adjunctions 
\[
\opnSm \vdash (\colim\colon \Ind-\V \to \V) \vdash (\V \hookrightarrow \Ind-V)
\]
In particular, the inclusion functor $\V \to \Ind-\V$ commutes with limits, the colimit functor $\Ind-\V \to \V$ commutes with all limits and colimits, and $\opnSm$ commutes with all colimits. \label{vadjunctions}
\item If $\C$ is a $\V$-$2$-bimodule then the inclusion functor $\C \to \Pro-\C$ commutes with finite limits. \label{cfinlim}
\end{enumerate}
\end{lemma}
\begin{proof}
\eqref{vlimcolimcommute} holds because $\V$ has a set of small generators. Indeed, if $J$ is a filtered category, $F$ is a finite category, and $X\colon J \times F \to \V$ is a functor then
\[
\V(S,\colim_J \lim_F X) = \colim_J \lim_F \V(S,X) = \lim_F \colim_J \V(S,X) = \V(S,\lim_F \colim_J X) 
\]
for each small generator $S$.

The first adjunction of \eqref{vadjunctions} also uses this fact: let $X\colon I \to \V$ be an element of $\Ind-\V$ and $Y \in \V$. Then
\begin{align*}
\C(Y,\colim_i X(i)) = & \C(\colim \Sm{Y},\colim X(i)) = \lim_{K < Y} \C(K,\colim X(i))\\
\underset{K \text{ small}} = & \lim_{K < Y} \colim_i \C(K,X(i)) = \Ind-\C(\Sm{Y},X),
\end{align*}
where $K$ runs through all small subobjects of $Y$.
Finally, \eqref{cfinlim} is a standard fact for pro-categories \cite{artin-mazur:etale}.
\end{proof}

The following result is proved in Appendix~\ref{indproapp}:
\begin{thm} \label{indproenriched}
If $\V$ is a $2$-ring then so is $\Ind-\V$. If $\C$ is a $2$-bimodule over $\V$ then $\Pro-\C$ is a $2$-bimodule over $\Ind-\V$. If $\C$ is a $2$-algebra over $\V$ then $\Pro-\C$ is also a $2$-algebra over $\Ind-\V$. \qed
\end{thm}

\begin{corollary} \label{vproenriched}
Let $\V$ be a $2$-ring with a set of small generators and $\C$ be a $2$-bimodule ($2$-algebra) over $\V$. Then also $\Pro-\C$ is a $2$-bimodule ($2$-algebra) over $\V$.
\end{corollary}
\begin{proof}
The enrichment of $\Pro-\C$ over $\Ind-\V$ becomes an enrichment over $\V$ by passing to the colimit. The left $2$-module structure is given by
\[
L \indprotensor X = \Sm{L} \indprotensor X,
\]
where the left hand symbol $\indprotensor$ is being defined and the right hand symbol is the left $2$-module structure from Thm.~\ref{indproenriched}. Lemma~\ref{basiclimcolimadj}\eqref{vadjunctions} shows that this is indeed a left $2$-module structure. The right $2$-module structure must therefore be defined as 
\[
\hom(L,X) = \hom(\Sm{L},X). 
\]
\end{proof}

In our situation (of a $2$-ring $\V$ with a set of small generators), we thus have an enrichment over $\V$ and one over $\Ind-\V$, but we will use the $\V$-enrichment much more often. The notation $\Pro-\C(X,Y)$ will always refer to the $\V$-enrichment.

\bigskip
Recall that a $\V-$functor $F\colon \C \to \V$ is called \emph{representable} if there is a (necessarily unique) object $A \in \C$ such that $F(X) = \map(A,X) \in \V$ for all $X$.

\begin{defn} \label{def:indrepfunctor}
Let $\C$ be a $2$-bimodule over $\V$. A $\V$-functor $F\colon \C \to \V$ is called \emph{ind-represented} by $A \in \Pro-\C$ if $F = \Pro-\C(A,\iota(-))$ for some $A \in \Pro-\C$, where $\iota\colon \C \hookrightarrow \Pro-\C$ denotes the inclusion as constant pro-objects.
\end{defn}

An ind-representable functor is the same as an ordinary representable $\V$-functor $F'\colon \Pro-\C \to \V$. Indeed, a representable functor $F'$ gives rise to an ind-representable functor $F=F'\circ \iota$. On the other hand, since any representable functor commutes with all limits and any object $X\colon I \to \C$ in $\Pro-\C$ is the $I$-limit in $\Pro-\C$ of the diagram $X \colon I \to \C \hookrightarrow \Pro-\C$, $F'$ is uniquely determined by its images on constant pro-objects. In conjunction with the enriched Yoneda Lemma \cite{kelly:enriched}, this also shows that the ind-representing object $A$ of an ind-representable functor $F$ is uniquely determined by $F$.

%
%


\section{Formal bimodules} \label{sec:formalbimod}

Let $k$ and $l$ be graded commutative rings. We denote by $\Mod_k$ the category of graded right $k$-modules. 

To motivate this chapter, consider one of the following examples:
\begin{itemize}
\item the set $M_p^q$ of all linear maps $l_{-q} \to k_p$, or
\item for two ring spectra $E$, $F$ with coefficients $E_* = k$, $F_*=l$, the cohomology $M_p^q=E^{p-q}(F)$.
\end{itemize}
In both cases, $M$ is a $k$-$l$-bimodule, i.~e. $M \in \Bimod{k}{l}$, graded in such a way that 
\[
M_p^q  \otimes k_{p'} \to M_{p+p'}^q \quad \text{and} \quad l_{q'} \otimes M_p^q \to M_p^{q+q'}.
\]
This bimodule can also be thought of as a representable linear functor $\hat M\colon \Mod_k \to \Mod_l$. Namely, the functor given by $\hat M(X)_q = \Mod_k(M^q,X) \in \Mod_\Z$ obtains a right $l$-module structure from the left $l$-structure on $M$. Conversely, if $M$ is a $k$-module representing a functor into $l$-modules, consider the map
\[
l \xrightarrow{\id \otimes \eta} l \otimes \Mod_k(M,M) \to \Mod_k(M,M),
\]
where $\eta\colon \Z \to \Mod_k(M,M)$ maps $1$ to the identity map. The adjoint of this map gives a left module structure on $M$.

\bigskip
The examples above have one additional piece of structure. In the first case, we can cofilter $M$ by restricting a map to some given finitely presented subgroup of $l$; in the second case, we have a canonical cofiltration of $M$ by $E^*(F')$, where $F'$ is a finite sub-CW-spectrum of $F$. This results in a diagram $I \to \Mod_k^\Z$; the values are $k$-modules but in general not $l$-modules individually. The aim of this section is to give an abstract definition of the category of such objects (called \emph{formal bimodules}) and study its structure.

\begin{defn}\label{def:formalbimodule}
Let $k$, $l$ be graded commutative rings. Consider $\Mod_k$ and $\Mod_l$ as $2$-bimodules over $\Mod_\Z^\Z$, the category of $\Z$-graded abelian groups. A \emph{formal $k$-$l$-bimodule} is an ind-representable functor (Def.~\ref{def:indrepfunctor}) $\hat M\colon \Mod_k \to \Mod_l$. Denote the category of formal bimodules by $\fBimod{k}{l}$ and the full subcategory of those represented by pro-flat $k$-modules by $\fBimodf{k}{l}$.
\end{defn}

We will now study the structure given by formal bimodules explicitly.

\begin{lemma}
For any formal $k$-$l$-bimodule $\hat M$, there is a unique
\[
M = \{M_p^q(i)\}_{i \in I, \; p,q \in \Z} \in \Pro-\Mod_k^\Z
\]
such that
\[
\hat M(N)_q = \Pro-\Mod_k(M^q,N) \in \Mod_\Z \quad \text{for $N \in \Mod_k$}
\]
with the structure of a left $l$-module, i.e. a ring map
\[
\mu\colon l \to \Pro-\Mod_k^\Z(M,M) \in \Mod_\Z^\Z.
\]
\end{lemma}

If $\hat M$ is a formal bimodule, we will denote the associated $M \in \Pro-\Mod_k^\Z$ by $\Reg{\hat M}$, and conversely we will write $\hat M = \Spf{M}$. Here we are borrowing notation from the nonlinear situation of formal schemes discussed in the next section. \label{regspfmod}

Note that there are proper inclusions
\[
\Pro-\Bimod{k}{l} \subsetneq \{\text{$l$-module objects in $\Pro-\Mod_k$}\} \subsetneq \fBimod{k}{l}
\]
An object in all three categories is given by a diagram $\{M(i)\}_{i \in I}$ in $\Mod_k^\Z$, but:
\begin{itemize}
\item In bimodules, compatible maps $l_{q'} \otimes M_p^q(i) \to M_p^{q+q'}(i)$ are required;
\item in $l$-module objects in $\Pro-\Mod_k$, certain maps $l_{q'} \otimes M_p^q(i) \to M_p^{q+q'}(j)$ are required;
\item in $\fBimod{k}{l}$, maps $F \otimes M(i) \to M(j)$ are required for each finitely presented subgroup of $l$, and $j$ may depend on $F$.
\end{itemize}

We want to spell out the last point in the following description, which is a direct result of Theorem~\ref{indproenriched}. The $2$-bimodule structure of $\Pro-\Mod_k^\Z$ over $\Mod_\Z^\Z$ given by Corollary~\ref{vproenriched} turn into a $2$-bimodule structure
on the opposite category (Rem.~\ref{oppositetwobimod})
\begin{equation} \label{modtwobimodule}
\hom_\Z\colon (\Mod_\Z)^\op \times \fBimod{k}{\Z} \to \fBimod{k}{\Z}; \quad \indprotensor_{\Z}\colon \Mod_\Z \times \fBimod{k}{\Z} \to \fBimod{k}{\Z}.
\end{equation}
 
\begin{corollary} \label{unitmapadjmodules}
The category $\fBimod{k}{l}$ is equivalent to the category of algebras over the monad $l \indprotensor_\Z$ on $\fBimod{k}{\Z}$ and to the category of coalgebras over the comonad $\hom_\Z(l,-)$ on $\fBimod{k}{\Z}$.
\end{corollary}

\begin{lemma} \label{flatbimodcocomplete}
The category $\fBimod{k}{\Z}$ is bicomplete, the subcategory $\fBimodf{k}{\Z}$ has all filtered colimits and coproducts, and the inclusion functor $i\colon \fBimodf{k}{\Z} \to \fBimod{k}{\Z}$ preserves them. If $k$ is a graded Pr\"ufer domain then $\fBimodf{k}{\Z}$ is bicomplete and the inclusion functor preserves all colimits.
\end{lemma}
\begin{proof}
The category $\fBimod{k}{\Z} \cong (\Pro-\Mod_k^\Z)^\op$ is bicomplete since $\Mod_k^\Z$ is. Furthermore, finite products of flat modules are flat, so the inclusion $i$ preserves finite coproducts. By definition, $\Pro-\C$ has all cofiltered limits for any category $\C$, so $\fBimodf{k}{\Z}$ has filtered colimits and $i$ preserves them as well.
Now let $k$ be a graded Pr\"ufer domain. This is equivalent to either of the following:
\begin{enumerate}
\item Submodules of flat $k$-modules are flat. \label{submoduleflat}
\item A $k$-module is flat iff it is torsion free. \label{flattorsionfree}
\end{enumerate}
Examples of Pr\"ufer rings are given by fields, PIDs, and Dedekind domains. In this case, equalizers of flat modules are flat by \eqref{submoduleflat} and hence the category of flat $k$-modules has all finite limits, and $i$ preserves them. Thus $i$ preserves all colimits. It even has a right adjoint given by the functor on $\Mod_k^\Z$, sending a $k$-module $M$ to the module $M/\tor(M)$, which is flat by \eqref{flattorsionfree}, and extending it to $\Pro-\Alg_k^\Z$. Thus $\fBimodf{k}{\Z}$ also has limits, given by computing the limit in $\fBimod{k}{\Z}$ and applying the right adjoint.
\end{proof}

\begin{example} \label{exa:bimodhom}
Let $M \in \Mod_l$ and $N \in \Mod_{k}$. Then $\tensunit{M}{N} = M \indprotensor_\Z \Spf{N} \in \fBimod{k}{l}$. The left $l$-action is given, using Corollary~\ref{unitmapadjmodules}, by
\[
l \indprotensor_\Z (M \indprotensor_\Z \Spf{N}) \cong (l \otimes M) \indprotensor_\Z \Spf{N} \xrightarrow{\mu \otimes \id} M \indprotensor_\Z \Spf{N}.
\]
\end{example}

\begin{example} \label{exa:circunitmod}
The identity functor $J_k \colon \Mod_k \to \Mod_k$ is represented by the (non-formal) bimodule $\Reg{J_k} = k \in \Bimod{k}{k}$.
\end{example}

\section{Formal algebra and module schemes}

In the previous section, we took $\Mod_\Z$ as a base category and defined $k$-modules and formal bimodules by adding structure. We can make this more general and also lift the additive structure from sets. To motivate this, we modify the two examples from the previous section. Let $k$, $l$ be graded commutative rings. Consider either
\begin{itemize}
\item for an $M \in \Mod_l$, the set $A_p^q = \hom(M,k)$ of all maps $M_{-q} \to k_p$, or
\item For two ring spectra $E$, $F$ with coefficients $E_* = k$, $F_* = l$, and an $F$-module spectrum $M$ (there is no harm in taking $F$ to be the sphere spectrum), the bigraded set $A_p^q = E^p(\underline M_q)$, the $E$-cohomology of the spaces constituting the $\Omega$-spectrum of $M$.
\end{itemize}
In both cases, $A$ is a $k$-algebra. The addition and $l$-linear structure on $M$ produce additional structure on $A$, but it is not quite a Hopf algebra: in the first case, $\hom(M\times M,k) \neq \hom(M,k) \otimes_k \hom(M,k)$ unless $M$ is finite; in the second case, the same problem occurs as the failure to have a K\"unneth isomorphism in cohomology. The aim of this chapter is to give a definition and study the properties of objects like these. If $M$ additionally has a multiplication (such as if $M=l$ in the first case and $M=F$ in the second), there is even more structure, which we will analyze.
 
\bigskip

We recall the definition of formal schemes (Def.~\ref{def:fschemefalgscheme}):

\begin{defn}
Let $\Sch{k}$ denote the category of affine formal schemes over $k$, i.~e. ind-representable functors $\Alg_k \to \Set^\Z$ (cf. Def.~\ref{def:indrepfunctor}). That is, $\Sch{k} = (\Pro-\Alg_k^\Z)^\op$. Denote by $\Schf{k}$ the full subcategory of functors that can be ind-represented by \emph{flat} $k$-algebras.
\end{defn}
The reader should be aware that the usual definition of an affine formal scheme requires it to be the completion of an ordinary scheme at a subscheme, which is more restrictive than our definition.

The following lemma is analogous to Lemma~\ref{flatbimodcocomplete}:
\begin{lemma} \label{flatschcocomplete}
The subcategory $\Schf{k}$ has all filtered colimits and coproducts, and the inclusion functor preserves them. If $k$ is a graded Pr\"ufer domain then $\Schf{k}$ is bicomplete and the inclusion functor preserves all colimits. \qed
\end{lemma}

\begin{lemma} \label{flatschproducts}
The subcategory $\Schf{k}$ has all products, and the inclusion functor into $\Sch{k}$ preserves them.
\end{lemma}
\begin{proof}
Let $\{A_i\}$ be a family of pro-$k$-algebras indexed by some set $I$, and let $A_i\colon J_i \to \Alg_k^\Z$ be a presentation with a cofiltered category $J_i$. Then $J = \prod_{i \in I} J_i$ is also cofiltered, and the functor
\[
A\colon J \to \Alg_k^\Z, \quad A((j_i)_{i \in I}) = \coprod_{i \in I} A_{j_i}
\]
is a presentation of $\coprod_{i \in I} A_i \in \Pro-\Alg_k^\Z$. This is clearly flat if each $A_{j_i}$ is.
\end{proof}

\begin{defn} \label{def:fmodsch}
Let $k$, $l$ be graded commutative rings. Consider $\Alg_k$ and $\Mod_l$ as $2$-bimodules over $\Set^\Z$. A \emph{formal $l$-module scheme} over $k$ is an ind-representable functor
\[
\hat M\colon \Alg_k \to \Mod_{l}.
\]
Denote by $\fModSch{k}{l}$ the category of formal module schemes and by $\fModSchf{k}{l}$ the full subcategory of those whose underlying formal schemes are in $\Schf{k}$.
\end{defn}

This is an $l$-enriched version of commutative formal group schemes. A formal $\Z$-module scheme is precisely a commutative formal group scheme over~$k$. This definition of a formal module scheme generalizes the notion of \emph{formal $A$-modules} \cite[Chapter 21]{hazewinkel:formal-groups}, which is the special case where $k$ is an $l$-algebra.

\begin{defn}
A \emph{formal $l$-algebra scheme} over $k$ is an ind-representable functor
\[
\hat A\colon \Alg_k \to \Alg_{l}.
\]
Denote by $\fAlgSch{k}{l}$ the category of formal algebra schemes and by $\fAlgSchf{k}{l}$ the full subcategory of those whose underlying formal schemes are in $\Schf{k}$.
\end{defn}

This is a pro-version of what is called a $k$-$l$-biring in \cite{tall-wraith,borger-wieland:plethystic}, but that terminology suggests a similarity with bialgebras, which is something completely different, so we will stick to our terminology.

We will now study the structure given by formal module schemes and formal algebra schemes explicitly. Note that there is a forgetful functor $U^\times\colon \fAlgSch{k}{l} \to \fModSch{k}{l}$, which means that the object representing a formal algebra scheme is equal to the object representing the underlying formal module scheme, but has more structure. \label{forgetalgmod}

\begin{lemma} \label{coringdescription}
For any formal $l$-module scheme $\hat M$ over $k$, there is a unique
\[
A = \{A^q_p(i)\}_{i \in I, p,q \in \Z} \in \Pro-\Alg_k^\Z
\]
such that
\[
\hat M(R)_q = \Pro-\Alg_k(A^q,R) \in \Set \quad \text{for any $R \in \Alg_k$}.
\]
This pro-algebra $A$ comes with the structure of a co-$l$-module, i.e. with maps
\begin{align*}
\psi_+\colon & A^q \to A^{q} \otimes_k A^{q} \in \Pro-\Alg_k &  \text{(coaddition)}\\
\epsilon_0\colon& A^q \to k \in \Pro-\Alg_k &  \text{(cozero)}
\intertext{as well as an additive and multiplicative map}
\lambda\colon & l \to \Pro-\Alg_k^\Z( A_*^*,A_*^*) \in \Set^\Z &  \text{($l$-linear structure)}.
\end{align*}
These maps are such that $\epsilon_0$ is the counit for $\psi_+$, $\lambda_0 = \eta \circ \epsilon_0$,  $\lambda_{-1}$ is the antipode for $\psi_+$, and such that $\psi_+$ is associative and (graded) commutative. Furthermore, $\lambda$ takes values in the sub-graded set of pro-algebra maps that commute with $\psi_+$ and $\epsilon_0$.

A formal $l$-algebra scheme $\hat A$ consists of the same data and in addition an associative, (graded) commutative pro-$k$-algebra map
\begin{align*}
\psi_\times \colon& A_p^q \to \prod_{q'+q''=q} A^{q'} \otimes_k A^{q''} \in \Pro-\Alg_k & \quad \text{(comultiplication)}\\
\intertext{as well as an additive and multiplicative map extending $\epsilon_0$:}
\epsilon\colon&  l_q \to \Pro-\Alg_k( A^q,k) \in \Set & \quad{(unit)}
\end{align*}
such that $\lambda_a$ is comultiplication with $\epsilon_a$, $\epsilon_1$ the counit for $\psi_\times$, and such that $\psi_\times$ distributes over $\psi_+$.
\end{lemma}

As before for modules, we write $A = \Reg{\hat M}$ and $\hat M = \Spf{A}$ if $\hat M$ is ind-represented by $A$, thinking of $A$ as the ring of functions on the formal scheme $\hat M$ and of $\hat M$ as the formal spectrum of the pro-algebra $A$. \label{regspfalg}

Since the constant functor $F(R) = l$ is not ind-representable, we cannot phrase the $l$-module and unit data as a map on representing objects. However, just as in Corollary~\ref{unitmapadjmodules}, it follows from Cor.~\ref{vproenriched} that we can describe $\lambda$ and $\epsilon$ in adjoint form as maps $l \indprotensor A \to A$ and $l \indprotensor A \to k$, respectively, satisfying additional properties, where $\indprotensor$ denotes the left $2$-module structure of $\Pro-\Alg_k^\Z$ over $\Set^\Z$. \label{proalgsettensor}

\begin{example} \label{exa:terminalfmodsch}
The terminal example of an $l$-algebra or module scheme $\hat 0$ over $k$ is the trivial $l$-algebra (or module) $\hat 0(A) = 0$. Here $\Reg{\hat 0} = k$, $\psi_+$ and $\psi_\times$ are the identity, and $\epsilon_\lambda = \id_k$ for all $\lambda \in l$.
\end{example}

Although the initial $l$-algebra is clearly $l$ itself, it is not obvious what the initial $l$-algebra scheme might be, if it exists. The following construction gives the somewhat surprising answer.

\begin{defn} \label{def:tensunit}
Let $M \in \Mod_{l}$. Let $\tensunit{k}{M} = \Spf{A} \in \fModSchf{k}{M}$ be defined, using the right $2$-module structure of $\Pro-\Alg_k$ over $\Set^\Z$, by
\[
A^q = \hom(M_q,k) \in \Pro-\Alg_k^\Z.
\]
Its structure is given as follows: if $+\colon M_q(i) \times M_q(i) \to M_q(j)$ is a component of the addition map on $M$, it gives rise to a coaddition
\[
\psi_+\colon \hom(M(j),k) \to \hom(M(i) \times M(i),k) \to \hom(M(i),k) \otimes_k \hom(M(i),k).
\]
If $0 \in M(i)$, the cozero is given by
\[
\epsilon_0\colon \hom(M(i),k) \to \hom(\{0\},k) = k.
\]
The $l$-module structure $l \indprotensor \hom(M,k) \to \hom(M,k)$ is the adjoint of
\[
(l \times M) \indprotensor \hom(M,k) \xrightarrow{\mu \otimes \id} M \indprotensor \hom(M,k) \xrightarrow{\text{eval}} k.
\]
\end{defn}

Note that this last evaluation map is adjoint to a map of $l$-modules
\[
e_k\colon M \to \Pro-\Alg_k(\hom(M,k),k)
\]
which can be extended to a map
\begin{equation}\label{eamap}
e_A\colon M \to \Pro-\Alg_k(\hom(M,k),A) = \tensunit{k}{M}(A) 
\end{equation}
by the unique $k$-algebra map $k \to A$.

\begin{lemma} \label{constantfmod}
Let $M$ be an $l$-module. Then the formal module scheme $\tensunit{k}{M}$ is the best possible representable approximation to the constant functor $\underline M$ with value $M$. More precisely, let $\mathfrak{F}_l$ denote the category of all functors $\Alg_k \to \Mod_l$. Then for any $\hat N \in \fModSch{k}{l}$, the map $e$ of \eqref{eamap} induces an isomorphism
\[
\fModSch{k}{l}(\tensunit{k}{M},\hat N) \xrightarrow{e^*} \mathfrak{F}_l(\underline M,\hat N).
\]
\end{lemma}
\begin{proof}
Let $f\colon M \to \hat N(k) = \Pro-\Alg_k(\Reg{\hat N},k)$ be the evaluation at the pro-$k$-algebra $k$ of a natural transformation in $\mathfrak F_l(\underline M,F)$. This map is adjoint to a map of pro-algebras $M \indprotensor \Reg{\hat N} \to k$, which in turn is adjoint to a map of pro-algebras $\Reg{\hat N} \to \hom(M,k)$. Since the original map was a map of $l$-modules, this map represents a map of $l$-module schemes. 
\end{proof}

\begin{example}[The initial formal algebra scheme]
The functor $\tensunit{k}{l}$ is a formal $l$-algebra scheme over $k$. We have already seen that it is a formal module scheme, and the comultiplication occurs in the same way. The unit is given by the evaluation map $l \indprotensor \hom(l,k) \to k$.

\begin{lemma}\label{initialfalgsch}
The formal algebra scheme $\tensunit{k}{l}$ is the initial object in $\fAlgSch{k}{l}$.
\end{lemma}
\begin{proof}
The adjoint of the unit map $l \indprotensor \Reg{F} \to k$ for a formal algebra scheme $F$ according to Lemma~\ref{vproenriched} gives the unique map $\Reg{F} \to \hom(l,k)$.
\end{proof}

The functor $\tensunit{k}{l}$ can be described explicitly. It assigns to a $k$-algebra $A$ the set of all $l$-tuples of complete idempotent orthogonal elements of $A$, i.e. tuples $(a_i)_{i \in l}$ with $a_i=0$ for almost all $i \in l$, $\sum_{i} a_i = 1$, $a_ia_j = 0$ if $i\neq j$ and $a_i^2=a_i$. The addition and multiplication are defined by
\[
((a_i)_{i \in l} + (b_j)_{j \in l})_m = \sum_{i+j = m} a_ib_j, \quad ((a_i)_{i \in l} (b_j)_{j \in l})_m = \sum_{ij = m} a_ib_j.
\]
The reader can check that this indeed defines a complete set of idempotent orthogonals if $(a_i)$ and $(b_j)$ are so.

If the algebra $A$ has no zero divisors, i.e. no nontrivial orthogonal elements, then an $l$-tuple of elements as above has to be of the form $\delta_i$, where $(\delta_i)_j = \delta_{ij}$ ($i,\ j \in l$). Thus in this case, $\tensunit{k}{l}(A) = l$, independently of $A$. Thus for these $A$, the map $e_A$ of \eqref{eamap} is an isomorphism.
\end{example}

\begin{example}[The identity functor] \label{exa:circunit}
The identity functor $J_k\colon \Alg_k \to \Alg_k$ is represented by the bigraded $k$-algebra
\[
(\Reg{J_k})_p =k[e_p].
\]
Here $e_p$ has bidegree $(p,p)$, and $k[e_p]$ denotes the free graded commutative algebra on $e_p$, i.e. polynomial if $p$ is even and exterior if $p$ is odd. The coaddition is given by $\psi_+(e_p) = e_p \otimes 1 + 1 \otimes e_p$, and the comultiplication
\[
(\Reg{J_k})_p^p \to \prod_{p'+p''=p} (\Reg{J_k})_{p'}^{p'} \otimes (\Reg{J_k})_{p''}^{p''}
\]
is given by $\psi_\times(e_p) = e_{p'} \otimes e_{p''}$. The unit map is the canonical isomorphism $k \leftrightarrows \Alg_k(k[e_0],k)$.
\end{example}

\begin{example}[The completion-at-zero functor]
The functor $\Nil(X)$ is a \emph{non-unital} $k$-algebra scheme over $k$ represented by the pro-$k$-algebra scheme
\[
(\Reg{\Nil})_p = k\pow{e_p} =_{\text{def}} \left\{k[e_p]/(e_p^n)\right\}_n
\]
with the structure induced by the canonical map $\Reg{J_k} \to \Reg{\Nil}$. Obviously this formal $k$-module scheme cannot have a unital multiplication since the unit element in a ring is never nilpotent. In an algebro-geometric picture, a nontrivial (formal) ring scheme always needs at least two geometric points, $0$ and $1$, whereas the kind of formal groups that appear in topology as cohomology rings of connected spaces only have one geometric point, given by the augmentation ideal.
\end{example}

\begin{example}[The formal completion functor]
The lack of a unit of the previous example can be remedied in much the same way as for ordinary algebras, namely, by taking a direct product with a copy of the base ring. We define the functor
\[
\hat A = \tensunit{k}{k} \times \Nil \colon \Alg_k \to \Alg_k
\]
Obviously, this functor is represented by $\hom(k,k) \otimes_k k\pow{e_p}$. The addition is defined componentwise, whereas the multiplication is defined as follows. Since $\Nil$ is a formal $k$-module scheme, it has a ring map
\[
\mu\colon\hom(k,k) \xrightarrow{\epsilon_0} k \to \Pro-\Alg_k(\Nil,\Nil).
\]
Using this module structure of $\Nil$ over $\Ind(\hom(k,k),-)$, we define the multiplication by
\[
(\lambda_1,x_1) (\lambda_2,x_2) = (\lambda_1\lambda_2,\mu(\lambda_1,x_2)+\mu(\lambda_2,x_1)+x_1x_2).
\]
The unit for the functor $F$ is given by
\[
(\epsilon,\epsilon_0)\colon k \to \Pro-\Alg_k(\hom(k,k),k) \times \Pro-\Alg_k(k\pow{e_p},k).
\]
\end{example}

\begin{example}[The divided power algebra]
Another example of a non-unital $\Z$-algebra scheme over $k$ is given by the divided power algebra. Let $H = \bigoplus_{i=0}^\infty k\langle x_i\rangle$ denote the divided polynomial algebra on a generator $x_1$ in some degree $d$, i.e. the Hopf algebra with
\begin{eqnarray*}
x_p x_q &=& \binom{p+q}{p} x_{p+q}\\
\psi_+(x_p) &=& \sum_{p'+p''=p} x_{p'} \otimes x_{p''}\\
\end{eqnarray*}
Let $H(n)$ denote the quotient algebra $H/(x_{n+1},x_{n+2},\dots)$. Then $\Gamma_p^0 = \{H(n)\}_{n \geq 0}$ and $\Gamma_p^q = 0$ for $q \neq 0$ represents a formal $\Z$-module scheme over $k$. For a generalized construction along these lines, see Section~\ref{sec:fmodschstructure}. This can be given the structure of a \emph{non-unital} algebra scheme by defining $\psi_\times(x_p) = p! (x_p \otimes x_p)$.
\end{example}

\begin{example}[The $\Lambda$-algebra] Let $k$ be a graded commutative ring and $\Lambda = k\pow{c_1,c_2,\dots}$ the power series pro-algebra with $c_i$ in bidegree $(0,2di)$ for some $d \in \Z$. We think of $c_n$ as the $n$th symmetric polynomial in $x_1,x_2,\dots$ (with $c_0=1$). We then define a formal $\Z$-algebra scheme structure on $\Spf{\Lambda}$ by
\[
\prod_{i,j=1}^\infty (1+ t(x_i \otimes 1))(1 + t(1 \otimes x_j)) = \sum_{n=0}^\infty \psi_+(c_n)t^n \in (\Lambda \otimes_k \Lambda) \pow{t}
\]
and
\[
\prod_{i,j=1}^\infty (1+ t(x_i \otimes x_j)) = \sum_{n=0}^\infty \psi_\times(c_n)t^n  \in (\Lambda \otimes_k \Lambda) \pow{t}
\]
For $a \in \Z$, the unit is given by
\[
(1+t)^a = \sum_{n=0}^\infty \epsilon_a(c_n) t^n,\quad \text{or,}\quad  \epsilon_a(n) = \binom{a}{n}.
\]
The polynomial version of this construction represents the functor which associates to a $k$-algebra its ring of big Witt vectors \cite[Chapter 17.2]{hazewinkel:formal-groups}. From a topological point of view, this formal ring scheme is isomorphic with $\Spf{K^0(BU)}$, with the addition and multiplication induced by the maps $BU \times BU \to BU$ classifying direct sums resp. tensor products of vector bundles.

For the sake of concreteness, the coaddition in $\Lambda$ is easily described:
\[
\psi_+(c_p) = \sum_{p'+p''=p} c_{p'} \otimes c_{p''},
\]
whereas the comultiplication does not have a handy closed formula:
\begin{eqnarray*}
\psi_\times(c_1) &=& c_1 \otimes c_1\\
\psi_\times(c_2) &=& c_1^2 \otimes c_2 + c_2 \otimes c_1^2 - 2 c_2 \otimes c_2\\
\psi_\times(c_3) &=& c_1^3 \otimes c_3 + c_3 \otimes c_1^3 - 3 c_3 \otimes c_1c_2 - 3 c_1 c_2 \otimes c_3 + c_1 c_2 \otimes c_1 c_2 \quad \text{etc.}
\end{eqnarray*}
\end{example}

\begin{remark}[The forgetful functor]
An object $\hat M \in \fModSch{k}{l}$ is represented by a pro-$k$-algebra $A = \Reg{\hat M}$ with additional structure, which in particular equips $A$ with a comultiplication
\begin{equation}\label{rightmodule}
l \xrightarrow{\lambda} \Pro-\Alg_k^\Z(A,A) \xrightarrow{\text{forget}} \Pro-\Mod_k^\Z(A,A).
\end{equation}
At first glance one might think that this equips us with a forgetful functor $\fModSch{k}{l} \to \fBimod{k}{l}$. This is not true since \eqref{rightmodule} is not a map of abelian groups in general. There are interesting functors from formal module schemes to bimodules (defined in Section~\ref{section:primindec}), but the forgetful functor is not one of them.
\end{remark}

\subsection{Operations in cohomology theories give formal algebra schemes}\label{sec:cohopnfalgsch}

Let $E$ be a multiplicative cohomology theory represented by an $\Omega$-spectrum $\{\underline E_n\}_{n \in \Z}$. We (re-)define $E$-cohomology as a functor $\hat E^*\colon CW \to \Pro-\Alg_{E_*}$ on CW-complexes given by $\hat E^*(X) = \{E^*(K)\}_K$, where $K$ runs through all finite sub-CW-complexes of $X$. This functor carries somewhat different information that the classical $E$-cohomology of $X$ in that there is a short exact Milnor sequence
\[
0 \to \lim\nolimits^1 \hat E^{*-1}(X) \to E^*(X) \to  \lim \hat E^*(X)  \to 0.
\]
Both ends of this sequence are functors of $\hat E^*(X)$, but the extension class might not be. On the other hand, the $\hat E$-terms cannot be recovered as a functor of $E^*(X)$. 

\begin{lemma}\label{ehatcohomologytheory}
The functor $\hat E^*$ is an Eilenberg-Steenrod cohomology theory with additivity axiom on the category of CW-complexes.
\end{lemma}
Note that, by lack of exactness, the $\lim$- and $\lim\nolimits^1$-terms are not. The target category $\Pro-\Mod_{E_*}$ is an abelian category, thus the axioms make sense.
\begin{proof}
Given a map $f\colon X \to Y$ of CW-complexes, to define the induced map $f^*\colon \hat E^*(Y) \to \hat E^*(X)$, we must produce for every finite sub-CW-complex $K$ of $X$ a finite sub-CW-complex $L$ of $Y$ and a map $E^*(L) \to E^*(K)$ in a compatible way. We choose $L=f(K)$ and $f^*$ as the induced map of the restriction $f|_K$. For homotopy invariance, let $H\colon [0,1] \times X \to Y$ be a cellular homotopy between $f$ and $g$. Given $K \subseteq X$ finite, choose $L = H([0,1] \times K) \subseteq Y$, which is a finite sub-CW complex. Then $H|_K$ gives a homotopy between $f|_K$ and $g|_K$.

To show the long exact sequence axiom, let $X \xrightarrow{i} Y \xrightarrow{j} Y/X$ be a cofibration sequence. It suffices to show that
{\small
\[
\Pro-\Mod_{E_*}(\hat E^*(X),Q) \to \Pro-\Mod_{E_*}(\hat E^*(Y),Q) \to \Pro-\Mod_{E_*}(\hat E^*(Y/X),Q)
\]
}
is exact for every injective $Q \in \Mod_{E_*}$. Thus we need to see that
{\small
\[
\colim_{K \subseteq X} \Mod_{E_*}(E^*(K),Q) \to \colim_{L \subseteq Y} \Mod_{E_*}(E^*(L),Q) \to \colim_{M \subseteq Y/X} \Mod_{E_*}(E^*(M),Q)
\]
}
is exact. Assume that $f\colon E^*(L) \to Q$, $L \subseteq Y$, is a map such that there is an $M \subseteq Y/X$ with $j(L) \subseteq M$ such that $f\circ j^* = 0\colon E^*(M) \to Q$. Choose a finite super-CW complex $L' \supseteq L$ such that $j(L') = M$. This is possible since $j$ is a surjective map of CW-complexes, thus for every cell $\sigma$ in $M-j(L)$ we can add an arbitrary cell of $j^{-1}(\sigma)$ to $L$. Then we have a cofiber sequence $X \cap L' \to L' \to M$. This means that the map $f\colon E^*(L) \to Q$ might not lift to $E^*(X \cap L)$, but the composite $E^*(L') \to E^*(L) \to Q$ will. This proves exactness. The suspension and additivity axioms are immediate.
\end{proof}

\begin{corollary} \label{proflatkunneth}
Assume $X$ is a space such that $\hat E^*(X)$ is pro-flat, i.e. it is pro-isomorphic to a filtered system of flat $E_*$-modules (not necessarily of the form $E^*(K)$ for a sub-CW-complex $K \subseteq X$). Then the K\"unneth map
\[
\hat E^*(X) \otimes_{E_*} \hat E^*(Y) \to \hat E^*(X \times Y)
\]
is an isomorphism.
\end{corollary}
\begin{proof}
If $\hat E^*(X)$ is pro-flat then the functor $\hat E^*(X) \otimes_{E_*} -$ is exact, thus $Y \to \hat E^*(X) \otimes_{E_*} \hat E^*(Y)$ is a cohomology theory, as is $Y \mapsto \hat E^*(X \times Y)$. Since the K\"unneth morphism is an isomorphism on $Y=\{*\}$, it is an isomorphism for any $Y$.
\end{proof}

\begin{corollary} \label{cor:cohomologyalgsch}
Let $E$, $F$ be homotopy commutative $\Omega$-ring spectra. Assume that for every $n \in \Z$, the $n$-th space of $F$, $\underline F_n$, is such that $\hat E^*(\underline F_n)$ is pro-flat. Then $\hat A = \Spf{\hat E^*(\underline F_*)} \in \fAlgSchf{F_*}{E_*}$ is a flat formal $F_*$-algebra scheme over $E_*$.
\end{corollary}
\begin{proof}
The functor $\hat A$ is clearly a formal scheme over $E_*$, and it acquires the structure of a formal $F_*$-algebra scheme by means of the maps
\[
\underline F_n \times \underline F_n \xrightarrow{\text{loop structure}} \underline F_n \quad \text{and} \quad 
\underline F_p \times \underline F_q \xrightarrow{\text{multiplication}} \underline F_{p+q}.
\]
By Cor.~\ref{proflatkunneth}, these maps translate to coalgebra structures on $\hat E^*(\underline F_*)$. The unit map $\epsilon \colon F_* \to \Pro-\Alg_{E_*}(\hat E^*\underline F_*,E_*)$ is induced by application of $E^*$ to an element of $F_* = \pi_* F = [\mathbf{S}^0,\underline F_{-n}]$.
\end{proof}

\section{The structure of the category of formal bimodules} \label{sec:fbimodstructure}

In this section we will study in more detail the algebraic structure of $\fBimod{k}{l}$, the category of formal bimodules, and its subcategory $\fBimodf{k}{l}$ of flat formal bimodules. The main points are that these categories are $2$-bimodules (Def.~\ref{def:tworingtwomod}) over $\Mod_l$ (Lemma~\ref{fbimodtwomodule}) and the existence of an approximation to the objectwise tensor product of modules over $l$, making it into a $2$-algebra (Def.~\ref{def:twoalgebra}, Thm.~\ref{bimodtwoalgebra}). This $2$-algebra structure is half of the $2$-monoidal structure (Def.~\ref{def:twomonoidalcat}) on the category of formal bimodules constructed in Section~\ref{sec:fringsfplethories}; the other monoidal structure will be cooked up from the $\Mod_l$-$2$-module structure defined in this section.

Recall from Cor.~\ref{vproenriched} that $\Pro-\Mod_k$ is a $2$-bimodule over $\Mod_\Z$.

\begin{lemma} \label{fbimodtwomodule}
The category $\fBimod{k}{l}$ is a bicomplete $2$-bimodule over $\Mod_{l}$. If $k$ is a Pr\"ufer domain then also $\fBimodf{k}{l}$ is a bicomplete $2$-bimodule over $\Mod_l$.
\end{lemma}

We will denote the left $2$-module structure by $\indprotensor_l$ and the right $2$-module structure by $\hom_l$. 

\begin{proof}
Let $\hat M$, $\hat N \colon \Mod_k \to \Mod_l$ be objects of $\fBimod{k}{l}$. Then the enrichment is given by the $l$-module of $l$-linear natural transformation $\hat M \to \hat N$, i.e. by the equalizer
\[
\fBimod{k}{l}(\hat M,\hat N) \to \fBimod{k}{\Z}(\hat M, \hat N)\rightrightarrows \fBimod{k}{\Z}(l \indprotensor_\Z \hat M,\hat N),
\]
where the two maps are given by the map $l \indprotensor \hat M \to \hat M$ of Cor.~\ref{unitmapadjmodules} and by
\[
\fBimod{k}{\Z}(\hat M,\hat N) \xrightarrow{l \indprotensor_\Z -}  \fBimod{k}{\Z}(l \indprotensor_\Z \hat M, l \indprotensor_\Z \hat N)
 \xrightarrow{\text{Cor.~\ref{unitmapadjmodules}}} \fBimod{k}{\Z}(l \indprotensor_\Z \hat M, \hat N).
\]
By the commutativity of $l$, this is again an $l$-module.

For $M \in \Mod_l$, the right $2$-module structure $\hom_l(M,\hat N)$ is given objectwise: $\hom_l(M,\hat N)(X) = \Mod_l(M,\hat N(X))$. To see that this is representable, we indentify it with the equalizer in $\fBimod{k}{\Z}$ (which exists by Lemma~\ref{flatbimodcocomplete})
\begin{equation}\label{fbimodrightmodrep}
\hom_l(M,\hat N) \to \hom_\Z(M,\hat N) \rightrightarrows \hom_\Z(l \otimes M,\hat N).
\end{equation}
Here the first map is the multiplication $l \otimes M \to M$ and the second induced by adjunction by the coaction $\hat N \to \hom_\Z(l,\hat N)$ of Cor.~\ref{unitmapadjmodules}.

If $k$ is Pr\"ufer and $\hat M$ flat, we instead take the equalizer \eqref{fbimodrightmodrep} in the category $\fBimodf{k}{\Z}$.

The left $2$-module structure $M \indprotensor_l \hat N$ is given by the coequalizer in $\fBimod{k}{\Z}$
\begin{equation} \label{fbimodleftmodrep}
(M \otimes l) \indprotensor_\Z \hat N \rightrightarrows M \indprotensor_\Z \hat N \to M \indprotensor_l \hat N. 
\end{equation}
This is flat if $\hat N$ is and $k$ is Pr\"ufer.
\end{proof}

The left $2$-module structure approximates the functor $X \mapsto M \otimes_{l} \hat N(X)$ in the following sense.  Clearly $X \mapsto M \otimes_{l} \hat N(X)$ is not representable in general because it is not right exact unless $M$ is flat. 

\begin{lemma} \label{freebimodadjunction}
Denote by $\mathfrak F_l$ the category of all additive functors $\Mod_k \to \Mod_l$, representable or not. For $\hat M, \hat N \in \fBimod{k}{l}$ and $M \in \Mod_{l}$, there is a natural adjunction isomorphism
\[
\mathfrak{F}_l(M \otimes_{l} \hat M(-), \hat N) \cong \fBimod{k}{l}(M \indprotensor_l \hat M,\hat N)
\]
\end{lemma}
\begin{proof}
\begin{align*}
\mathfrak{F}_l(M \otimes_l \hat M(-),\hat N) \cong & \mathfrak{F}_l(\hat M,\Mod_l(M,\hat N(-)))\\ = &
\fBimod{k}{l}(\hat M, \hom_l(M,\hat N)) \cong \fBimod{k}{l}(M \indprotensor_l \hat M,\hat N). 
\end{align*}
\end{proof}

By Cor.~\ref{vproenriched}, $\Pro-\Mod_k^\Z$ is a $2$-algebra (Def.~\ref{def:twoalgebra}) over $\Mod_\Z^\Z$ with tensor product $M \otimes_k N$ for $M$, $N \in \Pro-\Mod_k^\Z$, and thus $\fBimod{k}{\Z}$ is a $2$-algebra over $\Mod_\Z^\Z$; we denote the tensor product there by $\otimes_\Z$. This can be extended to $\fBimod{k}{l}$ for any $k$-algebra $l$:
\begin{defn} \label{def:sweedlerprod}
Let $\hat M$, $\hat N \in \fBimod{k}{l}$. Define their \emph{tensor product} as the coequalizer
\begin{equation}\label{sweedlercoequalizer}
 l \indprotensor_\Z (\hat M \otimes_\Z \hat N) \rightrightarrows \hat M \otimes_\Z \hat N \to \hat M \otimes_l \hat N
\end{equation}
where the two maps are given by
\[
l \indprotensor_\Z (\hat M \otimes_\Z \hat N) \xrightarrow[\eqref{alphamap}]{\alpha} (l \indprotensor_\Z \hat M) \otimes_\Z \hat N \xrightarrow{\mu \otimes \id} \hat M \otimes_\Z \hat N
\]
and the analogous map for $N$. The representing object of $\hat M \otimes_l \hat N$ is a variant of the \emph{Sweedler product} \cite{sweedler:simple-algebras}. Note that if $k$ is Pr\"ufer then $\hat M \otimes_l \hat N$ is flat if both $\hat M$ and $\hat N$ are.

The $l$-action on $\hat M \otimes_l \hat N$ is given by either of the composites of \eqref{sweedlercoequalizer}, which actually factor through $l \indprotensor_\Z (\hat M \otimes_l \hat N) \to \hat M \otimes_l \hat N$ because the two possible composites from $l \indprotensor_\Z (l \indprotensor_\Z (\hat M \otimes_l \hat N))$ coincide.
\end{defn}
One should think about the Sweedler product as the submodule of $\Reg{\hat M} \otimes_k \Reg{\hat N}$ where the $l$-actions on the left and the right factor agree.

%
%

The tensor product equips $\fBimod{k}{l}$ with a symmetric monoidal structure (with unit $\tensunit{k}{l}$, cf. Ex.~\ref{exa:bimodhom}). This symmetric structure is compatible with the enrichment:
\begin{thm}\label{bimodtwoalgebra}
The symmetric monoidal categories $(\fBimod{k}{l},\otimes_l,\tensunit{k}{l})$ and (for $k$ Pr\"ufer) its subcategory $\fBimodf{k}{l}$ are symmetric $2$-algebras over $(\Mod_l,\otimes_{l},l)$. 
\end{thm}
\begin{proof}
We have already seen the $2$-bimodule structure in Lemma~\ref{fbimodtwomodule} and that $\otimes_l$ is a symmetric monoidal structure. It remains to show that the functor $\hat M \mapsto \hat M \otimes_l \hat N$ is enriched over $\Mod_l$, i.e. that the $\Z$-linear map
\[
\fBimod{k}{l}(\hat M,\hat M') \to \fBimod{k}{l}(\hat M \otimes_l \hat N, \hat M' \otimes_l \hat N)
\]
is in fact $l$-linear. But this is true by the construction of the $l$-action on $\hat M \otimes_l \hat N$ and $\hat M' \otimes_l \hat N$, where we are free to choose the map induced by the $l$-action on $\hat M$ and $\hat M'$.
\end{proof}

\section{The structure of the category of formal module schemes} \label{sec:fmodschstructure}

Similarly to the previous section, the category of (flat) formal module schemes has various enrichments and a tensor product which will be constructed in this section. These will be used to define the monoidal structure on the category of formal algebra schemes which is the basis for the definition of a formal plethory (Def.~\ref{def:formalplethory}); the tensor product constructed here is the main technical property to show that the category of formal algebra schemes has coproducts.

Throughout this and the following chapters, we will often assume that $k$ is a graded Pr\"ufer domain (such that submodules of flat modules are flat). 

\begin{lemma} \label{fmodschlimcolim}
\begin{enumerate}
\item The category $\fModSch{k}{l}$ has all limits, and the forgetful functor to $\Sch{k}$ preserves them. \label{fmodschlimits}
\item The subcategory $\fModSchf{k}{l}$ has cofiltered limits and arbitrary products, and the inclusion into $\fModSch{k}{l}$ preserves them. \label{fmodschfproducts}
\item The categories $\fModSch{k}{l}$ and $\fModSchf{k}{l}$ have filtered colimits, and the forgetful functor to $\Sch{k}$ preserves them. \label{fmodschfilteredcolimits} 
\item The categories $\fModSch{k}{l}$ and $\fModSchf{k}{l}$ have arbitrary coproducts, and they agree. \label{fmodschcoprod}
\end{enumerate}
\end{lemma}
\begin{proof}
For \eqref{fmodschlimits}, note that if $\hat M_i$ is a diagram of formal $l$-module schemes over $k$ then its limit in $\Sch{k}$ is given pointwise: $(\lim^{\Sch{k}} M_i)(A) = \lim^{\Set^\Z} M_i(A)$. Since every $M_i(A)$ is an $l$-module, so is $\lim^{\Set^\Z} M_i(A) = \lim^{\Mod_l} M_i(A)$.

For \eqref{fmodschfproducts}, note that the limit constructed in \eqref{fmodschlimits} is represented by the colimit of representing objects, which is not necessarily flat. But coproducts and filtered colimits of flat objects are.

For \eqref{fmodschfilteredcolimits}, let $A\colon I \to \Pro-\Alg_k^\Z$ be a cofiltered diagram of objects representing formal $l$-module schemes. By \cite[Thm.~1.5]{adamek-rosicky:presentable-accessible} and \cite[Thm.~3.3]{isaksen:limits-colimits}, we may assume that $I$ is cofinite and that $A$ has a level representation
\[
A \colon I \times J \to \Alg_k^\Z
\]
where $J$ is some fixed cofiltered category. Then $A$, understood as indexed by the cofiltered category $I \times J$, is the limit in $\Sch{k}$. The $l$-module structure maps $\psi_+$, $\epsilon_0$, and $\lambda$ of Lemma~\ref{coringdescription} lift to $A$, showing that $A$ is in fact also the limit in $\fModSch{k}{l}$. From this description it is also evident that flatness is preserved.

Since finite coproducts and products agree in the additive categories $\fModSch{k}{l}$ and $\fModSchf{k}{l}$, and since infinite coproducts are filtered colimits of finite coproducts, \eqref{fmodschcoprod} follows from \eqref{fmodschfilteredcolimits}.
\end{proof}

By Example~\ref{exa:twomodoverset}, a category with (co)products and a shift functor $\Sigma^n$ is a $2$-bimodule over $\Set^\Z$, and a functor between two such categories is a functor of $2$-bimodules if it commutes with shifts. The category $\fModSch{k}{\Z}$ has a shift functor given by $(\Sigma^n \hat M)(R) = \Sigma^n(\hat M(R))$, using the shift on $\Mod_\Z$, and the forgetful functor obviously commutes with this shift. Thus we have:
\begin{corollary} \label{fmodschbimodule}
The category $\fModSch{k}{l}$ is a $2$-bimodule over $\Set^\Z$, and the subcategory $\fModSchf{k}{l}$ is a left and right strict $2$-sub-bimodule. The forgetful functor to $\Sch{k}$ resp. $\Schf{k}$ is a functor of $2$-bimodules.
\end{corollary}

\subsection{Free/cofree adjunctions for formal module schemes and enrichments}

\begin{thm}\label{cofreehopfalgebra}
Let $k$ be a Pr\"ufer domain. Then the forgetful functor $U^\Z\colon \fModSchf{k}{\Z} \to \Schf{k}$ has a left adjoint $\Fr$. For $\hat X \in \Schf{k}$, $\Fr(\hat X)$ is free in the larger category $\fModSch{k}{\Z}$, i.~e.
\[
\Sch{k}(\hat X,\hat M) \cong \fModSch{k}{\Z}(\Fr(\hat X),\hat M) \quad \text{for all } \hat M \in \fModSch{k}{\Z}.
\]
\end{thm}
The Pr\"ufer condition is not used for the construction of $\Fr_\N$.
\begin{proof}
The proof will proceed in two steps, constructing left adjoints $\Fr^\N$ and $\Fr^\Z$ for each forgetful functor in the factorization
\begin{equation}\label{diag:cofree}
\begin{matrix}
\begin{tikzpicture}
	\matrix (m) [matrix of math nodes, row sep=4em, column sep=5em, text height=1.5ex, text depth=0.25ex]
	{
		\fModSchf{k}{\Z} \\
		\fModSchf{k}{\N} & \Schf{k} ,	\\
	};
	\path[->,font=\scriptsize]
	(m-1-1)		edge node[auto]{$U^\Z_\N$} (m-2-1)
	([xshift=+1.41mm]m-1-1.south east)		edge node[auto]{$U^\Z$} ([xshift=+1.41mm]m-2-2.north west)
	([xshift=-1.41mm]m-2-2.north west)	edge[dashed] node[auto]{$\Fr_\Z$} ([xshift=-1.41mm]m-1-1.south east)
	([yshift=+1mm]m-2-1.east)		edge node[auto]{$U^\N$} ([yshift=+1mm]m-2-2.west)
	([yshift=-1mm]m-2-2.west)	edge[dashed] node[auto]{$\Fr_\N$} ([yshift=-1mm]m-2-1.east);
\end{tikzpicture}
\end{matrix}
\end{equation}
where $\fModSchf{k}{\N}$ denotes the category of flat formal graded commutative monoid schemes. Both adjoints will be free in the larger categories of not necessarily flat formal schemes. We will call the left adjoints of the level of representing objects (where, of course, they are right adjoints) $\Gamma^\Z$ and $\Gamma_\N$, respectively.

The functor $\Gamma^\Z$ should be the cofree cocommutative cogroup object in pro-algebras. In the setting of algebras instead of pro-algebras, an explicit description of the cofree bialgebra on an algebra is quite hard (see \cite[Chapter VI]{sweedler:hopf-algebras}, \cite{peterson-taft} for fields, \cite{takeuchi:commutative-cocommutative} for characteristic $p$, \cite{fox:cofree} for the general noncocommutative case) unless one restricts to \emph{conilpotent} coalgebras \cite{newman-radford:cofree}, and the antipode is yet another problem solved over fields in \cite{agore:cofree-hopf}. More recently, Porst \cite{porst:limits-and-colimits, porst:universal-constructions,porst:subcategories} showed the existence of cofree (cocommutative) Hopf algebras over any ring with adjoint functor theorem methods; unfortunately these proofs do not generalize easily because they require a generator for the category of algebras. The category $\Pro-\Alg_k$ does not have a generator. However, the situation for pro-algebras is simpler in other respects and resembles the conilpotent case. Note that the inclusion $\Alg_k \to \Pro-\Alg_k$ wil \emph{not} send the cofree construction to the cofree construction!

\subsubsection*{Construction of $\Fr_\N$ and $\Gamma^\N$}

Since as a right adjoint, $\Gamma^\N$ has to commute with limits, it suffices to construct $\Gamma^\N(A)$ for a constant flat pro-algebra $A \in \Alg_k^\Z$. For a flat module $M \in \Mod_k^\Z$, denote by $\Gamma^{\N}(M) = \Gamma_k^{\N}(M)$ the coalgebra in $\Pro-\Mod_k^\Z$
\[
\Gamma_k^{\N}(M) = \prod_{n=0}^\infty \Gamma_k^{(n)}(M) = \prod_{n=0}^{\infty} \left(M^{\otimes_k n}\right)^{\Sigma_n},
\]
where the product is taken in the category $\Pro-\Mod_k^\Z$ and the symmetric group $\Sigma_n$ acts by signed permutation of the tensor factors. Since $M$ is flat, so is $(M^{\otimes_k n})^{\Sigma_n}$, and thus we have an isomorphism
\[
\Gamma^{\N}(M) \otimes_k \Gamma^{\N}(M) \cong \prod_{m,n=0}^{\infty} \left(M^{\otimes_k m}\right)^{\Sigma_m} \otimes_k \left(M^{\otimes_k n}\right)^{\Sigma_n} \cong \prod_{m,n=0}^{\infty} \left(M^{\otimes_k m+n}\right)^{\Sigma_m \times\Sigma_n}.
\]
The canonical restriction maps
\[
\Gamma^{(m+n)}(M) = \left(M^{\otimes_k m+n}\right)^{\Sigma_{m+n}} \to \left(M^{\otimes_k m+n}\right)^{\Sigma_m \times\Sigma_n} \cong \Gamma^{(m)}(M) \otimes_k \Gamma^{(n)}(M)
\]
induce a cocommutative comultiplication $\psi\colon \Gamma^{\N}(M) \to \Gamma^{\N}(M) \otimes_k \Gamma^{\N}(M)$.
Furthermore, the projection $\epsilon_0\colon \Gamma^{\N}(M) \to \Gamma^{(0)}(M) = k$ is a cozero. There is also a natural transformation $\pi\colon \Gamma^{\N}(M) \to M$ given by projection to the $\Gamma^{(1)}$-factor.
Note that due to the failure of the tensor product to commute with infinite products, $\lim \Gamma^{\N}(M)$ is not a coalgebra. 

To see that this is the cofree cocommutative coalgebra, let $C$ be a pro-$k$-module with a cocommutative comultiplication $\psi\colon C \to C \otimes C$. Given a coalgebra map $C \to \Gamma_{\N}(M)$, composing with $\pi$ gives a $k$-module map $C \to M$. Conversely, given a $k$-module map $f\colon C \to M$, define a map $\tilde f = (\tilde f)_n\colon C \to \Gamma_{\N}(M) = \prod_n \Gamma_{\N}^n(M)$ by $(\tilde f)_n = (f \otimes \cdots \otimes f)\circ \psi^n$, where $\psi^n\colon C \to \left(C^{\otimes_k n}\right)^{\Sigma_n}$ denotes the $(n-1)$-fold comultiplication of $C$. It is clear that this gives an adjunction isomorphism.

If $A$ is an algebra in $\Alg_k^\Z$ then there is a map
\[
\mu'\colon \Gamma_{\N}(A) \otimes_k \Gamma_{\N}(A) \to A; \quad \mu'(x \otimes y) = \pi(x)\pi(y)
\]
which, by the universal property of $\Gamma_{\N}(A)$ as the cofree coalgebra, lifts to a unique coalgebra map
\[
\mu\colon \Gamma_{\N}(A) \otimes_k \Gamma_{\N}(A) \to \Gamma_{\N}(A).
\]
This makes $\Gamma_{\N}(A)$ a bialgebra (or, more precisely, an object representing a formal commutative monoid scheme) and $\pi$ a $k$-algebra map. It is easy to see that it is the cofree object on the algebra $A$.

This cofree bialgebra does in general not have an antipode. Note that the category of formal groups is a full subcategory of the category of formal monoids: an inverse, if it exists, is unique, and any formal monoid map between formal groups automatically respects inverses (much like a monoid map between groups is a group map).


\subsubsection*{Construction of $\Fr_\Z$ and $\Gamma^\Z$}

To construct $\Gamma_\Z^\N$, we want to mimick the Grothendieck construction for ordinary monoids on the level of formal schemes. Given an ordinary commutative monoid $G$, its Grothendieck construction $\Gr(G)$ can be described as the following coequalizer in sets:
\[
G \times G \times G \rightrightarrows G \times G \to \Gr(G),
\]
where the one map is given by $(g,h,k) \mapsto (g,h)$ and the other by $(g,h,k) \mapsto (g+k,h+k)$. This set coequalizer is a coequalizer in the category of monoids at the same time.

Correspondingly, let $A$ be a flat $k$-algebra and define $\Gamma_k^\Z(A)$ as the equalizer in $\Pro-\Alg_k^\Z$:
\begin{equation}\label{gammazequalizer}
\Gamma_k^\Z(A) \to \Gamma_k^\N(A) \otimes_k \Gamma_k^\N(A) \rightrightarrows \Gamma_k^\N(A) \otimes_k \Gamma_k^\N (A) \otimes_k \Gamma_k^\N (A)
\end{equation}
with the one map given by $\id \otimes \id \otimes 1$ and the other by the composite
\[
\left(\Gamma^\N A\right)^{\otimes_k 2} \xrightarrow{\psi \otimes \psi} \left(\Gamma^\N A\right)^{\otimes_k 4} \xrightarrow{\id \otimes \text{twist} \otimes \id} \left(\Gamma^\N A\right)^{\otimes_k 4} \xrightarrow{\id \otimes \id \otimes \mu} \left(\Gamma^\N A\right)^{\otimes_k 3}.
\]

\begin{lemma}
If $k$ is a Pr\"ufer domain then $\Gamma^\Z_k$ is a functor from $\Schf{k}$ to $\fModSchf{k}{\Z}$.
\end{lemma}
\begin{proof}
Since $\Gamma^\Z_k$ is a submodule of $\Gamma^\N_k \otimes \Gamma^\N_k$ and $\Gamma^\N_k(A)$ is flat if $A$ is, $\Gamma^\Z_k(A)$ is flat whenever $A$ is by the assumption that $k$ is a Pr\"ufer domain. In fact, using this flatness, $\Gamma^\Z_k$ is a sub-bialgebra of $\Gamma^\N_k \otimes_k \Gamma^\N_k$ by the equalizer inclusion. Since the diagram
\[
\begin{tikzpicture}
	\matrix (m) [matrix of math nodes, row sep=2em, column sep=2em, text height=1.5ex, text depth=0.25ex]
	{
		\Gamma^\Z(A) & \Gamma^\N A \otimes_k \Gamma^\N A & \Gamma^\N A \otimes_k \Gamma^\N A \otimes_k \Gamma^\N A\\
		\Gamma_\Z(A) & \Gamma^\N A \otimes_k \Gamma^\N A & \Gamma^\N A \otimes_k \Gamma^\N A \otimes_k \Gamma^\N A\\
	};
	\path[->,font=\scriptsize]
	(m-1-1)	edge (m-1-2)
	(m-2-1)	edge (m-2-2)
	([yshift=-1mm]m-1-2.east)		edge ([yshift=-1mm]m-1-3.west)
	([yshift=+1mm]m-1-2.east)	edge ([yshift=+1mm]m-1-3.west)
	([yshift=-1mm]m-2-2.east)		edge ([yshift=-1mm]m-2-3.west)
	([yshift=+1mm]m-2-2.east)	edge ([yshift=+1mm]m-2-3.west)
	(m-1-2)	edge node[auto]{twist}(m-2-2)
	(m-1-3)	edge node[auto]{twist $\otimes_k \id$} (m-2-3)
	(m-1-1) edge[dashed] node[auto]{$\epsilon_{-1}$} (m-2-1);
\end{tikzpicture}
\]
commutes, the twist map restricts to an antipode $\epsilon_{-1}$. 
\end{proof}

There is a natural transformation of algebras
\[
\pi\colon \Gamma^\Z_k A \to \Gamma^\N_k A \otimes_k \Gamma^\N_k A \xrightarrow{\pi \otimes \epsilon_0} A
\]
which we need to verify to be universal. Thus, given a pro-$k$-algebra $B$ with $\Spf{B} \in \fModSch{k}{\Z}$ and a map $f\colon B \to A$, we need to produce a unique lift to $\Gamma^\Z(A)$. By the previous subsection, $f$ lifts uniquely to a map of bialgebras $\tilde f\colon B \to \Gamma^\N_k(A)$. A lift to $\Gamma^\Z_k(A)$ is then given by
\[
B \xrightarrow{\psi} B \otimes_k B \xrightarrow{\id \otimes \epsilon_{-1}} B \otimes_k B \xrightarrow{\tilde f \otimes \tilde f} \Gamma^\N_k A \otimes_k \Gamma^\N_k A,
\]
and one verifies readily that this takes values in $\Gamma^\Z_k(A)$ by observing that
\[
B \xrightarrow{\psi} B \otimes_k B \xrightarrow{\id \otimes \epsilon_{-1}} B \otimes_k B \rightrightarrows B \otimes_k B \otimes_k B
\]
is a fork. To see that this lift is unique, note that $\otimes_k$ is not only the coproduct, but also the product in the category of bialgebras. Two maps $f,\;g\colon B \rightrightarrows \Gamma^\Z(A)$ can only be different if the composites $B \rightrightarrows \Gamma^\Z_k A \hookrightarrow \Gamma^\N_k A \otimes_k \Gamma^\N_k A \xrightarrow{\pi_i} \Gamma^\N_k A$ are different for at least one of the projections $\pi_1$, $\pi_2$. If they are different for $\pi_1=\id \otimes \epsilon_0$ then also $p \circ f \neq p \circ g$, and we are done. In the other case, $f \circ \epsilon_{-1} \neq g \circ \epsilon_{-1}$, and since $\epsilon_{-1}$ is an isomorphism, we again obtain that $p \circ f \neq p \circ g$.

This concludes the proof of Thm.~\ref{cofreehopfalgebra}.
\end{proof}

\subsubsection*{Monadicity of $\fModSch{k}{\Z}$}

Recall that $\mathfrak{F}_\Z$ is the category of all functors $\Alg_k \to \Mod_{\Z}$. Let $\Fr=\Fr_\Z\colon \Set \to \Mod_{\Z}$ also denote the free $\Z$-module functor. The following lemma says that $\Fr(\hat X)$ is the best representable approximation to the functor $R \mapsto \Fr_\Z(\hat X(R))$ for a flat formal scheme $\hat X \in \Schf{k}$.
\begin{corollary}\label{fmodschfreeadj}
Let $k$ be a Pr\"ufer domain. Then for $\hat X \in \Schf{k}$, $\hat N \in \fModSch{k}{\Z}$, there is a natural isomorphism
\[
\mathfrak{F}_\Z(\Fr \circ \hat X,\hat N) \cong \fModSch{k}{\Z}(\Fr(\hat X),\hat N).
\]
\end{corollary}
\begin{proof}
$\mathfrak F_\Z(\Fr \circ \hat X,\hat N) \cong \Sch{k}(\hat X,\hat N) \underset{\text{Thm.~\ref{cofreehopfalgebra}}}\cong \fModSch{k}{\Z}(\Fr(\hat X),\hat N).$
\end{proof}

\begin{corollary} \label{constantzfmod}
For a Pr\"ufer domain $k$ and $\hat 0 = \Spf{k}$, we have $\Fr(\hat 0) = \tensunit{k}{\Z}$.
\end{corollary}
\begin{proof}
Both functors are characterized as best approximations to the constant functor $\underline \Z$, hence equal. More precisely, note that $\hat 0 \in \Schf{k}$ is the constant functor with value the singleton, thus for any $\hat N \in \fModSch{k}{\Z}$ we have
\[
\fModSch{k}{\Z}(\Fr(\hat 0),\hat N) \underset{\text{Cor.~\ref{fmodschfreeadj}}}\cong \mathfrak F_\Z(\underline \Z,\hat N) \underset{\text{Lemma~\ref{constantfmod}}}\cong \fModSch{k}{\Z}(\tensunit{k}{\Z},\hat N).
\]
\end{proof}



\begin{lemma}\label{coeqreflection}
For a Pr\"ufer domain $k$, the forgetful functor $U^\Z\colon \fModSchf{k}{\Z} \to \Schf{k}$ creates coequalizers. In particular, $\fModSchf{k}{l}$ has all colimits.
\end{lemma}
This means: Let $f,\;g\colon \hat M \to \hat N \in \fModSchf{k}{\Z}$ be two maps of flat formal module schemes with coequalizer $\epsilon\colon U^\Z(\hat N) \to \hat Q$ in $\Schf{k}$. Then $\hat Q$ carries the structure of a flat formal $\Z$-module scheme and is, with this structure, the coequalizer in $\fModSchf{k}{\Z}$.
\begin{proof}
The equalizer $\Reg{\hat Q} \to \Reg{\hat N} \rightrightarrows \Reg{\hat M}$ is pro-flat by Lemma~\ref{flatschcocomplete}, and by the same observation, the comultiplication on $\Reg{\hat Q}$ takes values in $\Reg{\hat Q} \otimes_k \Reg{\hat Q}.$ Since $\fModSchf{k}{\Z}$ has coequalizers by the above and coproducts by Lemma~\ref{fmodschlimcolim}\eqref{fmodschcoprod}, it has all colimits.
\end{proof}

As a result of Duskin's formulation of the Beck monadicity theorem \cite{duskin:beck}, see also \cite[Theorem 4.4.4]{borceux:voltwo}, Thm.~\ref{cofreehopfalgebra}, Lemma \ref{coeqreflection}, and Lemma~\ref{flatschcocomplete} imply
\begin{corollary} \label{fmodschcomonadic}
For a Pr\"ufer domain $k$, the forgetful functor $U^\Z$ is monadic, i.~e. the category $\fModSchf{k}{\Z}$ is equivalent to the category of algebras in $\Schf{k}$ for the monad $U^\Z \circ \Fr_\Z$. In particular, it is bicomplete.\qed
\end{corollary}

\begin{corollary} \label{fmodschcolimadj}
Let $k$ be Pr\"ufer and $\hat M\colon I \to \fModSchf{k}{\Z}$ be a diagram of monoidal formal module schemes with colimit $\colim \hat M$. Denote by $\colim\nolimits^{\mathfrak{F}} \hat M$ the colimit of $\hat M$ as a diagram of functors in $\mathfrak{F}_\Z$. Then for any $\hat N \in \fModSch{k}{\Z}$, we have a natural isomorphism
\[
\mathfrak{F}_\Z(\colim\nolimits^{\mathfrak{F}} \hat M, \hat N) \cong \fModSch{k}{\Z}(\colim \hat M,\hat N).
\]
\end{corollary}
\begin{proof}
\begin{multline*}
\mathfrak{F}_\Z(\colim\nolimits^{\mathfrak{F}} \hat M,\hat N) \cong \lim \mathfrak{F}_\Z(\hat M,\hat N) \\
\underset{\text{Cor.~\ref{fmodschfreeadj}}} = \lim \fModSch{k}{\Z}(\hat M,\hat N) = \fModSch{k}{\Z}(\colim \hat M,\hat N).
\end{multline*}
\end{proof}

\subsubsection*{Free $l$-module schemes}

The morphisms in $\fModSch{k}{\Z}$, i.e. the graded set of natural transformations of $\Mod_{\Z}$-valued functors, obviously form a graded $\Z$-module, and if $\hat M \in \fModSch{k}{\Z}$ and $\hat N \in \fModSch{k}{l}$ then $\fModSch{k}{\Z}(\hat M,\hat N)$ is naturally an $l$-module.

For $M \in \Mod_l$ and $\hat M \in \fModSch{k}{l}$, let $\overline\hom_{l}(M,\hat M)$ denote the functor $\Alg_k \to \Mod_{\Z}$ which sends $R$ to $\Hom_{l}(M,\hat M(R))$. Thus $\overline\hom_{l}(M,\hat M)$ can be expressed, using the right $2$-module structure of Cor.~\ref{fmodschbimodule} and the completeness from Lemma~\ref{fmodschlimcolim}\eqref{fmodschlimits}, as the simultaneous equalizer of

\begin{equation} \label{homadditive}
\hom(M,\hat M) \rightrightarrows \hom(M \times M,\hat M)
\end{equation}
and
\begin{equation} \label{homlinear}
\hom(M,\hat M) \rightrightarrows \hom(M \times l,\hat M),
\end{equation}
where the maps in \eqref{homadditive} are induced by the addition on $M$ and by
\[
\hom(M,\hat M) \xrightarrow{\Delta} \hom(M,\hat M)^2 \to \hom(M^2, \hat M^2) \xrightarrow{\hom(M^2,+)} \hom(M^2,\hat M),
\]
and the maps in \eqref{homlinear} are the map induced by the action $l \times M \to M$ and by
\[
\hom(M,\hat M) \to \hom(M \times l, \hat M \times l) \xrightarrow{\hom(M \times l, \cdot)} \hom(M \times l,\hat M).
\]

If $\hat M \in \fModSchf{k}{l}$ and $M \in \Mod_l$ are both flat then so is $\hom_l(M,\hat M)$. To see this, write $M = \colim F_i$ as a filtered colimit of finitely generated free $l$-modules \cite{govorov:flat,lazard:platitude}. Since $\overline\hom_l(-,\hat M)$ sends colimits to limits, $\overline\hom_l(l,\hat M) \cong M$, and by Lemma~\ref{fmodschlimcolim}\eqref{fmodschfproducts}, $\overline\hom_l(M,\hat M) \in \fModSchf{k}{\Z}$.

If $k$ is Pr\"ufer and $M$ is not flat (but $\hat M$ is), we can take the equalizer $\hom_l(M,\hat M)$ of \eqref{homadditive} and \eqref{homlinear} in $\fModSchf{k}{l}$, which is bicomplete by Cor.~\ref{fmodschcomonadic}. Then $\hom_l(M,-) = \overline \hom_l(M,-)$ if $M$ is flat. We summarize:
\begin{lemma}\label{fmodschztensoradj}
Let $k$ be Pr\"ufer, $\hat N \in \fModSchf{k}{\Z}$, $M \in \Mod_l$, and $\hat M \in \fModSchf{k}{l}$. Then $\hom_l(M,\hat M) \in \fModSchf{k}{\Z}$ and we have adjunctions
\[
\fModSchf{k}{\Z}(\hat N,\hom_{l}(M,\hat M)) \cong \mathfrak{F}_{l}(\hat N(-) \otimes_{\Z} M,\hat M) \cong \Mod_{l}(M,\fModSch{k}{\Z}(\hat N,\hat M)).
\]
\end{lemma}
The functor $\hat N(-) \otimes_{\Z} M$ in the middle is not representable unless $M$ is flat. But there is an optimal approximation by a representable functor:
\begin{lemma}\label{fmodschztensor}
Let $k$ be a Pr\"ufer domain. Then there is a functor 
\[
\indprotensor_{\Z}\colon \fModSchf{k}{\Z} \times \Mod_{l} \to \fModSchf{k}{l}
\]
with a natural transformation of $l$-modules
\[
(\hat N \indprotensor_{\Z} M)(R) \to \hat N(R) \otimes_{\Z} M
\]
which produces an adjunction isomorphism for $\hat M \in \fModSch{k}{l}$:
\[
\fModSch{k}{l}(\hat N  \indprotensor_{ \Z} M,\hat M) \cong \mathfrak{F}_{l}(\hat N(-) \otimes_{ \Z} M,\hat M).
\]
\end{lemma} 
\begin{proof}
Let $N$ be a $\Z$-module. If $S \indprotensor N$ denotes the left module structure of $\Mod_{\Z}$ over sets, we can express the tensor product $M \otimes_{\Z} N$ as a coequalizer
\[
(M \times M) \indprotensor N  \rightrightarrows M \indprotensor N \to M \otimes_{\Z} N,
\]
where one map is the addition on $M$ and the other map is given by
\[
 (M ^2) \indprotensor N \xrightarrow{\id^2 \indprotensor \Delta}  (M^2)\indprotensor  (N^2)  \xrightarrow[\text{Lemma~\ref{tensorjuggle}}]{\zeta} (M \indprotensor N)^2 \xrightarrow{+} M \indprotensor N
\]
Modelled by this, we use Lemma~\ref{coeqreflection} to define in the category $\fModSchf{k}{\Z}$:
\[
M \indprotensor_\Z \hat N  = \coeq( (M \times M) \indprotensor \hat N \rightrightarrows M \indprotensor \hat N)
\]
The claim then follows from Corollary~\ref{fmodschcolimadj}.
\end{proof}

\begin{corollary} \label{fmodschzinduction}
Let $k$ be Pr\"ufer and $l$ any graded commutative ring. Then the forgetful functor $U^l_\Z\colon \fModSchf{k}{l} \to \fModSchf{k}{\Z}$ has a left and a right adjoint.
\end{corollary}
\begin{proof}
For $\hat N \in \fModSchf{k}{\Z}$, the left adjoint is given by $l \indprotensor_{\Z} \hat N$. This follows from Lemma~\ref{fmodschztensoradj} and Lemma~\ref{fmodschztensor} along with the observation that $U^l_\Z = \hom_l(l,-)$. 

The right adjoint is given by $\hom_{\Z}(l,\hat N)$, which becomes a formal $l$-module scheme by the left action of $l$ on itself by multiplication.
\end{proof}

\begin{corollary}\label{fmodschstructure}
For $k$ Pr\"ufer, the forgetful functor $U^l\colon \fModSchf{k}{l} \to \Schf{k}$ has a left adjoint $\Fr_{l}$. For $\hat M \in \Schf{k}$, $\hat N \in \fModSch{k}{l}$, there is a natural adjunction isomorphism
\[
\mathfrak{F}_l(\Fr_{l} \circ \hat M,\hat N) \cong \fModSch{k}{l}(\Fr_{l}(\hat M),\hat N).
\]
The category $\fModSchf{k}{l}$ is bicomplete and equivalent to the category of algebras for the monad $\Fr_{l}$.
\end{corollary}
\begin{proof}
The forgetful functor factors as $\fModSchf{k}{l} \xrightarrow{U^l_\Z} \fModSchf{k}{\Z} \xrightarrow{U^\Z} \Schf{k}$ and both have left adjoints by Cor.~\ref{fmodschzinduction} and Thm.~\ref{cofreehopfalgebra}. The adjunction follows from Cor.~\ref{fmodschfreeadj} and Lemma~\ref{fmodschztensor}. An $l$-module scheme structure on a $\Z$-module scheme $\hat M \in \fModSchf{k}{\Z}$ is given by an action map $\hat M(R) \otimes_\Z l \to \hat M(R)$, which by Lemma~\ref{fmodschztensor} is the same as an action map $\hat M \indprotensor_\Z l \to \hat M$, which in turn is the same as an action of the monad $\Fr^\Z_l = - \otimes_\Z l$. Combined with Cor.~\ref{fmodschcomonadic}, this shows that $\fModSchf{k}{l}$ is equivalent to the category of $\Fr_l$-algebras.

By Lemma~\ref{fmodschlimcolim}\eqref{fmodschfilteredcolimits}, $U^l$ commutes with filtered colimits, and by Lemma~\ref{flatschcocomplete}, $\Schf{k}$ is bicomplete. Therefore \cite[Prop.~4.3.6]{borceux:voltwo}, $\fModSch{k}{l}$ is bicomplete.
\end{proof}

For $k$ Pr\"ufer, the functor of Lemma~\ref{fmodschztensoradj} extends to bimodules:
\begin{equation}\label{rightlfmodsch}
\hom_{l}\colon \Bimod{l}{l'} \times \fModSchf{k}{l} \to \fModSchf{k}{l'}.
\end{equation}
by the functoriality in the $l$-module variable. Similarly, there is a tensor-type functor
\begin{equation}\label{leftlfmodsch}
\indprotensor_{l'}\colon \Bimod{l'}{l} \times \fModSchf{k}{l'} \to \fModSchf{k}{l}
\end{equation}
given by the coequalizer in $\fModSchf{k}{l'}$
\[
(l' \otimes_\Z M) \indprotensor_{\Z} \hat N \rightrightarrows M \indprotensor_{\Z}  \hat N \to M \indprotensor_{l'} \hat N.
\]
\label{fmodschtensoring}
These are of particular interest if there is a ring map $\alpha\colon l \to l'$ and $l'$ is considered an $l'-l$-bimodule by means of right and left multiplication.
\begin{corollary} \label{fmodschinduction}
Let $k$ be Pr\"ufer and $l \to l'$ be a map of graded commutative rings. Then $l' \indprotensor_{l} -$ is left adjoint and $\hom_{l}(l',-)$ is right adjoint to the restriction functor $U_{l}^{l'}\colon \fModSchf{k}{l'} \to \fModSchf{k}{l}$. \qed
\end{corollary}

For $l=l'$, we obtain thus from \eqref{rightlfmodsch} and \eqref{leftlfmodsch}:
\begin{corollary}\label{modellenrichment}
For $k$ Pr\"ufer and any $l$, the category $\fModSchf{k}{l}$ is a $2$-bimodule over $\Mod_{l}$. \qed
\end{corollary}

\subsection{Tensor products of formal module schemes} \label{subsec:modschtensor}

Our next objective is to construct the tensor product $\hat M \otimes_{l} \hat N$ of two formal $l$-module schemes in $\fModSchf{k}{l}$. There are two equivalent characterizations of what this tensor product is supposed to accomplish.

\begin{defn} \label{def:fmodschtensor}
For $\hat M$, $\hat N \in \fModSch{k}{l}$, an object $\hat M \otimes_l \hat N$ together with an $l$-bilinear natural transformation $\hat M \times \hat N \to \hat M \otimes_{l} \hat N$ is called a \emph{tensor product of formal module schemes} if the following two equivalent conditions hold for all $\hat H \in \fModSch{k}{l}$:
\begin{enumerate}
\item There is an adjunction $\fModSch{k}{l}(\hat M \otimes_{l} \hat N,\hat H) \cong \mathfrak{F}_l(\hat M(-) \otimes_{l} \hat N(-),\hat H)$, where $\hat M(-)\otimes_{l} \hat N(-)$ is the objectwise tensor product of $l$-modules. \label{charonetensor}
\item Any $l$-bilinear natural transformation of formal module schemes $\hat M \times \hat N \to \hat H$ factors uniquely through a linear morphism $\hat M \otimes_{l} \hat N \to \hat H$. \label{chartwotensor}
\end{enumerate}
\end{defn}

This concept is a dual and $l$-module enhanced version of the tensor product of bicommutative Hopf algebras \cite{hunton-turner:coalgebraic,goerss:hopf-rings}. It is immediate from this definition that if a tensor product exists then it will be unique.

\begin{thm} \label{tensorsexist}
If $k$ is a Pr\"ufer domain then the tensor product of flat formal $l$-module schemes over $k$ exists.
\end{thm}
\begin{proof}
First consider the case $l=\Z$. Since $\fModSchf{k}{\Z}$ is the category of abelian group objects in the cocomplete category $\Schf{k}$ which has all products, and since the forgetful functor $U^\Z\colon \fModSchf{k}{\Z} \to \Schf{k}$ has a left adjoint by Thm.~\ref{cofreehopfalgebra}, \cite[Prop.~5.5]{goerss:hopf-rings} implies that $\otimes_\Z$ exists.

Explicitly, $\hat M \otimes_\Z \hat N$ is given as the simultaneous coequalizer of
\begin{equation}\label{ztensorcoeq}
\Fr_\Z(\hat M^2 \times \hat N) \rightrightarrows \Fr_\Z(\hat M \times \hat N) \leftleftarrows \Fr_\Z(\hat M \times \hat N^2),
\end{equation}
where the maps on the left are given by addition on $M^2$ and by the adjoint of
\[
\hat M^2 \times \hat N \xrightarrow{\id \times \Delta} \hat M^2 \times \hat N^2 \xrightarrow{\text{shuffle}} (\hat M \times \hat N)^2 \xrightarrow{\eta} \Fr(\hat M \times \hat N)^2 \xrightarrow{+}  \Fr(\hat M \times \hat N);
\]
similarly on the right hand side.

For arbitrary $l$, define $\hat M \otimes_l \hat N$ as the coequalizer
\begin{equation}\label{tensorcoeq}
l \indprotensor_{\Z} (\hat M \otimes_{\Z} \hat N) \rightrightarrows \hat M \otimes_{\Z} \hat N \to \hat M \otimes_{l} \hat N \in \fModSchf{k}{\Z}.
\end{equation}
The two maps are given by \eqref{alphamap} and the $l$-action on $\hat M$ and $\hat N$, respectively. The common composite of these maps factors through $l \indprotensor_\Z (\hat M \otimes_l \hat N)$ and therefore equips $\hat M \otimes_l \hat N$ with the structure of a formal $l$-module scheme. To see that it satisfies the universal property \eqref{charonetensor}, let $H \in \fModSch{k}{l}$ and first consider the equalizers
\begin{align*}
\fModSch{k}{l}(\hat M \otimes_l \hat N, \hat H) \to \fModSch{k}{\Z}(\hat M \otimes_l \hat N,\hat H) \rightrightarrows \hom_\Z(l,\fModSch{k}{\Z}(\hat M \otimes_l \hat N, \hat H))\\
\mathfrak F_l(\hat M(-) \otimes_l \hat N(-)) \to \mathfrak F_\Z(\hat M(-) \otimes_l \hat N(-),\hat H) \rightrightarrows \hom_\Z(l,\mathfrak F_\Z(\hat M(-) \otimes_l \hat N(-), \hat H)
\end{align*}
to conclude that it suffices to construct a natural isomorphism $\fModSch{k}{\Z}(\hat M \otimes_l \hat N,\hat H) \to  \mathfrak F_\Z(\hat M(-) \otimes_l \hat N(-),\hat H)$. For this, consider the two equalizers
\[
\begin{tikzpicture}
	\matrix (m) [matrix of math nodes, row sep=1.5em, column sep=.5em, text height=1.5ex, text depth=0.25ex]
	{
		\fModSch{k}{\Z}( \hat M \otimes_{l} \hat N,\hat H)  & \fModSch{k}{\Z} (\hat M \otimes_{\Z} \hat N,\hat H)& \fModSch{k}{\Z} (l \indprotensor_{\Z}(\hat M \otimes_{\Z} \hat N), \hat H)\\
		\mathfrak F(\hat M(-) \otimes_l \hat N(-),\hat H) & \mathfrak F(\hat M(-) \otimes_\Z \hat N(-),\hat H) & \mathfrak F(l \otimes_\Z \hat M(-) \otimes \hat N(-),\hat H)\\
	};
	\path[->,font=\scriptsize]
	(m-1-1)	edge (m-1-2)
			edge[dashed] (m-2-1)
	(m-2-1)	edge (m-2-2)
	([yshift=-1mm]m-1-2.east)		edge ([yshift=-1mm]m-1-3.west)
	([yshift=+1mm]m-1-2.east)	edge ([yshift=+1mm]m-1-3.west)
	([yshift=-1mm]m-2-2.east)		edge ([yshift=-1mm]m-2-3.west)
	([yshift=+1mm]m-2-2.east)	edge ([yshift=+1mm]m-2-3.west)
	(m-1-2)	edge (m-2-2)
	(m-1-3)	edge (m-2-3);
\end{tikzpicture}
\]
The two solid vertical arrows are isomorphisms by the $l=\Z$ case and thus the induced map on the left is as well.
\end{proof}

\begin{lemma} \label{tensoroffreemodsch}
Let $k$ be Pr\"ufer and $\hat X$, $\hat Y \in \Schf{k}$. Then
\[
\Fr_{l}(\hat X) \otimes_{l} \Fr_{l}(\hat Y) \cong \Fr_{l}(\hat X \times \hat Y).
\]
\end{lemma}
\begin{proof}
Immediate from the characterization \eqref{chartwotensor} of Def.~\ref{def:fmodschtensor}.
\end{proof}

We will denote the representing object of $\hat M \otimes_l \hat N$ by $\Reg{\hat M} \multtensor{k}{l} \Reg{\hat N} =_{\text{def}} \Reg{\hat M \otimes_l \hat N}$. Using Example~\ref{twoalgebraoverset}, we thus summarize:

\begin{lemma}\label{hopftwoalgoverindset}
If $k$ is Pr\"ufer then the symmetric monoidal category $(\fModSchf{k}{l},\otimes_{l},\tensunit{k}{l})$ is a $2$-algebra over $\Set^\Z$.  \qed
\end{lemma}

\begin{thm} \label{hopftwoalgebra}
For $k$ Pr\"ufer and any $l$, the symmetric monoidal category $(\fModSchf{k}{l},\otimes_{l},\tensunit{k}{l})$ is a $2$-algebra over $(\Mod_l,\otimes_{l},l)$. 
\end{thm}
\begin{proof}

We need to see that the map in $\Set$
\[
\fModSch{k}{l}(\hat M,\hat M') \xrightarrow{- \otimes_l \hat N} \fModSch{k}{l}(\hat M \otimes_l \hat N, \hat M' \otimes_l \hat N)
\]
is in fact a map in $\Mod_l$. To see this, note that it factors as a map
\begin{align*}
\fModSch{k}{l}(\hat M,\hat M') = & \mathfrak{F}_l(\hat M,\hat M')\\
\xrightarrow{\otimes_l \hat N(-)} & \mathfrak{F}_l(\hat M(-) \otimes_l \hat N(-), \hat M'(-) \otimes_l \hat N(-))\\
\to & \mathfrak{F}_l(\hat M(-) \otimes_l \hat N(-),\hat M' \otimes_l \hat N)\\
\xrightarrow[\cong]{\text{Def.~\ref{def:fmodschtensor}\eqref{charonetensor}}} & \fModSch{k}{l}(\hat M \otimes_l \hat N, \hat M' \otimes_l \hat N).
\end{align*}
All maps in this diagram are $l$-module maps.
\end{proof}

\section{The structure of the category of formal algebra schemes} \label{sec:falgschstructure}

The category of formal algebra schemes behaves differently from the category of module schemes or bimodules in that it does not have an enrichment over $\Mod_l$ or $\Alg_l$. It does not fit into our framework of $2$-algebras or $2$-modules. We are already lacking a $2$-ring structure on the category $\Alg_l$.

As an immediate consequence of the tensor product construction (Subsection~\ref{subsec:modschtensor}) we obtain:
\begin{corollary} \label{comultintotensor}
Let $k$ be a Pr\"ufer ring and $l$ any ring, and $\hat A \in \fAlgSchf{k}{l}$. Then the multiplication $\hat A(-) \times \hat A(-) \to \hat A(-)$ can be uniquely extended to a map of formal $l$-module schemes
\[
\mu_\times \colon \hat A \otimes_l \hat A \to \hat A.
\]
This map is in fact a map of formal $l$-algebra schemes (by the commutativity of the multiplication of $\hat A$). \qed
\end{corollary}

The $2$-module structure on $\fModSchf{k}{l}$ over $\Mod_{l}^f$ does not descend to a $2$-module structure on $\fAlgSchf{k}{l}$. The only remnant of it is a cotensor-like functor which pairs an algebra and an algebra scheme and gives a formal scheme. For this, let $\hat A \in \fAlgSchf{k}{l}$ and let $R \in \Alg_{l}$. Recall from Lemma~\ref{fmodschztensoradj} the construction of a right $2$-module structure $\hom_l(M,\hat N)$ for an $l$-module $M$ and $\hat N \in \fModSchf{k}{l}$. For an $l$-flat ring $R$, define $\hom_l^\times(R,\hat A)$ to be the functor which sends a $k$-algebra $T$ to the set of algebra maps $R \to \hat A(T)$. We can extend this definition for non-flat $R$, by defining $\hom_l^\times(R,\hat A)$
as the simultaneous equalizer in $\Schf{k}$ of
\begin{equation} \label{algtensoradd}
\hom_{l}(R,\hat A) \rightrightarrows \hom_{l}(R\times R, \hat A)
\end{equation}
and
\begin{equation} \label{algtensorlin}
\hom_{l}(R,\hat A) \rightrightarrows \hat A,
\end{equation}
where the maps in \eqref{algtensoradd} are given by the multiplication on $R$ and
\[
\hom_{l}(R,\hat A) \xrightarrow{\Delta} \hom_{l}(R,\hat A)^2 \xrightarrow[\text{Lemma~\ref{tensorjuggle}}]{\zeta'} \hom_{l}(R^2,\hat A^2) \xrightarrow{\hom_{l}(R^2,\cdot)} \hom_{l}(R^2,\hat A),
\]
respectively, and the maps in \eqref{algtensorlin} are given by evaluation at $1 \in R$ and the constant map with value $1 \in \hat A(T)$ for all $T$.

\section{Formal coalgebras and formal plethories} \label{sec:fringsfplethories}

We have seen in the previous sections that the categories $\fBimod{k}{l}$, $\fModSchf{k}{l}$, and $\fAlgSchf{k}{l}$ are symmetric monoidal categories with respect to the tensor product of formal bimodules resp. formal module schemes. On $\fAlgSchf{k}{l}$, this tensor product is actually the categorical coproduct; on $\fModSchf{k}{l}$, it is not. Furthermore, we have various $2$-bimodule structures over $\Mod_l$. The aim of this section is to construct a second monoidal structure $\circ$ on the categories $\fBimod{k}{k}$ and $\fAlgSchf{k}{k}$ which corresponds to composition of ind-representable functors. These monoidal structures have an interesting compatibility with the tensor product monoidal structure.

\subsection[2-monoidal categories]{$2$-monoidal categories}

The situation of a category with two compatible monoidal structures has been studied, although not with our examples in mind \cite{aguiar-mahajan:monoidal-functors,vallette:manin-products,joyal-street:braided-tensor}, and they are known as \emph{$2$-monoidal categories}. As always with higher categorical concepts, there is much leeway in the definitions as to what level of strictness one wants to require; Aguiar and Mahajan's definition of a $2$-monoidal category \cite[Section 6]{aguiar-mahajan:monoidal-functors} is the laxest in the literature and fits our application, although in our case more strictness assumptions could be made.

\begin{defn}[Aguiar-Mahajan] \label{def:twomonoidalcat}
A \emph{$2$-monoidal category} is a category $\C$ with two monoidal structures $(\otimes,I)$ and $(\circ,J)$ with natural transformations
\[
\zeta\colon (A \circ B) \otimes (C \circ D) \to (A \otimes C) \circ (B \otimes D)
\]
and
\[
\Delta_I\colon I \to I \circ I,\quad \mu_J\colon J \otimes J \to J, \quad \iota_J = \epsilon_I\colon I \to J,
\]
such that:
\begin{enumerate}
\item the functor $\circ$ is a lax monoidal functor with respect to $\otimes$, the structure maps being given by $\zeta$ and $\mu_J$;
\item the functor $\otimes$ is an oplax monoidal functor with respect to $\circ$, the structure maps being given by $\zeta$ and $\Delta_I$;
\item $(J,\mu_J,\iota_J)$ is a $\otimes$-monoid;
\item $(I,\Delta_I,\epsilon_I)$ is a $\circ$-comonoid.
\end{enumerate}
\end{defn}

A $2$-monoidal category is the most general categorical setup where a bialgebra can be defined, although in this context it is more common to call it a bimonoid.

\begin{defn}[Aguiar-Mahajan]
A \emph{bimonoid} in a $2$-monoidal category $\C$ as above is an object $H$ with a structure of a monoid in $(\C,\otimes,I)$ and a structure of a comonoid in $(\C,\circ,J)$ satisfying the compatibility condition that the monoid structure maps are comonoid maps and the comonoid structure maps are monoid maps.
\end{defn}
To make sense of the compatibility condition, notice that if $H$ is an $\otimes$-monoid with multiplication $\mu$ and unit $\iota$, then so is $H \circ H$ by virtue of the maps
\[
(H \circ H) \otimes (H \circ H) \xrightarrow{\zeta} (H \otimes H) \circ (H \otimes H) \xrightarrow{\mu \circ \mu} H \circ H
\]
and
\[
I \xrightarrow{\Delta_I} I \circ I \xrightarrow{\iota \circ \iota} H \circ H;
\]
similarly if $H$ is a $\circ$-comonoid with comultiplication $\Delta$ and counit $\epsilon$ then so is $H \otimes H$ by virtue of the maps
\[
H \otimes H \xrightarrow{\Delta \otimes \Delta} (H \circ H) \otimes (H \circ H) \xrightarrow{\zeta} (H \otimes H) \circ (H \otimes H)
\]
and
\[
H \otimes H \xrightarrow{\epsilon \otimes \epsilon} J \otimes J \xrightarrow{\mu_J} J.
\]

Finally, we define a suitable notion of a functor between $2$-monoidal categories in order to map bimonoids to bimonoids:
\begin{defn}[Aguiar-Mahajan]
A \emph{bilax monoidal functor} $F\colon \C \to \D$ between $2$-monoidal categories is a functor which is lax monoidal with respect to $\otimes$ and oplax monoidal with respect to $\circ$ and whose lax structure $\phi$ and oplax structure $\psi$ (Def.~\ref{def:laxmorphism}) are compatible in the sense that the following diagrams commute:
\begin{equation}\label{biunitality}
\begin{matrix}
\begin{tikzpicture}
	\matrix (m) [matrix of math nodes, row sep=2em, column sep=2em, text height=1.5ex, text depth=0.25ex]
	{
		I_\D &		J_\D \\
		F(I_\C) & 	F(J_\C)\\
	};
	\path[->,font=\scriptsize]
	(m-1-1)		edge node[auto]{$\iota_J$}				(m-1-2)
				edge node[auto]{$\phi_0$}				(m-2-1)
	(m-2-1)		edge node[auto]{$F(\iota_J)$}				(m-2-2)
	(m-2-2)		edge node[auto]{$\psi_0$}					(m-1-2);
\end{tikzpicture}
\end{matrix}\end{equation}
\begin{equation}\label{leftunitality}\begin{matrix}
\begin{tikzpicture}
	\matrix (m) [matrix of math nodes, row sep=2em, column sep=1.2em, text height=1.5ex, text depth=0.25ex]
	{
		I_\D &			F(I_\C) & 	F(I_\C \circ I_\C) \\
		I_\D \circ I_\D &		&	F(I_\C) \circ F(I_\C)  \\
	};
	\path[->,font=\scriptsize]
	(m-1-1)		edge node[auto]{$\phi_0$}				(m-1-2)
				edge node[auto]{$\Delta_I$}				(m-2-1)
	(m-1-2)		edge node[auto]{$F(\Delta_I)$}			(m-1-3)
	(m-1-3)		edge node[auto]{$\psi$}					(m-2-3)
	(m-2-1)		edge node[auto]{$\phi_0 \circ \phi_0$}	(m-2-3);
\end{tikzpicture}
\end{matrix}\end{equation}
\begin{equation}\label{rightunitality}\begin{matrix}
\begin{tikzpicture}
	\matrix (m) [matrix of math nodes, row sep=2em, column sep=1.2em, text height=1.5ex, text depth=0.25ex]
	{
		F(J_\C) \otimes F(J_\C) &		&	J_\D \otimes J_\D\\
		F(J_\C \otimes J_\C) &		F(J_\C) &	J_\D \\
	};
	\path[->,font=\scriptsize]
	(m-1-1)		edge node[auto]{$\psi_0 \otimes \psi_0$}	(m-1-3)
				edge node[auto]{$\phi$}					(m-2-1)
	(m-1-3)		edge node[auto]{$\mu_J$}					(m-2-3)
	(m-2-1)		edge node[auto]{$F(\mu_J)$}				(m-2-2)
	(m-2-2)		edge node[auto]{$\psi_0$}					(m-2-3);
\end{tikzpicture}
\end{matrix}
\end{equation}
\begin{equation}\label{bilaxcompat}
\begin{matrix}
\begin{tikzpicture}[xscale=3.0, yscale = 1.7]
	\node(m-1-2) at (0,1) {$(FA \circ FB) \otimes (FC \circ FD)$};
	\node(m-2-1) at (-.866,.33) {$F(A \circ B) \otimes F(C \circ D) $};
	\node(m-3-1) at (-.866,-.33) {$F((A \circ B) \otimes (C \circ D))$};
	\node(m-4-2) at (0,-1) {$F((A \otimes C) \circ (B \otimes D))$};
	\node(m-3-3) at (.866,-.33) {$F(A \otimes C) \circ F(B \otimes D)$};
	\node(m-2-3) at (.866,.33) {$(FA \otimes FC) \circ (FB \otimes FD)$};

	\path[->,font=\scriptsize]
	(m-2-1)		edge node[auto]{$\psi \otimes \psi$}		(m-1-2)
				edge node[auto]{$\phi$}					(m-3-1)
	(m-1-2)		edge node[auto]{$\zeta$}					(m-2-3)
	(m-2-3)		edge node[auto]{$\phi \circ \phi$}			(m-3-3)
	(m-3-1)		edge node[auto]{$F(\zeta)$}				(m-4-2)
	(m-4-2)		edge node[auto]{$\psi$}					(m-3-3);
\end{tikzpicture}
\end{matrix}
\end{equation}

\end{defn}

\begin{prop}[{\cite[Cor.~6.53]{aguiar-mahajan:monoidal-functors}}] \label{bilaxfunctorspreservebimonoids}
Bilax monoidal functors preserve bimonoids and morphisms between them.\qed
\end{prop}

\begin{prop}\label{bilaxcondition}
Let $F\colon \C \to \D$, $G \colon \D \to \C$ be functors between $2$-monoidal categories. Assume
\begin{enumerate}
\item $F$ is left adjoint to $G$;
\item $F$ is strictly monoidal with respect to $(\otimes_\C,I_\C)$ and $(\otimes_\D,I_\D)$; \label{fstrictmon}
\item $G$ is strictly monoidal with respect to $(\circ_\D,J_\D)$ and $(\circ_\C,J_\C)$. \label{gstrictmon}
\end{enumerate}
Then $F$ is bilax if and only if $G$ is.
\end{prop}
\begin{proof}
Let $\eta\colon \id _\C\to G \circ F$ denote the unit and $\epsilon\colon F \circ G \to \id_\D$ the counit of the adjunction.

Given \eqref{gstrictmon}, $F$ is oplax monoidal with respect to $\circ$ by the adjoints $F(J_\C) \to J_\D$ and $F(A \circ_\C B) \to F(A) \circ_\D F(B)$ of $\psi_0^{-1}\colon J_\C \to G(J_\D)$ and
\[
A \circ B \xrightarrow{\eta \circ \eta} G(F(A)) \circ G(F(B)) \xrightarrow{\psi^{-1}} G(F(A) \circ F(B)),
\]
respectively. Similarly $G$ is lax monoidal with respect to $\otimes$ using \eqref{fstrictmon}.

Under the given conditions, the following composites are identities:
\begin{align*}
I \xrightarrow[\cong]{(\phi_F)_0} F(I) \xrightarrow{F((\phi_G)_0)} F(G(I)) \xrightarrow{\epsilon} I \\
J \xrightarrow{\eta} G(F(J)) \xrightarrow{G((\psi_F)_0)} G(J) \xrightarrow[\cong]{(\psi_G)_0} J.
\end{align*}
If $G$ is bilax then \eqref{biunitality} commutes for $F$ because it can be factored as:
\[
\begin{tikzpicture}
	\matrix (m) [matrix of math nodes, row sep=2em, column sep=6em, text height=1.5ex, text depth=0.25ex]
	{
		I_\D & 		I_\D & 			J_\D\\
		F(I_\C) &	F(G(I_\D))\\
		F(J_\C) &	F(G(J_\D))\\
	};
	\path[->,font=\scriptsize]
	(m-1-1)		edge[-,double,double equal sign distance]	(m-1-2)
				edge node[auto]{$(\phi_F)_0$}			(m-2-1)
	(m-1-2)		edge node[auto]{$\iota_J$}				(m-1-3)
	(m-2-1)		edge node[auto]{$F((\phi_G)_0)$}			(m-2-2)
				edge node[auto]{$F(\iota_J)$}				(m-3-1)
	(m-2-2)		edge node[auto]{$\epsilon$}				(m-1-2)
				edge node[auto,swap]{$F(G(\iota_J))$}		(m-3-2)
	(m-3-1)		edge node[auto]{$F((\psi_G)_0^{-1})$}		(m-3-2)
	(m-3-2)		edge node[auto]{$\epsilon$}				(m-1-3);
\end{tikzpicture}
\]
The lower rectangle commutes because $G$ was assumed to be bilax monoidal. The other diagrams follow by similar exercises in adjunctions.
\end{proof}

\subsection[The 2-monoidal categories of formal algebra schemes and formal bimodules]{The $2$-monoidal categories of formal algebra schemes and formal bimodules}

Throughout the rest of this section, we assume that $k$ is a graded Pr\"ufer domain and $l$ is an arbitrary graded commutative ring. We will now supply the $2$-algebras $\fBimod{k}{k}$ and $\fAlgSchf{k}{k}$ with a new monoidal structures $\circ$, making them into $2$-monoidal categories.

\begin{defn} \label{def:circprod}
Consider the following setups:
\begin{enumerate}
\item $A \in \Pro-\Mod_{l}$ and $\hat Y \in \fBimod{k}{l}$.
\item $A \in \Pro-\Mod_{l}$ and $\hat Y \in \fBimodf{k}{l}$.
\item $A \in \Pro-\Mod_{l}$ and $\hat Y \in \fModSchf{k}{l}$.
\item $A \in \Pro-\Alg_{l}$ and $\hat Y \in \fAlgSchf{k}{l}$.
\end{enumerate}
In each of these cases, define a functor $\circ_{l}$ by 
\[
\Spf{A} \circ_{l} \hat Y =\colim_i \hom(A(i),\hat Y),
\]
where the colimit is taken in $\fBimod{k}{\Z}$, $\fBimodf{k}{\Z}$, and $\Schf{k}$ in the last two cases, respectively. Here $\hom$ denotes $\hom_l$ in the first three cases (Lemma~\ref{fbimodtwomodule} and \eqref{rightlfmodsch}) and $\hom_l^\times$ (Section~\ref{sec:falgschstructure}) in the last case.
\end{defn}

As a consequence of the various $2$-module structures exhibited in Sections~\ref{sec:fbimodstructure}, \ref{sec:fmodschstructure}, and~\ref{sec:falgschstructure}, we obtain:

\begin{corollary}\label{tensoradjunctions}
We have natural isomorphisms $(\hat M \circ_l \hat N)(T) \cong \hat M(\hat N(T))$ where:
\begin{enumerate}
\item \label{bimodtensoradjunction}
$\hat M \in \fBimod{l}{\Z}$, $\hat N \in \fBimod{k}{l}$, and $T \in \Mod_k$;
\item \label{flatbimodtensoradjunction}
$\hat M \in \fBimod{l}{\Z}$, $\hat N \in \fBimodf{k}{l}$, and $T \in \Mod_k^f$;
\item  \label{modschtensoradjunction}
$\hat M \in \fBimod{l}{\Z}$, $\hat N \in \fModSchf{k}{l}$, and $T \in \Alg_k^f$;
\item \label{algschtensoradjunction}
$\hat M \in \Pro-\Alg_{l}$, $\hat N \in \fAlgSchf{k}{l}$, and $T \in \Alg_k^f$.
\end{enumerate}
\end{corollary}
\begin{proof}
For the proof of \eqref{bimodtensoradjunction}, 
\begin{multline*}
(\hat M \circ_l \hat N)(T) = \colim_i \hom_l(\Reg{\hat M}(i),\hat N)(T)\\
\underset{\eqref{fbimodrightmodrep}}= \colim_i \Mod_l(\Reg{\hat M}(i),\hat N(T)) = \hat M (\hat N(T)).
\end{multline*}
For \eqref{flatbimodtensoradjunction}--\eqref{algschtensoradjunction}, the proofs are formally the same.
\end{proof}

If $k$, $l$, $m$ are rings (Pr\"ufer where required) then by Cor.~\ref{tensoradjunctions}, additional structure on $M$ gives rise to additional structure in the target: the product $\circ_l$ extends to products
\begin{eqnarray}
\circ_{l}\colon & \fBimod{l}{m} \times \fBimod{k}{l} &\to \fBimod{k}{m} \label{bimodcirc}\\
\circ_{l}\colon & \fBimodf{l}{m} \times \fBimodf{k}{l} &\to \fBimodf{k}{m} \label{bimodfcirc}\\
\circ_{l}\colon & \fBimod{l}{m} \times \fModSchf{k}{l} &\to \fModSchf{k}{m} \label{modhopfcirc}\\
\circ_{l}\colon & \fModSchf{l}{m} \times \fAlgSchf{k}{l} &\to \fModSchf{k}{m} \label{modalgcirc}\\
\circ_{l}\colon & \fAlgSchf{l}{m} \times \fAlgSchf{k}{l} &\to \fAlgSchf{k}{m} \label{algalgcirc}.
\end{eqnarray}
All of these products represent compositions of functors. Therefore all of these products are associative in the appropriate sense, i.e. \eqref{bimodcirc}, \eqref{bimodfcirc}, \eqref{algalgcirc} are associative, \eqref{modhopfcirc} is a tensoring over the \eqref{bimodcirc}, and \eqref{modalgcirc} is a tensoring over \eqref{algalgcirc}.

When $k=l=m$, \eqref{bimodcirc}, \eqref{bimodfcirc}, and \eqref{algalgcirc} define monoidal structures on $\fBimod{k}{k}$, $\fBimodf{k}{k}$, and $\fAlgSchf{k}{k}$, respectively, but they are neither symmetric nor closed. The identity functors $J_k \in \fBimodf{k}{k}$ (Ex.~\ref{exa:circunitmod}) and $J_k \in \fAlgSchf{k}{k}$ (Ex.~\ref{exa:circunit}) are units for $\circ_k$. We summarize:

\begin{lemma} \label{circmonoidal}
The category $(\fBimod{k}{k},\circ_k,J_k)$ and its subcategory $\fBimodf{k}{k}$ are $2$-algebra over $\Mod_k$. The category  $(\fAlgSchf{k}{k},\circ_k,J_k)$ is a $2$-algebra over $\Set^\Z$. Finally, the category $\fModSchf{k}{k}$ is tensored over $\fBimod{k}{k}$ and $\fAlgSchf{k}{k}$. \qed
\end{lemma}

We can finally define the object of the title of this paper.
\begin{defn}\label{def:formalplethory}
A \emph{formal $k$-coalgebra} is a comonoid in $(\fBimod{k}{k},\circ_k,J_k)$. A \emph{formal plethory} is a comonoid in $(\fAlgSchf{k}{k}, \circ_k,J_k)$.
\end{defn}

\begin{prop}
The category $(\fAlgSchf{k}{k},\otimes_k,\circ_k)$ is a $2$-monoidal category, and the forgetful functor from bimonoids in $\fAlgSchf{k}{k}$ to formal plethories is an equivalence of categories.
\end{prop}
\begin{proof}
This is straightforward because $\otimes_k$ is the categorical coproduct, and any monoidal category is automatically $2$-monoidal with respect to the categorical coproduct \cite[Example~6.19]{aguiar-mahajan:monoidal-functors}, cf. Example~\ref{exa:coprodtwomonoidal}.
Moreover, a monoid structure with respect to the categorical coproduct always exists uniquely (cf.~\cite[Example~6.42]{aguiar-mahajan:monoidal-functors}).
\end{proof}

The rest of this section is devoted to proving the following theorem:

\begin{thm}
The category $(\fBimod{k}{k},\otimes_k,\circ_k)$ and its subcategory $\fBimodf{k}{k}$ are $2$-mo\-noi\-dal.
\end{thm}

In Section~\ref{section:primindec}, we will construct bilax monoidal functors between $\fAlgSchf{k}{k}$, $\fBimod{k}{k}$, and $\fBimodf{k}{k}$.

\begin{proof}
To define $\zeta$, consider more generally the functor $\circ_l\colon \fBimod{l}{m} \times \fBimod{k}{l} \to \fBimod{k}{m}$ and first assume that $m=\Z$. Let $\hat M_1,\; \hat N_1 \in \fBimod{l}{\Z}$ and $\hat M_2,\;\hat N_2 \in \fBimod{k}{l}$. Let $\Reg{\hat M_1}\colon I \to \Mod_l$, $\Reg{\hat N_1}\colon J \to \Mod_l$ be representations of the corresponding pro-$l$-modules. Then we have a map
\begin{align*}
\zeta\colon (\hat M_1 \circ_{l} \hat M_2) \otimes_\Z (\hat N_1 \circ_l \hat N_2) & = \colim_{i\in I, j\in J} \hom_l(\Reg{\hat M_1}(i),\hat M_2) \otimes_\Z \hom(\Reg{\hat N_1}(j),\hat N_2)\\
\xrightarrow[\text{Lemmas~\ref{bimodtwoalgebra}, \ref{tensorjuggle}}]{\zeta'} & \colim_{i,j} \hom(\Reg{\hat M_1}(i) \otimes_l \Reg{\hat N_1}(j), \hat M_2 \otimes_l \hat N_2)\\
= & (\hat M_1 \otimes_{\Z} \hat N_1) \circ_l (\hat M_2 \otimes_{l} \hat N_2).
\end{align*}

The general case follows from writing $\hat M \otimes_l \hat N$ as the coequalizer \eqref{sweedlercoequalizer}.

The remaining structure maps are easier:
\begin{itemize}
\item $\iota_J = \epsilon_I \colon \tensunit{k}{k} = k \indprotensor_\Z \Spf{k} \to \Spf{k}$ is given by the bimodule structure map of Cor.~\ref{unitmapadjmodules}.
\item $\mu_J\colon J_k \otimes_k J_k \to J_k$: this is the isomorphism given by \eqref{sweedlercoequalizer}.
\item $\Delta_I\colon \tensunit{k}{k} \to \tensunit k k \circ \tensunit k k$: this map is defined by 
\begin{multline*}
k \indprotensor_\Z \Spf{k} \to (k \indprotensor_\Z (k \indprotensor_\Z \Spf{k} \cong 
(k \otimes k) \indprotensor_\Z (\Spf{k} \circ_k \Spf{k})\\
\xrightarrow[\text{Lemma~\ref{tensorjuggle}}]{\zeta'} (k \indprotensor_\Z \Spf{k}) \circ_k (k \indprotensor_\Z \Spf{k}).
\end{multline*}
The first map is induced by compositions of $\Z$-linear morphisms $k \to k$.
\end{itemize}
That $\zeta$ is compatible with the associativity isomorphisms of $\circ$ and $\otimes_k$ follows easily from Lemma~\ref{tensorjuggle} and the definition. We will verify the unitality conditions required in a $2$-monoidal category:
\begin{enumerate}
\item \textbf{$J_k$ is a monoid with respect to $\otimes$:} Since $\mu_J$ is an isomorphism, associativity follows from that of $k$. For unitality, observe that
\[
J_k \cong \tensunit{k}{k} \otimes_k J_k \xrightarrow{\iota_J \otimes_k \id} J_k \otimes_k J_k \cong J_k
\]
is the identity map as well.
\item \textbf{$I$ is a comonoid with respect to $\circ$:} The associativity is immediate from the definition of $\Delta_I$ and the counitality follows from the fact that composition with the identity map $k \to k$ is the identity.
\end{enumerate}
\end{proof}

We can now give a proof of Theorem~\ref{thm:topplethory}:

\begin{thm}\label{thm:topplethorypf}
Let $E$ be a cohomology theory such that $E_*$ is a Pr\"ufer ring and $\hat E^*\underline E_n$ is pro-flat for all $n$. Then $\hat P_E = \Spf{E^*\underline E_*}$ is a formal plethory, and for any space $X$, $\hat C_E = \Spf{\hat E^*(X)}$ is a comodule over it in the sense that there is a coassociative and counital action $\hat C_E \to \hat C_E \circ_{E_*} \hat P_E$.
\end{thm}
\begin{proof}
By Cor.~\ref{cor:cohomologyalgsch}, $\hat P_E$ is a formal $E_*$-algebra scheme over $E_*$. If $X$ is a finite CW-complex then there is a natural map
\[
\hat E^*(X) = [X,E] \xrightarrow{\hat E^*} \Pro-\Alg_{E_*}(\hat E^*(E),\hat E^*(X)),
\]
By Cor.~\ref{tensoradjunctions}\eqref{algschtensoradjunction}, this gives rise to a morphism in $\Sch{E_*}$
\[
\Spec E^*(X) \to \Spec E^*(X) \circ_{E_*} \hat P_E,
\]
and passing to colimits gives the desired map $\hat C_E \to \hat C_E \circ_{E_*} \hat P_E$. Applying this to $X=\underline E_n$ yields the plethory structure on $\hat P_E$. The coassociativity and counitality conditions are immediate.
\end{proof}

\section{Primitives and indecomposables}\label{section:primindec}

For any Hopf algebra $A$ over $k$, there are two particularly important $k$-modules: the module of primitive elements $PA = \{ a \in A \mid \psi(a) = a \otimes 1 + 1 \otimes a\}$ and the module of indecomposable elements $QA = A_+ / (A_+)^2$, where $A_+$ denotes the augmentation ideal with respect to the counit $\epsilon\colon A \to k$. The aim of this long section is to generalize these to linearization functors $\Primnoarg$ and $\Indecnoarg$ for formal plethories. I will show that they are bilax monoidal functors from the category of formal algebra schemes to the category of formal bimodules (Cor.~\ref{Pbilaxmonoidal} and \ref{Qbilaxmonoidal}; Thm.~\ref{thm:primindecbimodulefunctors}) and thus map formal plethories to bimonoids. The proof of this proceeds in three steps, each having a subsection devoted to it: the algebraic structure of the functors $\Primnoarg$ and $\Indecnoarg$ on formal module schemes (\ref{subsec:prim} and \ref{subsec:indec}), the additional structure if these functors are applied to formal algebra schemes (\ref{subsec:pqfalgsch}), and finally, the compatibility with the composition monoidal structure (\ref{subsec:pqfpleth}).

Throughout this section, we will assume that $k$ is a Pr\"ufer ring and $l$ is an arbitrary ring.

\subsection{Indecomposables} \label{subsec:indec}

Let $\Alg_k^+$ be the category of augmented graded commutative unital $k$-algebras, i.e. algebras $A$ with an algebra morphism $A \xrightarrow{\epsilon} k$ such that the composite with the algebra unit $k \xrightarrow{\eta} A \xrightarrow{\epsilon} k$ is the identity. Passing to pro-categories, the inclusion
\[
\Pro-(\Alg_k^+) \to (\Pro-\Alg_k)^+
\]
is an equivalence, so we will omit the parentheses. We denote its opposite category, the category of pointed formal schemes, by $\Sch{k}^+$. Note that a formal module scheme is automatically pointed.

For $M \in \Pro-\Mod_k$, let $k\oplus M \in \Pro-\Alg_k^+$ be the square-zero extension, i.e. the algebra where $\eta\colon k \hookrightarrow k\oplus M$ is a unit map and the multiplication is given by decreeing that $M^2=0$. The augmentation is given by the projection $\epsilon\colon k \oplus M \to k$.
\begin{defn} \label{def:indec}
The left adjoint of $k \oplus -\colon \Pro-\Mod_k \to \Pro-\Alg_k^+$ is denoted by $A \mapsto Q(A)$, the module of indecomposables of the augmented pro-algebra $A$.
\end{defn}
If $A_+ = \ker(\epsilon\colon A \to k)$ is the augmentation ideal, then $QA \cong A_+/(A_+)^2$, in particular, the left adjoint exists and is defined levelwise. If $\hat M \in \fModSch{k}{l}$ is a formal module scheme then it induces a new functor
\[
\hat M_+ = \Pro-\Alg_k^+(\Reg{\hat M},-)\colon \Alg_k^+ \to \Mod_l.
\]
The left hand side of
\[
\hat M_+(k \oplus M) = \Pro-\Alg_k^+(\Reg{\hat M},k\oplus M) \cong \Pro-\Mod_k(Q(\Reg{\hat M}),M)
\]
is thus in $\Mod_l$, thus $Q(\Reg{\hat M})$ actually represents an object of $\fBimod{k}{l}$ and $Q$ yields a functor
\begin{equation}\label{fmodschQ}
\overline\Indecnoarg = \Spf{Q}\colon \fModSch{k}{l} \to \fBimod{k}{l}.
\end{equation}
Even if $\hat M \in \fModSchf{k}{l}$ is flat, $\Spf{Q}$ might not be. However, also the square-$0$-extension morphism $\fBimodf{k}{\Z} \to (\Schf{k})^+$ has a right adjoint, which extends to
\begin{equation}\label{fmodschfQ}
\Indecnoarg = \Spf{Q}\colon \fModSchf{k}{l} \to \fBimodf{k}{l},
\end{equation}
given by composing \eqref{fmodschQ} with the left adjoint of $\fBimodf{k}{l} \to \fBimod{k}{l}$ (dividing out by torsion).

The functor $\Indecnoarg$ is also a right adjoint, as we will see now. Recall from Cor.~\ref{fmodschstructure} that the forgetful functor $U^l\colon \fModSchf{k}{l} \to \Schf{k}$ has a left adjoint $\Fr_l$. There is also a pointed version, i.e. a left adjoint $\Fr_l^+ = \Spf{\Gamma_l^+}$ of $U^l_+\colon \fModSchf{k}{l} \to (\Schf{k})^+$, given by $\Fr_l^+(\hat M) = \Spf{\Gamma_{l}((\Reg{\hat M})_+)}$. In particular, $\Gamma_{l}^+(k \oplus M) = \Gamma_{l}(M)$ as $k$-coalgebras. 

For $\hat M=\Spf{M} \in \fBimodf{k}{l}$ with associated square-zero formal algebra scheme $\hat M' = \Spf{k \oplus \Reg{M}} \in \Schf{k}$ we obtain a ring map
\[
l \xrightarrow{\mu} \fBimod{k}{\Z}(\hat M,\hat M) \xrightarrow{k\oplus -} \Sch{k}^+(\hat M',\hat M') \xrightarrow{\Fr_\Z^+} \fModSch{k}{\Z}(\Fr_\Z(\hat M),\Fr_\Z(\hat M))
\]
In this way, $\Fr(\hat M) = \Fr_\Z(\hat M)$ obtains an $l$-module scheme structure. We obtain an adjunction
\[
\fModSch{k}{l}(\Fr(\hat M),\hat N) \cong \fBimod{k}{l}(\hat M,\Indec{\hat N}) \quad (\hat M \in \fBimodf{k}{l},\; \hat N \in \fModSchf{k}{l})
\]

\begin{lemma} \label{freelcharacterization}
For $\hat M \in \fBimodf{k}{\Z}$, we have $\Fr_{l}(\hat M) \cong \Fr(l \indprotensor_{\Z} \hat M)$ (cf. Lemma~\ref{fbimodtwomodule}).
\end{lemma}
\begin{proof}
Let $\hat N \in \fModSch{k}{l}$. Then we have
\begin{align*}
\fModSch{k}{l}(\Fr_{l}(\hat M),\hat N) \cong & \fBimod{k}{\Z}(\hat M,\Indec{\hat N})\\
 \cong & {\fBimod{k}{l}}(l \indprotensor_\Z \hat M,\Indec{\hat N}) \cong \fModSch{k}{l}(\Fr(l \indprotensor_\Z \hat M),\hat N).
\end{align*}
\end{proof}

\begin{lemma} \label{qgammatwomodmorph}
The functor $\Indecnoarg\colon \fModSchf{k}{l} \to \fBimodf{k}{l}$ is a right strict morphism of $2$-bimodules over $\Mod_{l}$. Its adjoint $\Fr\colon \fBimodf{k}{l} \to \fModSchf{k}{l}$ is a left strict morphism of $2$-bimodules over $\Mod_{l}$.

More explicitly, for $M \in \Mod_{l}$, $\hat M \in \fBimodf{k}{l}$, and $\hat N \in \fModSch{k}{l}$, the maps
\begin{align*}
\beta\colon \Indec{\hom_l(M,\hat N)} \to & \hom_l(M,\Indec{\hat N})  & \text{(cf. \eqref{betamap})}\\
\intertext{and}
\alpha\colon M \indprotensor_l \Fr(\hat M) \to & \Fr(M \indprotensor_l \hat M) & \text{(cf. \eqref{alphamap})}
\end{align*}
are isomorphisms.
\end{lemma}
\begin{proof} Let $\hat N \in \fBimodf{k}{l}$ be a test object. Then
\begin{multline*}
\fBimod{k}{l}(\hat N,\Indec{\hom_l(M,\hat M)}) \cong  {\fModSch{k}{l}}(\Fr(\hat N),\hom_l(M,\hat M)) \\
\cong  \Mod_{l}(M,\fModSch{k}{l}(\Fr(\hat N),\hat M))
\cong \Mod_{l}(M,\fBimod{k}{l}(\hat N,\Indec{\hat M})) \\
\cong \fBimod{k}{l}(\hat N,\hom(M,\Indec{\hat M})).
\end{multline*}
The statement about $\Fr$ follows by adjointness.
\end{proof}

The functor $\Indecnoarg$ is not strict with respect to the tensor product $\otimes_l$ of formal module schemes (Thm.~\ref{hopftwoalgebra}). This can already be seen by observing that for the unit $\tensunit{k}{l} = \Spf{\hom(l,k)}$ of $\otimes_l$,
\begin{equation}\label{QIisnull}
Q(\hom(l,k)) = (\prod_{\lambda \in l-\{0\}} k) / (\prod_{\lambda \in l-\{0\}} k)^2 = 0,
\end{equation}
which is different from the unit $\tensunit{k}{l} = \Spf{\hom_\Z(l,k)}$ of $\fBimod{k}{l}$.

\begin{lemma} \label{gammamonoidal}
The functor $\Fr\colon \fBimod{k}{l} \to \fModSch{k}{l}$ is a strict morphism of $2$-algebras, i.e. there are natural isomorphisms
\[
\phi_0\colon \Fr(\tensunit{k}{l}) \to \tensunit{k}{l} \quad \text{and} \quad \phi\colon \Fr(\hat M \otimes_l \hat N) \to \Fr(\hat M) \otimes_l \Fr(\hat N)
\]
making $\Fr$ into a strict monoidal functor.
\end{lemma}
\begin{proof}
This follows from Cor.~\ref{constantzfmod}, Lemma~\ref{freelcharacterization}, and Lemma~\ref{tensoroffreemodsch}.
\end{proof}

\begin{corollary} \label{Qopmonoidal}
The functor $\Indecnoarg\colon \fModSch{k}{l} \to \fBimod{k}{l}$ is a lax morphism of $2$-algebras.
\end{corollary}
\begin{proof}
We need to produce natural transformations
\[
\phi_0\colon \Indec{\tensunit{k}{l}} \to \tensunit{k}{l} \quad \text{and} \quad \phi\colon \Indec{\hat M} \otimes_l \Indec{\hat N} \to \Indec{\hat M \otimes_l \hat N}.
\]
making $\Indecnoarg$ into a lax monoidal functor, and satisfying the compatibility relation for the enrichments.
Since $\Indecnoarg{\tensunit{k}{l}}=0$ by \eqref{QIisnull}, $\phi_0$ is the zero map.

The map $\phi$ comes from the fact that $\Indecnoarg$ is right adjoint to $\Fr$. Explicitly, it is given as the adjoint of
\[
\Fr(\Indec{\hat M} \otimes_l \Indec{\hat N}) \xrightarrow[\text{Lemma~\ref{gammamonoidal}}]{\phi} \Fr\Indec{\hat M} \otimes_l \Fr\Indec{\hat N} \xrightarrow{\epsilon_{\hat M} \otimes \epsilon_{\hat N}} \hat M \otimes_l \hat N.
\]
\end{proof}
\subsection{Primitives}\label{subsec:prim}

The situation for primitives is almost, but not quite dual to that of indecomposables. Whereas for indecomposables, we considered a pair of adjoint functors $(\Fr,\Indecnoarg)$ where $\Indecnoarg$ was defined levelwise, there will be an adoint pair $(\Primnoarg,\Cof)$ of functors where $\Cof$ is defined levelwise. \label{cofdefn}

A functor $B\colon \Mod_k \to \Mod_l$ in $\fBimod{k}{l}$ can be propagated to a functor
\[
\Cof(B)\colon \Alg_k \to \Mod_l \quad \text{ in }\fModSch{k}{l}
\]
by forgetting the $k$-algebra structure. To see that this is representable, note that the forgetful functor $\Pro-\Alg_k \to \Pro-\Mod_k$ has a left adjoint given levelwise by $\Sym\colon \Pro-\Mod_k \to \Pro-\Alg_k$, the free commutative algebra on a pro-$k$-module. Explicitly, for $M \in \Mod_k$, $\Sym(M)$ is the $k$-algebra $\bigoplus_{i=0}^\infty \Sym^i(M)$, where $\Sym^i(M) = M^{\otimes_k i}/(m_1 \otimes m_2 - m_2 \otimes m_1)$. The adjunction
\begin{equation}\label{symadjunction}
\Pro-\Alg_k(\Sym(M),R) \cong \Pro-\Mod_k(M,R)
\end{equation}
shows that it indeed represents $\Cof(B)$ if $M = \Reg{B}$. Note that $\Sym(M)$ is flat if $M$ is: this is true for free modules, any flat module is a filtered colimit of free modules \cite{lazard:platitude,govorov:flat}, and filtered colimits of flat modules are flat.

This algebra has a cozero $\Sym(M) \to \Sym^0(M)=k$ and a coaddition given by decreeing that all elements of $M = \Sym^1(M)$ are primitive. Furthermore, if $M = \Reg{B}$ for a formal bimodule $B \in \fBimod{k}{l}$, then it has an $l$-action defined by
\[
l \to \Pro-\Mod_k(M,M) \to \Pro-\Alg_k(\Sym(M),\Sym(M)),
\]
which makes the formal $l$-module scheme structure over $k$ explicit. Note that $\Sym(M)$ is in fact a pro-Hopf algebra. Summarizing, we have constructed an enriched functor
\[
\overline\Cof\colon \fBimod{k}{l} \to \fModSch{k}{l} \quad \text{and} \quad \Cof\colon \fBimodf{k}{l} \to \fModSchf{k}{l}.
\]
\begin{defn} \label{def:prim}
The left adjoint of $\Cof$ is denoted by $\Primnoarg = \Reg{P}$, the formal bimodule of primitives of a formal module scheme.
\end{defn}
Of course, one could define the primitives $P(C)$ for a pointed coalgebra $C$, but we will make no use of that. Dually to the situation for indecomposables, we can define $P(\Reg{\hat M})$ explicitly as those elements $a \in \Reg{M}$ such that $\psi_+(a) = a \otimes 1 + 1 \otimes a$. More precisely in our pro-setting, writing $U$ for the forgetful functor $\Sch{k} \to \fModSch{k}{\Z}$, $\Prim{\hat M}$ is the coequalizer of
\begin{equation}\label{primcoeq}
U(\hat M \times \hat M) \rightrightarrows U(\hat M) \to \Prim{\hat M},
\end{equation}
where one map is the addition on $\hat M$ and the other is induced by $\id \otimes 1 + 1 \otimes \id\colon \Reg{\hat M} \to \Reg{\hat M} \otimes_k \Reg{\hat M}$. This also shows that $\Primnoarg$ restricts to a functor $\fModSchf{k}{l} \to \fBimodf{k}{l}$.

For $\hat M \in \fModSch{k}{\Z}$, there is an alternative useful construction of $\Prim{\hat M}$. Note that there are two maps
\[
U(\Fr(\hat M)) \rightrightarrows U(\hat M)
\]
given by the module scheme structure (Cor.~\ref{fmodschcomonadic}) and by the map represented by the inclusion $\Reg{\hat M} \hookrightarrow \Gamma^\Z(\Reg{\hat M})$.

\begin{lemma} \label{primindecinverses}
The functor $\Primnoarg$ is left inverse to $\Fr$ and the functor $\Indecnoarg$ is left inverse to $\Cof$. More precisely, for $\hat M \in \fBimodf{k}{l}$, the natural map
\[
\Fr(\hat M) \to \Cof(\hat M)
\]
has adjoints $\Prim{\Fr(\hat M)} \to \hat M$ and $\hat M \to \Indec{\Cof(\hat M)}$ which are isomorphisms. Moreover, if $\hat X \in \Schf{k}$ then $\Prim{\Fr(\hat X)} = U(\hat X)$.
\end{lemma}
\begin{proof}
One adjoint is an isomorphism if and only if the other is, so it suffices to check this for $\Prim{\Fr(\hat M)}$, which follows from the second claim on a general formal scheme $\hat X$. In that case, note that the coequalizer defining $\Prim{\Fr(\hat X)}$,
\[
U(\Fr(\Fr(\hat X))) \rightrightarrows U(\Fr(\hat X)) \to \Prim{\Fr(\hat X)}
\]
splits by the canonical inclusion $U(\hat X) \to U(\Fr(\hat X))$, thus the composite $U(\hat X) \to \Prim{\Fr(\hat X)}$ gives the claimed isomorphism.
\end{proof}

\begin{lemma} \label{symodot}
The functor $\Cof\colon \fBimodf{k}{l} \to \fModSchf{k}{l}$ is a right strict morphism of $2$-modules over $\Mod_l$, and the functor $\Primnoarg\colon \fModSchf{k}{l} \to \fBimodf{k}{l}$ is left strict, i.e. for $M \in \Mod_{l}$, $\hat M \in \fModSchf{k}{l}$, and $\hat N \in \fBimodf{k}{l}$, the natural maps
\[
\beta\colon \hom_l(M,\Cof(\hat N)) \to \Cof(\hom_l(M,\hat N))
\]
and
\[
\alpha\colon M \indprotensor_l \Prim{\hat M} \to \Prim{M \indprotensor_l \hat M}
\]
are isomorphisms.
\end{lemma}
\begin{proof} Analogous to the proof of Lemma~\ref{qgammatwomodmorph}.
\end{proof}

\begin{prop} \label{Pmonoidal}
The functor $\Primnoarg\colon \fModSchf{k}{l} \to \fBimodf{k}{l}$ is a strict morphism of $2$-algebras over $\Mod_{l}$.
\end{prop}
\begin{proof}
Recall that the units of $\otimes_l$ in $\fModSchf{k}{l}$ and in $\fBimodf{k}{l}$ are both denoted by $\tensunit{k}{l}$ and given by $l \indprotensor \Spf{k}$ and $l \indprotensor_\Z \Spf{k}$, respectively.
The augmentation ideal of $\Reg{l \indprotensor \Spf{k}}$ consists those partial functions $f\colon l \to k$ such that $f(0)=0$; the primitives consist of those functions that are a fortiori additive. Thus,
\[
\Prim{\tensunit{k}{l}} = \tensunit{k}{l}.
\]
For the map $\psi$, first consider the case $l=\Z$. In this case, the tensor product in $\fModSch{k}{\Z}$ was defined in \eqref{ztensorcoeq} as the colimit of
\[
\Fr(\hat M \times \hat M \times \hat N) \rightrightarrows \Fr(\hat M \times \hat N) \leftleftarrows \Fr(\hat M \times \hat N \times \hat N).
\]
Applying $\Primnoarg$, which, being a left adjoint, commutes with colimits, and applying Lemma~\ref{primindecinverses}, we obtain that $\Prim{\hat M \otimes_\Z \hat N}$ is the colimit of
\begin{equation}\label{primtensorcoeq}
U(\hat M \times \hat M \times \hat N) \rightrightarrows U(\hat M \times \hat N) \leftleftarrows U(\hat M \times \hat N \times \hat N).
\end{equation}
Note here that $U(\hat M \times \hat N) \cong U(\hat M) \otimes_\Z U(\hat N)$, where $\otimes_\Z$ is the tensor product of formal bimodules.
The upper maps in \eqref{primtensorcoeq} are given by the addition on $M$ and $N$, respectively, and the lower maps are obtained from the map $U(\hat M \times \hat M) \to U(\hat M)$ of \eqref{primcoeq} by tensoring with the identity on $U(\hat N)$. 

Using the left hand side of \eqref{primtensorcoeq}, by \eqref{primcoeq} there is a natural map
\[
\coeq\left(U(\hat M) \otimes_\Z U(\hat M) \otimes_\Z U(\hat N) \rightrightarrows U(\hat M) \otimes_\Z U(\hat N)\right) \to \Prim{M} \otimes_\Z U(\hat N)
\]
which is an isomorphism since $- \otimes_\Z U(\hat N)$ is exact, i.~e. since $\hat N$ is flat. Applying the same reasoning to the right hand side, we obtain the desired map
\[
\psi\colon \Prim{\hat M \otimes_\Z \hat N} \to \Prim{\hat M} \otimes_\Z \Prim{\hat N},
\]
which is an isomorphism since in addition, $\Prim{\hat M} \otimes_\Z -$ is exact because $P(\Reg{\hat M})$ is pro-flat. Clearly the roles of $M$ and $N$ can be reversed.

For general $l$, we can write $\hat M \otimes_l \hat N$ as a coequalizer \eqref{tensorcoeq}
\[
l \indprotensor_\Z (\hat M \otimes_\Z \hat N) \rightrightarrows \hat M \otimes_\Z \hat N \to \hat M \otimes_l \hat N,
\]
which by the $l=\Z$ case gives a diagram
\[
\begin{tikzpicture}
	\matrix (m) [matrix of math nodes, row sep=2em, column sep=2em, text height=1.5ex, text depth=0.25ex]
	{
		\Prim{l \indprotensor_\Z (\hat  M \otimes_\Z \hat N)}		&  \Prim{\hat M \otimes_\Z \hat N}		& \Prim{\hat M \otimes_l \hat N}\\
		l \indprotensor_\Z (\Prim{\hat M} \otimes_\Z \Prim{\hat N})	&  (\Prim{\hat M} \otimes_\Z \Prim{\hat N})	& \Prim{\hat M} \otimes_l \Prim{\hat N}\\
	};
	\path[->,font=\scriptsize]
	([yshift=-1mm]m-1-1.east)		edge ([yshift=-1mm]m-1-2.west)
	([yshift=+1mm]m-1-1.east)	edge ([yshift=+1mm]m-1-2.west)
	([yshift=-1mm]m-2-1.east)		edge ([yshift=-1mm]m-2-2.west)
	([yshift=+1mm]m-2-1.east)	edge ([yshift=+1mm]m-2-2.west)
	(m-1-2)	edge (m-1-3)
	(m-2-2) edge (m-2-3)
	(m-1-1)	edge (m-2-1)
	(m-1-2)	edge (m-2-2)
	(m-1-3)	edge[dashed] (m-2-3);
\end{tikzpicture}
\]
The lower row is a coequalizer diagram by Def.~\ref{def:sweedlerprod}; the upper row by Lemma~\ref{symodot}, and since $\Primnoarg$ commutes with coequalizers. Hence the dotted arrow $\psi$ exists and is an isomorphism.
\end{proof}

The following corollary follows immediately from the adjunctions.
\begin{corollary} \label{symopmonoidal}
The functor $\Cof$ is a lax morphism of $2$-algebras, i.e. for $\hat M$, $\hat N \in \fBimodf{k}{l}$, there are natural transformations
\[
\phi_0\colon \tensunit{k}{l} \to \Cof(\tensunit{k}{l}) \quad \text{and} \quad
\phi\colon \Cof(\hat M) \otimes_l \Cof(\hat N) \to \Cof(\hat M \otimes_l \hat N)
\]
turning $\Cof$ into a lax monoidal functor. \qed
\end{corollary}

\subsection{Primitives and indecomposables of formal algebra schemes} \label{subsec:pqfalgsch}

We will now consider what the additional structure of a formal \emph{algebra} scheme maps to under the functors $\Primnoarg$ and $\Indecnoarg$.

Let $\hat A \in \fAlgSchf{k}{l}$ be a flat formal $l$-algebra scheme over $k$ and let $R \in \Alg_l^+$ be an $l$-algebra augmented by $\epsilon_R\colon R \to l$. Then $\hom_l^\times(R,\hat A)$ (cf. Section~\ref{sec:falgschstructure}) is a pointed formal scheme over $k$ by the map $\Spf{k} \to \hom_l^\times(R,\hat A)$ adjoint to $R \xrightarrow{\epsilon_R} l \xrightarrow{\epsilon} \hat A(k)$. This produces thus a functor
\[
\hom_l^{\times}(-,\hat A)\colon (\Alg_{l}^+)^\op \to (\Schf{k})^+
\]
which has a left adjoint
\begin{equation}\label{ahatplus}
\hat A_+ = l \times_{\hat A(k)} \hat A(-)\colon \Alg_k^+ \to \Alg_l^+,
\end{equation}
an augmented version of $\hat A$ with $\hat A_+(k)=l$.

\begin{prop} \label{basicalgschadj}
For $\hat A \in \fAlgSchf{k}{l}$, $M \in \Mod_k^f$ there is a natural isomorphism
\[
\hat A_+(k \oplus M) \cong l \oplus \Indec{\hat A}(M) \in \Alg_l^+
\]
\end{prop}
\begin{proof}
Equivalently, we need to see that
\begin{equation}\label{augmentedalgebraiso}
\Pro-\Alg_k^+(\Reg{\hat A},k \oplus M) \cong \Pro-\Alg_k^+(\Reg{\hat A},k) \oplus \Pro-\Mod_k(Q\Reg{\hat A},M).
\end{equation}
An augmented pro-algebra map $\Reg{\hat A} \to k \oplus M$ consists of a pair $(f,g)$ with $f\colon \Reg{\hat A} \to k$ and $g\colon \Reg{\hat A} \to \hat M$ such that $f$ is an algebra map, $g$ is a $k$-module map, and $g(ab) = f(a)g(b)+g(a)f(b)$. We produce the isomorphism \eqref{augmentedalgebraiso} by sending $(f,g)$ to $(f, a \mapsto g(a-f(a)))$. To see that this is an isomorphism of $l$-algebras, with the square-zero structure on the right hand side, we note that the comultiplication $\psi_\times$ restricts to a map
\[
\psi_\times\colon (\Reg{\hat A})_+ \to (\Reg{\hat A})_+ \otimes_k (\Reg{\hat A})_+
\]
because, dually, multiplication with $0$ on either side gives $0$. Thus the product of any two maps in $\Pro-\Mod_k(Q\Reg{\hat A}, M)$ is zero.
\end{proof}

\begin{corollary} \label{Qodotmonoidal}
For $R \in \Alg_{l}^+$ and $\hat A \in \fAlgSchf{k}{l}$,
\[
\Indec{\hom_l^\times(R,\hat A)} \cong \hom_l(Q(R),\Indec{\hat A}) \in \fBimod{k}{\Z}
\]
\end{corollary}
\begin{proof}
This follows from the adjunction of Prop.~\ref{basicalgschadj}: for $M \in \Mod_k$,
\begin{align*}
\Indec{\hom_l^\times(R,\hat A)}(M) \cong & \hom(R,\hat A)(k \oplus M)
\cong  \Alg_l^+(R,\hat A_+(k \oplus M))\\
\underset{\text{Prop.~\ref{basicalgschadj}}}\cong & \Alg_l^+(R,l \oplus \Indec{\hat A}(M))
\cong  \Mod_l(QR,\Indec{\hat A}(M))\\
\cong & \hom_l(QR,\Indec{\hat A})(M).
\end{align*}
\end{proof}

In order to understand how $\Primnoarg$ behaves with respect to the pairing
\[
\hom_l^\times\colon \Alg_l \times \fAlgSchf{k}{l} \to \Schf{k}
\]
of Section~\ref{sec:falgschstructure}, we need the following lemma, whose proof is a short series of standard adjunctions and will be left to the reader.
\begin{lemma} \label{odotleftfree}
Let $M \in \Mod_{l}$, $\hat A \in \fAlgSchf{k}{l}$, and denote by $U^\times(\hat A) \in \fModSchf{k}{l}$ the formal module scheme obtained by forgetting the multiplicative structure. Then there is a natural isomorphism
\[
\hom_l^\times(\Sym(M),\hat A) \cong \hom_l(M,U^\times(\hat A)).
\]
\end{lemma}

To even make sense of a compatibility statement of $\Primnoarg$ with $\hom_l^\times(R,\hat A)$, the latter has to have a Hopf algebra structure. By naturality in the $R$-variable, this happens if $R$ is a Hopf algebra. Denote the category of $l$-Hopf algebras (or $\Z$-module schemes over $l$) by $\ModSch{l}{\Z}$.
\begin{prop}\label{Podotmonoidal}
For $R \in \ModSch{l}{\Z}$ and $\hat A \in \fAlgSchf{k}{l}$, there is a natural map
\[
\beta\colon \Prim{\hom_l^\times(R,\hat A)} \to \hom_l(P(R), \Prim{\hat A}).
\]
\end{prop}
\begin{proof}
The natural map $\beta$ is given as the adjoint of the map
\begin{align*}
\hom_l^\times(R,\hat A) \xrightarrow{\epsilon^*} & \hom_l^\times(\Sym(P(R)),\hat A)
\xrightarrow[\cong]{\text{Lemma~\ref{odotleftfree}}} \hom_l(P(R),U^\times(\hat A))\\
\xrightarrow{\eta_*} & \hom_l(P(R),\Cof(\Prim{\hat A}))
\xrightarrow[\cong]{\text{Lemma~\ref{symodot}}}  \Cof(\hom_l(P(R),\Prim{\hat A})),
\end{align*}
where $\epsilon\colon \Sym(P(R)) \to R$ and $\eta\colon \hat A \to \Cof(\Prim{\hat A})$ are counit and unit of the respective adjunctions. 
\end{proof}

The transformation $\beta$ will in general not be an isomorphism. Indeed, if we choose $R = l[z]$ as a polynomial ring then $\hom_l^\times(R,\hat A) = \hom_l(l,U^\times(\hat A)) = U^\times(\hat A)$ by Lemma~\ref{odotleftfree} and thus $\Prim{\hom_l^\times(R,\hat A)}\cong \Prim{\hat A}$. On the other hand, $P(R) = P(l[z])$ is in general greater than $l$ (for instance if $l$ has positive characteristic), and thus $\hom_l(P(R),\Prim{\hat A}) \neq \Prim{\hat A}$.

\begin{lemma}\label{psymodotcompat}
The following diagram commutes:
\[
\begin{tikzpicture}
	\matrix (m) [matrix of math nodes, row sep=3em, column sep=4.5em, text height=1.5ex, text depth=0.25ex]
	{
		\Prim{\hom_l(PR,U^\times(\hat A))} & \Prim{\hom_l^\times(\Sym P R,\hat A)}\\
		\hom_l(PR,\Prim{\hat A}) & \Prim{\hom_l^\times(R,\hat A)}\\
	};
	\path[<-,font=\scriptsize]
	(m-1-1)	edge node[auto]{Lemma~\ref{odotleftfree}} (m-1-2)
			edge node[auto]{$\Mod_l$-enrichment}  node[auto,swap]{$\beta$} (m-2-1)
	(m-2-1)	edge node[auto]{$\beta$}	(m-2-2)
	(m-1-2)	edge node[auto]{$\epsilon^*$}	(m-2-2);
\end{tikzpicture}
\]
\end{lemma}
\begin{proof}
By adjointness of $\Primnoarg$ and $\Cof$, it suffices to show that the exterior of the following diagram commutes:
\[
\begin{tikzpicture}
	\matrix (m) [matrix of math nodes, row sep=3em, column sep=2.2em, text height=1.5ex, text depth=0.25ex]
	{
		\Cof \Prim{\hom_l(PR,U^\times(\hat A))} & \hom_l(PR,U^\times(\hat A)) & \hom_l^\times(\Sym P R,\hat A)\\
		\Cof \hom_l(PR,\Prim{\hat A}) & & \hom_l^\times(R,\hat A)\\
		\hom_l(PR, \Cof \Prim{\hat A}) & \hom_l(PR,U^\times(\hat A)) & \hom_l^\times(\Sym P R,\hat A)\\
	};
	\path[<-,font=\scriptsize]
	(m-1-1)	edge node[auto]{$\eta$} (m-1-2)
			edge node[auto]{$\Cof(\beta)$} (m-2-1)
	(m-1-2)	edge node[auto]{\ref{odotleftfree}} node[auto,swap]{$\cong$} (m-1-3)
			edge[-,double,double equal sign distance]  (m-3-2)
	(m-1-3)	edge node[auto]{$\epsilon^*$} (m-2-3)
	(m-2-1)	edge node[auto]{Lemma~\ref{symodot}} node[auto,swap]{$\cong$} (m-3-1)
	(m-3-1)	edge node[auto]{$\epsilon^*$} (m-3-2)
	(m-3-2)	edge node[auto]{\ref{odotleftfree}} node[auto,swap]{$\cong$} (m-3-3)
	(m-3-3)	edge node[auto,swap]{$\hom_l^\times(\eta,\hat A)$} (m-2-3);
\end{tikzpicture}
\]
In fact, the smaller pentagon on the left already commutes because $\Cof$ and $\Primnoarg$ are enriched functors.
\end{proof}

From Corollary~\ref{comultintotensor} we find that the multiplication $\mu_\times$ on $\hat A$ gives rise to morphisms of bimodule schemes
\[
\Indec{\hat A} \otimes_l \Indec{\hat A} \xrightarrow{\text{Cor.~\ref{Qopmonoidal}}} \Indec{\hat A \otimes_l \hat A} \xrightarrow{\Indec{\mu_\times}} \Indec{\hat A}
\]
and, under the conditions in Prop.~\ref{Pmonoidal} for $\Primnoarg$ being strict,
\[
\Prim{\hat A} \otimes_l \Prim{\hat A} \xrightarrow{\text{Prop.~\ref{Pmonoidal}}} \Prim{\hat A \otimes_l \hat A} \xrightarrow{\Prim{\mu_\times}} \Prim{\hat A}
\]
Thus both $\Indec{\hat A}$ and $\Prim{\hat A}$ are algebras in $(\fBimod{k}{l},\otimes_l,\tensunit{k}{l})$.

\begin{prop}\label{Qofmulttensor}
Let $\hat A$, $\hat B \in \fAlgSchf{k}{l}$. Then the lax structure of $\Indecnoarg$ of Cor.~\ref{Qopmonoidal} factors as a map
\[
\Indec{\hat A} \otimes_l \Indec{\hat B} \to \Indec{\hat A} \otimes_l \Prim{\hat B} \times_{\Indec{\hat A} \otimes_l \Indec{\hat B}} \Prim{\hat A} \otimes_l \Indec{\hat B} \to \Indec{\hat A \otimes_l \hat B}.
\]
For $\hat A=\hat B$, this makes $\Indec{\hat A}$ into a two-sided module over the algebra $\Prim{\hat A}$. 
\end{prop}
\begin{proof}
By commutativity, it suffices to produce the factorization 
\[
\Indec{\hat A} \otimes_l \Indec{\hat B} \to \Indec{\hat A} \otimes_l \Prim{\hat B} \to \Indec{\hat A \otimes_l \hat B}
\]
First assume $l=\Z$. The first step is to produce a factorization in $\Pro-\Mod_k$
\[
Q(\Reg{\hat A} \multtensor{k}{\Z} \Reg{\hat B}) \to Q\Reg{\hat A} \otimes_k (\Reg{\hat B})_+ \to Q\Reg{\hat A} \otimes_{k} Q\Reg{\hat B},
\]
where $(\Reg{\hat B})_+$ denotes the augmentation ideal of $\Reg{\hat B}$.

The adjunction between $U^\Z_+\colon {\fModSch{k}{\Z}}^{\op} \to \Pro-\Alg_k^+$ and $\Gamma_+$ gives a unit map
\[
\Reg{\hat B} \xrightarrow{\eta'} \Gamma^+(U^\Z_+(\Reg{\hat B})) = \Gamma((\Reg{\hat B})_+)
\]
which together with the unit adjunction $\Reg{\hat A} \xrightarrow{\eta} \Gamma Q\Reg{\hat A}$ gives a map
\begin{multline}\label{Qtensorplus}
\Reg{\hat A} \multtensor{k}{\Z} \Reg{\hat B} \to \Gamma^+(k \oplus Q\Reg{\hat A} \multtensor{k}{\Z} \Gamma^+ U^\Z_+\Reg{\hat B}) \\\xrightarrow[\cong]{\text{Lemma~\ref{tensoroffreemodsch}}} \Gamma^+((k \oplus Q\Reg{\hat A}) \otimes_k^+ U^\Z_+\Reg{\hat B}).
\end{multline}
Here $\otimes_k^+$ denotes the tensor product of augmented pro-$k$-algebras, i.e. the operation classifying the smash product of two pointed formal schemes. Since the augmented algebra $k \oplus Q\Reg{\hat A}$ is a square-zero extension, so is
\[
(k \oplus Q\Reg{\hat A}) \otimes_k^+ U^\Z_+\Reg{\hat B} \cong k \oplus (Q\Reg{\hat A} \otimes_k (\Reg{\hat B})_+).
\]
The composite of \eqref{Qtensorplus} with the isomorphism
\[
\Gamma^+((k \oplus Q\Reg{\hat A}) \otimes_k^+ U^\Z_+\Reg{\hat B}) \cong \Gamma^+(k \oplus (Q\Reg{\hat A} \otimes_k (\Reg{\hat B})_+)) \cong \Gamma(Q\Reg{\hat A} \otimes_k (\Reg{\hat B})_+)
\]
has an adjoint $\tilde\psi\colon Q(\Reg{\hat A} \multtensor{k}{\Z} \Reg{\hat B}) \to Q\Reg{\hat A} \otimes_k (\Reg{\hat B})_+$, which is the desired factorization.

To show that this map factors further through $Q\Reg{\hat A} \otimes_k P\Reg{\hat B}$, it suffices by the flatness assumption on $Q\Reg{\hat A}$ to show that
\[
Q(\Reg{\hat A} \multtensor{k}{\Z} \Reg{\hat B}) \to Q\Reg{\hat A} \otimes (\Reg{\hat B})_+ \xrightarrow{\id \otimes \psi_+} Q\Reg{\hat A} \otimes (\Reg{\hat B})_+ \otimes (\Reg{\hat B})_+
\]
is null. This map fits into a diagram
{\small\[
\begin{tikzpicture}
	\matrix (m) [matrix of math nodes, row sep=3em, column sep=2em, text height=1.5ex, text depth=0.25ex]
	{
		Q\Reg{\hat A} \otimes (\Reg{\hat B})_+ \otimes (\Reg{\hat B})_+ 	& Q(\Reg{\hat A} \multtensor{k}{\Z} (\Reg{\hat B} \otimes \Reg{\hat B})) & Q\left((\Reg{\hat A} \multtensor{k}{\underline \Z} \Reg{\hat B})^{\otimes 2}\right)\\
		Q\Reg{\hat A} \otimes (\Reg{\hat B})_+ &	Q(\Reg{\hat A} \multtensor{k}{\Z} \Reg{\hat B}) & Q(\Reg{\hat A} \multtensor{k}{\Z} \Reg{\hat B})^2.\\
	};
	\path[->,font=\scriptsize]
	(m-1-2)	edge node[auto,swap]{$\tilde\psi$} (m-1-1)
	(m-1-3)	edge node[auto,swap]{$\cong$}(m-1-2)
	(m-2-1)	edge node[auto]{$QA \otimes \psi_+$} (m-1-1)
	(m-2-2)	edge node[auto]{$\tilde\psi$} (m-2-1)
			edge node[auto]{$Q(\id \multtensor{k}{\Z} \psi_+)$} (m-1-2)
			edge node[auto]{$Q(\psi_+)$} (m-1-3)
			edge node[auto,swap]{$\Delta$} (m-2-3)
	(m-2-3)	edge node[auto]{$\cong$} (m-1-3);
\end{tikzpicture}
\]
}
The composites $Q(\Reg{\hat A} \multtensor{k}{\Z} \Reg{\hat B}) \to Q(\Reg{\hat A} \multtensor{k}{\Z} (\Reg{\hat B} \otimes \Reg{\hat B}))$ from the lower right hand corner to the upper middle entry are induced by the inclusions $\eta \times \id,\; \id \times \eta\colon \Reg{\hat B} \to \Reg{\hat B} \otimes \Reg{\hat B}$ and thus both compose to the zero map when followed with $\tilde\psi$.

It remains to produce a formal bimodule map $\Indec{\hat A} \otimes_l \Prim{\hat B} \to \Indec{\hat A \otimes_l \hat B}$ for arbitrary rings $l$. For this, consider the diagram
\[
\begin{tikzpicture}
	\matrix (m) [matrix of math nodes, row sep=2em, column sep=2em, text height=1.5ex, text depth=0.25ex]
	{
		\Indec{\hat A \otimes_l \hat B} & \Indec{\hat A \otimes_\Z \hat B} & \Indec{l \otimes_\Z \hat A \otimes_\Z \hat B}\\
		& & l \indprotensor_\Z \Indec{\hat A \otimes_\Z \hat B}\\
		\Indec{\hat A} \otimes_l \Indec{\hat B} & \Indec{\hat A} \otimes_\Z \Indec{\hat B} & l \otimes \Indec{\hat A} \otimes_\Z \Indec{\hat B}.\\
	};
	\path[<-,font=\scriptsize]
	(m-1-1)	edge (m-1-2)
	([yshift=-1mm]m-1-2.east)	 edge ([yshift=-1mm]m-1-3.west)
	([yshift=1mm]m-1-2.east)	 edge ([yshift=1mm]m-1-3.west)
	(m-1-2) edge node[auto]{$\tilde\psi$} (m-3-2)
	(m-1-3) edge node[auto]{Lemma~\ref{qgammatwomodmorph}} (m-2-3)
	(m-2-3) edge node[auto]{$\tilde\psi$} (m-3-3)
	(m-3-1) edge (m-3-2)
	([yshift=-1mm]m-3-2.east)	 edge ([yshift=-1mm]m-3-3.west)
	([yshift=1mm]m-3-2.east)	 edge ([yshift=1mm]m-3-3.west)
	(m-1-1)	edge[dotted] (m-3-1);
\end{tikzpicture}
\]
The upper row is not necessarily a coequalizer, but a fork, coming from the definition of the tensor product \eqref{tensorcoeq}. The lower row is the coequalizer defining the tensor product of formal bimodules. Therefore the dotted map exists.
\end{proof}

\subsection{Primitives and indecomposables of formal plethories}\label{subsec:pqfpleth}

Now let $\hat A\in \fAlgSchf{k}{k}$ be a formal plethory. By Lemma~\ref{Qopmonoidal} and Prop.~\ref{Pmonoidal}, $\Prim{\hat A}$ and $\Indec{\hat A}$ are both algebras with respect to the $\otimes_k$-product. 
By Prop.~\ref{Qofmulttensor}, $\Indec{\hat A}$ is also $\Prim{\hat A}$-bimodule algebra with respect to $\otimes_k$. 

\begin{lemma}\label{Symhatodotmonoidal}
The functor $\Cof\colon \fBimodf{k}{k} \to \fAlgSchf{k}{k}$ is strict monoidal with respect to the $\circ_k$-product.
\end{lemma}
\begin{proof}
The unit isomorphism $J_k = \Spf{k[z]} \to \Cof(\Spf{k}) = \Cof(J_k)$ is represented by the canonical isomorphism $\Sym(k) \cong k[z]$. If $\hat M$, $\hat N \in \fBimodf{k}{k}$, we construct an isomorphism
\begin{align*}
\Cof(\hat M) \circ_k \Cof(\hat N) = & \colim_i \hom_k^\times(\Sym(\Reg{\hat M}(i)),\Cof(\hat N))\\
\underset{\text{Lemma~\ref{odotleftfree}}}\cong & \colim_i \hom_k(\Reg{\hat M}(i),U^\times(\Cof(\hat N)))\\
\underset{\text{Lemma~\ref{symodot}}}\cong &\colim_i \Cof(\hom_k(\Reg{\hat M}(i),\hat N))\\
\cong & \Cof(\colim_i \hom_k(\Reg{\hat M}(i),\hat N)) \cong \Cof(\hat M \circ_k \hat N).
\end{align*}
Although $\Cof$ is a right adjoint, it does commute with filtered colimits (used in the last line) because its representing functor $\Sym$ is induced up from $\Mod_k \to \Alg_k$ to $\Pro-\Mod_k \to \Pro-\Alg_k$ by objectwise application.
\end{proof}

\begin{thm} \label{Cofbilaxmonoidal}
The functor $\Cof$ is a bilax monoidal functor.
\end{thm}
\begin{proof}
We showed in Lemma~\ref{Symhatodotmonoidal} and Corollary~\ref{symopmonoidal} that $\Cof$ is strict monoidal with respect to $\circ_k$ and lax monoidal with respect to $\otimes_k$, and it remains to check the compatibility conditions \eqref{biunitality}--\eqref{bilaxcompat}.

For \eqref{biunitality}, first recall that the comparison map between the units $\iota_J\colon \tensunit{k}{k} \to J_k$ is represented by
\[
k \to \hom_\Z(k,k); \quad 1 \mapsto \id \quad \text{in }\fBimod{k}{k}
\]
and
\[
k[x] = \Sym(k) \to \hom(k,k); \quad x \mapsto \id \quad \text{in } \fAlgSch{k}{k}.
\]
The commutativity of \eqref{biunitality} follows from the factorization
\[
k \xrightarrow{\iota_{\hom_\Z(k,k)}} \hom_\Z(k,k) \hookrightarrow U(\hom(k,k)),
\]
where $U$ is the forgetful functor from $\Pro-\Alg_k$ to $\Pro-\Mod_k$ and the last map is the inclusion of linear maps into all maps.

For \eqref{leftunitality}, we proceed in two steps. First we use the tensoring $\overline \circ_k$ between $\fBimod{k}{k}$ and $\fModSchf{k}{k}$ from Lemma~\ref{circmonoidal} to prove that the following diagram commutes:
\[
\begin{tikzpicture}
	\matrix (m) [matrix of math nodes, row sep=2em, column sep=3em, text height=1.5ex, text depth=0.25ex]
	{
		\tensunit{k}{k} & \Cof(\tensunit{k}{k})	&	 											\Cof(\tensunit{k}{k} \circ_k \tensunit{k}{k}))\\
		&\tensunit{k}{k} \overline \circ_k  U^\times(\tensunit{k}{k}) &		\tensunit{k}{k} \overline \circ_k  U^\times\Cof(\tensunit{k}{k}) \\
		\tensunit{k}{k} \circ_k \tensunit{k}{k} & \Cof(\tensunit{k}{k}) \circ_k \tensunit{k}{k} &				 	\Cof\tensunit{k}{k}\circ_k \Cof\tensunit{k}{k}\\
	};
	\path[->,font=\scriptsize]
	(m-1-1)		edge node[auto]{$\psi_0$}					(m-1-2)
				edge node[auto]{$\mu$}					(m-2-2)
				edge node[auto]{$\Delta_I$}				(m-3-1)
	(m-1-2)		edge node[auto]{$\Cof(\Delta_I)$}			(m-1-3)
	(m-1-3)		edge node[auto]{$\beta$}					(m-2-3)
	(m-2-2)		edge node[auto]{$\id \circ \psi_0$}			(m-2-3)
	(m-2-2)		edge node[auto]{$\cong$}					(m-3-2)
	(m-2-3)		edge node[auto]{$\cong$}					(m-3-3)
	(m-3-1)		edge node[auto]{$\phi_0 \circ \id$}		(m-3-2)
	(m-3-2)		edge node[auto]{$\id \circ \psi_0$}			(m-3-3);
\end{tikzpicture}
\]
The commutativity of the top pentagon is Lemma~\ref{laxmorphproperties}, while the commutativity of the lower right hand square is the naturality of the isomorphism of Lemma~\ref{odotleftfree}. 

The commutativity of \eqref{rightunitality} is formal because $\otimes_k$ is the categorical coproduct in $\fAlgSchf{k}{k}$.

It remains to show the commutativity of \eqref{bilaxcompat}. We proceed similarly as above and first note that the following diagram commutes for $\hat M_i$, $\hat N_i \in \fBimodf{k}{k}$ because of Lemma~\ref{laxmorphproperties}:
\[
\begin{tikzpicture}
	\matrix (m) [matrix of math nodes, row sep=2em, column sep=2.5em, text height=1.5ex, text depth=0.25ex]
	{
		\Cof(\hat M_1 \circ_k \hat M_2) \otimes_k \Cof(\hat N_1 \circ_k \hat N_2) &  \Cof((\hat M_1 \circ_k \hat M_2) \otimes_k (\hat N_1 \circ_k \hat N_2))\\ 
		(\hat M_1 \overline \circ_k \Cof(\hat M_2)) \otimes_k (\hat N_1 \overline \circ_k \Cof(\hat N_2)) &  \Cof((\hat M_1 \otimes_k \hat N_1) \circ_k (\hat M_2 \otimes_k \hat N_2))\\
		(\hat M_1 \otimes_k \hat N_1) \overline \circ_k (\Cof(\hat M_2) \otimes_k \Cof(\hat N_2)) & (\hat M_1 \otimes_k \hat N_1) \circ_k \Cof(\hat M_2 \otimes_k \hat N_2)\\	
	};
	\path[<-,font=\scriptsize]
	(m-1-2)		edge node[auto]{$\phi$}					(m-1-1)
	(m-2-1)		edge node[auto]{$\cong$}					(m-1-1)
	(m-2-2)		edge node[auto]{$\Cof(\zeta)$}			(m-1-2)
	(m-3-1)		edge node[auto]{$\zeta$}					(m-2-1)
	(m-3-2)		edge node[auto]{$\id \circ \phi$}					(m-3-1)
				edge node[auto]{$\cong$}					(m-2-2);
\end{tikzpicture}
\]
The result then follows from the commutativity of the following diagram:
{\small
\[
\begin{tikzpicture}
	\matrix (m) [matrix of math nodes, row sep=2em, column sep=1em, text height=1.5ex, text depth=0.25ex]
	{
		(\Cof \hat M_1 \circ_k \Cof \hat M_2) \otimes_k (\Cof \hat N_1 \circ_k \Cof \hat N_2) &  (\hat M_1 \overline\circ_k \Cof \hat M_2) \otimes_k (\hat N_1 \overline\circ_k \Cof \hat N_2)\\
		(\Cof \hat M_1 \otimes_k \Cof \hat N_1) \circ_k (\Cof \hat M_2 \otimes_k \Cof \hat N_2)\\
		(\Cof(\hat M_1 \otimes_k \hat N_1)) \circ_k (\Cof \hat M_2 \otimes_k \Cof \hat N_2) & (\hat M_1 \otimes_k \hat N_1) \overline\circ_k (\Cof \hat M_2 \otimes_k \Cof \hat N_2)\\
		\Cof(\hat M_1 \otimes_k \hat N_1) \circ_k \Cof(\hat M_2 \otimes_k \hat N_2) & (\hat M_1 \otimes_k \hat N_1) \overline\circ_k \Cof(\hat M_2 \otimes_k \hat N_2)\\	
	};
	\path[->,font=\scriptsize]
	(m-1-1)		edge node[auto]{$\cong$}					(m-1-2)
	(m-2-1)		edge node[auto]{$\zeta$}					(m-1-1)
	(m-3-1)		edge node[auto]{$\phi \circ_k \id$}	(m-2-1)
				edge node[auto]{$\cong$}					(m-3-2)
	(m-3-2)		edge node[auto]{$\zeta$}					(m-1-2)
	(m-4-1)		edge node[auto]{$\id \circ_k \phi$}	(m-3-1)
				edge node[auto]{$\cong$}					(m-4-2)
	(m-4-2)		edge node[auto]{$\id\circ_k \phi$}	(m-3-2);
\end{tikzpicture}
\]
}
\end{proof}

\begin{lemma} \label{Qhatodotmonoidal}
The functor $\Indecnoarg\colon \fAlgSchf{k}{k} \to \fBimodf{k}{k}$ is a strict monoidal functor with respect to the $\circ_k$-product.
\end{lemma}
\begin{proof}
The unit isomorphism $J_k\to \Indec{J_k}$ is represented by the canonical isomorphism $Q(k[z]) \cong k$. If $\hat A$, $\hat B \in \fAlgSch{k}{k}$, in particular $\hat A \in \Sch{k}^+$, we construct an isomorphism
 \begin{multline*}
\Indec{\hat A} \circ_k \Indec{\hat B} =  \colim_i \hom_k(Q\Reg{\hat A}(i),\Indec{\hat B})\\
\underset{\text{Cor.~\ref{Qodotmonoidal}}}\cong  \colim_i \Indec{\hom_k^\times(\Reg{\hat A}(i),\hat B)}
\cong \Indec{\colim_i \hom_k^\times(\Reg{\hat A}(i),\hat B)} \cong \Indec{\hat A \circ_k \hat B}.
\end{multline*}
As in the proof of Lemma~\ref{Symhatodotmonoidal}, $\Indecnoarg$ commutes with filtered colimits  because its representing functor $Q$ is induced up from $\Alg_k^+ \to \Mod_k$ to $\Pro-\Alg_k^+ \to \Pro-\Mod_k$ by objectwise application.
\end{proof}

\begin{thm} \label{Qbilaxmonoidal}
The functor $\Indecnoarg$ is a bilax monoidal functor.
\end{thm}
\begin{proof}
It was shown in Lemma~\ref{Qhatodotmonoidal} and Cor.~\ref{Qopmonoidal} that $\Indecnoarg$ is strictly monoidal with respect to $\circ_k$ and lax monoidal with respect to $\otimes_k$. As in Thm.~\ref{Cofbilaxmonoidal}, it remains to check the compatibility conditions \eqref{biunitality}--\eqref{bilaxcompat}. The proofs proceed in the same way, utilizing Lemma~\ref{laxmorphproperties} throughout.
\end{proof}

\begin{corollary}\label{Pbilaxmonoidal}
The functors $\Primnoarg\colon \fAlgSchf{k}{k} \to \fBimodf{k}{k}$ and $\Fr\colon \fBimodf{k}{k} \to \fAlgSchf{k}{k}$ are bilax monoidal functors.
\end{corollary}
\begin{proof}
This follows immediately from Thm.~\ref{Cofbilaxmonoidal} and Thm.~\ref{Qbilaxmonoidal} in conjunction with Prop.~\ref{bilaxcondition}.
\end{proof}

It follows that $\Prim{\hat A}$ and $\Indec{\hat A}$ are formal $k$-coalgebras if $\hat A$ is a formal plethory. In fact, by Lemma~\ref{bilaxfunctorspreservebimonoids}, more is true:
\begin{thm}
Let $\hat A \in \fAlgSchf{k}{k}$ be a formal plethory. Then $\Prim{\hat A}$ and $\Indec{\hat A} \in \fBimodf{k}{k}$ are bimonoids. \qed
\end{thm}

As a corollary of Prop.~\ref{Qofmulttensor}, we also get a stronger statement for $\Indec{\hat A}$:
\begin{corollary}
The bimonoid $\Indec{\hat A}$ is also a two-sided module over $\Prim{\hat A}$ with respect to the $\otimes_k$-product.
\end{corollary}

\section{Dualization} \label{section:dualization}

In this section, we will consider the dualization of bimodules. We will show that the subcategory of $\fBimod{k}{l}$ consisting of objects represented by pro-finitely generated free $k$-modules is equivalent to the category $\Bimod{k}{l}$ of ordinary $k$-$l$-bimodules which are flat as $k$-modules. Furthermore, the categories are equivalent as $2$-monoidal categories if one defines the $2$-monoidal structures on $\Bimod{k}{l}$ as in Example~\ref{exa:bimodtwomonoidal}.

Let $\hat M \in \fBimod{k}{l}$ be a formal bimodule. Define $\predual{\hat M}$ to be the object in $\Bimod{l}{k}$
\[
(\predual{\hat M})_p^q = \hat M(\Sigma^q k)_p =  \Pro-\Mod_k((\Reg{\hat M})_p^q,k).
\]
From this description, it is obvious that it is a $k$-module (because $\Mod_k(\Reg{\hat M},k)$ is) and that it is an $l$-module (because $\hat M\colon \Mod_{k} \to \Mod_{l}$).

Conversely, given an object $M \in \Bimod{l}{k}$, we can define a dual
\[
(\dual{M})_p^q = \tensunit{k}{M} = M \indprotensor_k \Spf{k} \in \fBimod{k}{l} \quad \text{(cf. Example~\ref{exa:bimodhom})}.
\]
To be explicit, this is the pro-$k$-module $\{\Hom_k(F,k)\}$ where $F$ runs through all finitely presented $k$-submodules of $M$.

These functors are not equivalences of categories, but suitable restrictions are. Denote by $\fBimod{k}{l}'$ the full subcategory of formal bimodules which are represented by pro-(finitely generated free) $k$-modules, and $\Bimod{k}{l}'$ the full subcategory of $k$-$l$-bimodules which are flat as $k$-modules.

\begin{lemma}\label{doubledual}
The restrictions of the functors $\predual{(-)}\colon \fBimod{k}{l}' \to \Bimod{k}{l}':\!\dual{(-)}$ are inverse equivalences of categories.
\end{lemma}
\begin{proof}
First assume $l=\Z$. By the Govorov-Lazard theorem \cite{lazard:platitude,govorov:flat}, the category of flat $k$-modules $\Bimod{k}{\Z}'$ is equivalent to the category of ind-finitely generated free $k$-modules. Since the category of finitely generated free $k$-modules is self-dual by the functor $\Mod_k(-,k)$, the opposite category of $\Bimod{k}{\Z}'$ is thus isomorphic to pro-finitely generated free $k$-modules by the functors $\dual{(-)}$ and $\predual{(-)}$. The case for general $l$ follows because the forgetful functors $\fBimod{k}{l}' \to \fBimod{k}{\Z}'$ and $\Bimod{k}{l}' \to \Bimod{k}{\Z}'$ are faithful.
\end{proof}

\begin{lemma} \label{dualstrict}
The category $\Bimod{k}{l}$ is a $2$-bimodule over $\Mod_l$. The functor $\dual{(-)}$ is a left strict and the functor $\predual{(-)}$ is a right strict morphism of $2$-bimodules.
\end{lemma}
\begin{proof}
The graded homomorphism set $\Bimod{k}{l}(M,N)$ is naturally an $l$-module because $l$ is commutative. The left action of $\Mod_l$ is given by the usual tensor product $M \indprotensor_{l} N = M \otimes_l N$, which is again a left $l$-module by the commutativity of $l$. The right action is given by $\hom_l(M,N)$.

It is straightforward to check that $\predual{(-)}$ and $\dual{(-)}$ are enriched functors. To see that $\predual{(-)}$ is right strict, we need to verify that
\[
\predual{\hom_l(M,\hat M)} \cong \hom_l(M,\predual{\hat M}),
\]
which is true since both sides are given by $\Hom_l(M,\hat M(k))$. To see that $\dual{(-)}$ is left strict, we check the identity
\[
M \indprotensor_l \dual{N} \cong \dual{(M \indprotensor_l N)}
\]
by noting that the left hand side is represented by $\hom_l(M,\hom_k(N,k))$, whereas the right hand side is represented by $\hom_k(M \indprotensor_l N,k)$. These are equal by the standard hom-tensor adjunction.
\end{proof}

Note that the $\Mod_l$-enrichment of $\Bimod{k}{l}'$ and $\fBimod{k}{l}'$ does not descend to a $2$-bimodule structure on these categories. For example, if $N \in \Bimod{k}{l}'$ then $M \otimes_l N$ is in general not flat unless $M$ is a flat $l$-module. On the other hand, for $M$, $N \in \Bimod{k}{l}'$, $\Bimod{k}{l}(M,N)$ need not be a flat $l$-module.

Let $\hat M \in \fBimod{k}{l}$ and $N \in \Bimod{m}{l}$. In Section~\ref{sec:fbimodstructure} we defined an object $N \indprotensor_l \hat M \in \fBimod{k}{m}$ as part of the left module structure of $\fBimod{k}{l}$ over $\Mod_l$.

\begin{lemma} \label{dualizablelemma}
Let $\hat M \in \fBimod{k}{l}$ and $N \in \Bimod{m}{l}$. Then there is a natural map
\[
\Phi\colon N \indprotensor_l \hat M \to (\dual{N}) \circ_l \hat M \in \fModSch{k}{m}
\]
which is an isomorphism if $N \in \Bimod{m}{l}'$. Similarly, there is a natural map
\[
\Psi\colon N \circ_l \predual{\hat M} \to \predual{(N \indprotensor_l \hat M)}
\]
which is an isomorphism if $N \in \Bimod{m}{l}'$.
\end{lemma}
\begin{proof}
For $\Phi$, it suffices by naturality to consider the case $m=\Z$. The adjoint of the identity on $\dual{N} = N \indprotensor_k \Spf{K}$ give a map
\[
J_l \to \hom_l(N,\dual{N}).
\]
Taking the $\circ_l$-product with $\hat M$ and applying the map $\beta$ of \eqref{betamap} coming from the $2$-algebra structure of $\fModSch{k}{l}$ over $\Mod_l$, we obtain a map
\[
\hat M = J_l \circ \hat M \to \hom_l(N,\dual{N}) \circ \hat M \xrightarrow{\beta} \hom_l(N,(\dual{N}) \circ \hat M)
\]
which is adjoint to the desired map $\Phi$.

Note that as functors of $N$, both sides of $\Phi$ preserve colimits. Thus to show that the map is an isomorphism for $l$-flat $N$, it suffices to let $N = l^n$ be a single finitely generated free $l$-module. In that case $\Phi$ is the identity on $\hat M^n$.

The map $\Psi$ is given by the canonical map $N \otimes_l \hat M(k) \to (N \indprotensor_l \hat M)(k)$ given by Lemma~\ref{freebimodadjunction}. Since $T \mapsto N \otimes \hat M(T)$ is representable if $N$ is $l$-flat, this map is an isomorphism for $N \in \Bimod{l}{m}'$.
\end{proof}

Recall from Ex.~\ref{exa:bimodtwomonoidal} that $\Bimod{k}{k}$ is a $2$-monoidal category with respect to the two-sided tensor product $\btensor{k}{k}$ and the left-right tensor product $\circ_k$.
 
\begin{thm}\label{dualbimonoidal}
The functors $\predual{(-)}$ and $\dual{(-)}$ are bilax monoidal functors.
\end{thm}
\begin{proof}
The unit of $\circ$ in $\fBimod{k}{k}$ is $J_k = \Spf{k}$, the unit of $\circ$ in $\Bimod{k}{k}$ is $k$, and we have that $\predual{\Spf{k}} = k$ and $\dual{k} = \Spf{k}$. Thus we find that $\dual{(-)}$ and $\predual{(-)}$ strictly preserve the units.

We now show that $\dual{(-)}$ is lax monoidal with respect to the $\circ$-products. There is a natural transformation
\begin{align*}
\dual{(M \circ N)} =& \Spf{\hom_k(M \circ N,k)}\\
 = &\Spf{\hom_k(M,\hom_k(N,k))} = M \indprotensor_{k} \dual{N} \xrightarrow{\text{Lemma~\ref{dualizablelemma}}} \dual{M} \circ \dual{N},
\end{align*}
and this map is an isomorphism if $M \in \Bimod{k}{k}'$.

Next we show that $\predual{(-)}$ is lax monoidal with respect to the $\circ_k$-products. There is a natural transformation
\begin{multline*}
\predual{\hat M} \circ_k \predual{\hat N} =  \hat M(k) \circ_k \hat N(k)
\to \Pro-\Mod_k(\Reg{\hat M} \otimes_k \Reg{\hat N},k) \\
 \xrightarrow{\text{Lemma~\ref{protensornattrf}\eqref{proadjoin}}}   \Pro-\Mod_k(\Reg{\hat M},\underline\Pro_k(\Reg{\hat N},k)) = \predual{(\hat M \circ_k \hat N)}.
\end{multline*}
The arrow is an isomorphism if $\hat M \in \fModSch{k}{k}'$.

We now study the compatibility with the other monoidal structure, i.e. $\otimes_k$ on $\fBimod{k}{k}$ and $\btensor{k}{k}$ on $\Bimod{k}{k}$. Note that the unit in $\fBimod{k}{k}$ is $\tensunit{k}{k} = \Spf{\hom(k,k)}$, whereas the unit in $\Bimod{k}{k}$ is $k \otimes k$. We have
\[
\Reg{\dual{(k \otimes k)}} = \hom_k(k \otimes k,k) = \hom_k(k,k).
\]
For $\predual{(-)}$, we have a comparison map
\[
k \otimes k \to \predual{\Spf{\hom_k(k,k)}} = \hom_k(\hom_k(k,k),k)
\]
which is adjoint to 
\[
k \otimes k \indprotensor_k \hom_k(k,k) = k \indprotensor_k \hom_k(k,k) \xrightarrow{\text{evaluate}} k,
\]
but this map is not necessarily an isomorphism. It is if $k \otimes k$ is flat over $k$ (in this case, $\hom(k,k) = \hom_k(k \otimes k, k) \in \fBimod{k}{k}'$).

To show that $\dual{(-)}$ is oplax monoidal with respect to $\otimes_k$ and $\btensor{k}{k}$, consider the diagram 
\[
\begin{tikzpicture}
	\matrix (m) [matrix of math nodes, row sep=2em, column sep=2em, text height=1.5ex, text depth=0.25ex]
	{
		\dual{(M \btensor{k}{k} N)} & \dual{(M \otimes_k N)} & \dual{(k  \otimes M \otimes_k N)}\\
		&	&	k \indprotensor_\Z(\dual{(M \otimes_k N)})\\
		\dual{M} \otimes_k \dual{N} & \dual{M} \otimes_\Z \dual{N} & k \indprotensor_\Z(\dual{M} \otimes_\Z \dual{N})\\
	};
	\path[<-,font=\scriptsize]
	(m-1-1)	edge (m-1-2)
	([yshift=-1mm]m-1-2.east)		edge ([yshift=-1mm]m-1-3.west)
	([yshift=1mm]m-1-2.east)		edge ([yshift=1mm]m-1-3.west)
	(m-3-1)	edge (m-3-2)
	([yshift=-1mm]m-3-2.east)		edge ([yshift=-1mm]m-3-3.west)
	([yshift=1mm]m-3-2.east)		edge ([yshift=1mm]m-3-3.west)
	(m-3-2)	edge (m-1-2)
	(m-3-3)	edge (m-2-3)
	(m-3-1)	edge[dashed] (m-1-1);
	\path
	(m-2-3)	edge[double,double equal sign distance] node[auto]{\eqref{dualstrict}} (m-1-3);
\end{tikzpicture}
\]
Both rows are coequalizers, and the vertical maps are iso if $M$, $N \in \Bimod{k}{k}'$.
 
The oplax monoidal structure for $\predual{(-)}$ is given by the canonical map
\[
\predual{\hat M} \btensor{k}{k} \predual{\hat N} = \hat M(k) \btensor{k}{k} \hat N(k)
 \to (\hat M \otimes_k \hat N)(k) = \predual{(\hat M \otimes_k \hat N)}.
\]
By duality (Lemma~\ref{doubledual}), since $\dual{(-)}$ is strongly monoidal on $\Bimod{k}{l}'$, so is $\predual{(-)}$ on $\fBimod{k}{l}'$.
\end{proof}

We summarize the main results of this section in the following theorem.
\begin{thm} \label{thm:dualsummary}
Assume that $k \otimes k$ is flat over $k$. Then the categories $\fBimod{k}{k}' \subseteq \fBimod{k}{k}$ and $\Bimod{k}{k}' \subseteq \Bimod{k}{k}$ are full $2$-monoidal subcategories, and the functors $\predual{(-)}$ and $\dual{(-)}$ are inverse strong bimonoidal equivalences between them.
\end{thm}
\begin{proof}
The only point left to show is that $\fBimod{k}{k}'$ and $\Bimod{k}{k}'$ are closed under the two monoidal structures. This is obvious for $\Bimod{k}{k}'$: if $M$, $N$ are flat as right $k$-modules then so are $M \circ_l N$ and $M \btensor{k}{k} N$. It follows for $\fBimod{k}{k}'$ by Thm.~\ref{dualbimonoidal} and Lemma~\ref{dualizablelemma}.
\end{proof}

\section{Formal plethories and the unstable Adams spectral sequence} \label{section:anss}

Let $E$ be a multiplicative cohomology theory such that $E_*$ is a Pr\"ufer domain and $E^*\underline E_*$ is pro-flat. Let $X$ be a space. We denote by $X_E$ the formal scheme
$\Spf{\hat E^*(X)}$. By Thm.~\ref{thm:topplethory}, $\underline E_E = \Spf{\hat E^*\underline E_*}$ represents a formal plethory and $X_E$ is a comodule over it. This short section is meant to provide a basic connection between this structure and the $E$-based unstable Adams-Novikov spectral sequence for $X$. Further study and applications will appear in another paper.

The unstable Adams-Novikov spectral sequence can be constructed monadically \cite{bendersky-curtis-miller,bendersky-thompson}: define a monad $E\colon \Top \to \Top$ by $E(X) = \Omega^\infty(E \sm X)$. This gives rise to a cosimplicial bar construction
\[
X \to E^\bullet(X) = \{ E(X) \rightrightarrows E(E(X)) \cdots \},
\]
and the unstable Adams-Novikov spectral sequence arises as the homotopy Bousfield-Kan spectral sequence associated to this cosimplicial space \cite{yellowmonster}:
\[
E_2^{s,t} = \pi^s (\pi_t E^\bullet(X)) \Longrightarrow \pi_{t-s} X\hat{{}_E} 
\]
This $E_2$-term can be described as a monadic derived functor under the assumption that both $E_*(X)$ and $E_*(\underline E_*)$ are free $E_*$-modules \cite[Thm.~6.17]{bendersky-curtis-miller} as
\[
E_2^{s,t} = \Ext^s_{\G}(E_*(\S^t),E_*(X))
\]
where $\G$ is a comonad on free $E_*$-modules characterized by $\G(E_*(X)) = \pi_* E(X)$. Coalgebras over this comonad are unstable coalgebras for $E$-homology, and any $E_*$-module of the form $E_*(X)$ is a coalgebra over $\G$. The functor $\Ext^s_\G$ denotes the $s$th non-additive derived functor of $\Hom_\G(-,-)$, i.e. of $\G$-coalgebra morphisms with respect to the injective models given by $\G(S)$ for any space $S$ \cite{bousfield:cosimplicial-resolutions}. The advantage of this description is that one is no longer tied to the bar resolution for computing the $E_2$-term, but can take any \emph{weak} resolution, i.e. any contractible  cosimplicial object in $\G$-coalgebras which is levelwise $\G$-injective \cite[Thm.~6.2]{bousfield:cosimplicial-resolutions}.

To describe the unstable Adams-Novikov spectral sequence in terms of plethories and their comodules, we need the following ``nonlinear K\"unneth'' theorem:

\begin{prop}\label{kappaprop}
Let $E$, $F$ be multiplicative cohomology theories such that $E_*$ is a Pr\"ufer domain, $\hat E^*(\underline F)$ is pro-flat, and let $X$ be a pointed space or spectrum. Then there is a natural map of formal  $F_*$-module schemes over $E_*$:
\[
\kappa\colon (\Omega^\infty(F \sm X))_E \to X_F \circ_{F_*} \underline F_E
\]
which is an isomorphism if $X$ is a filtered colimit of finite CW-spaces $K$ such that $F_*(K)$ is a free $F_*$-module. 
\end{prop}
\begin{proof}
By writing $X$ as a filtered colimit of finite CW-spaces $K$, is suffices to consider the case where $X$ is finite. We will work in slightly greater generality and produce a morphism
\[
\kappa\colon (\Omega^\infty M)_E \to M_F \circ_{F_*} \underline F_E = \hom_{F_*}([M,F]^F,\underline F_E)
\]
for any finite $F$-module spectrum $M$, where $M_F = \Spf{[M,F]^F]}$ denotes the formal $F_*$-bimodule of graded $F$-module spectrum maps $M \to F$. By adjunction, the map $\kappa$ is given as the adjoint of
\[
[M,F]^F \xrightarrow{\Omega^\infty} [\Omega^\infty M,\underline F_*] \to \Sch{E_*}((\Omega^\infty M)_E,\underline F_E).
\]
Note that source and target of $\kappa$ send filtered colimits in $M$ to limits and finite direct sums to tensor products, thus it suffices to show that $\kappa$ is an isomorphism if $M$ is finitely generated free on one generator. But if $M = \Sigma^n F$ then
\[
(\Omega^\infty)_E \cong (\underline F_n)_E \xrightarrow[\cong]{\kappa} \hom_{F_*}(F_{*-n},\underline F_E).
\]
\end{proof}

By Lemma~\ref{odotleftfree}, if $X$ is a space, $X_F \circ_{F_*} \underline F_E \cong (\Cof{X_F}) \circ_{F_*} \underline F_E$, where the pairing $\circ_{F_*}$ on the left is the pairing of a formal module scheme with a formal bimodule, whereas the pairing on the right is the pairing of a formal algebra scheme with a formal module scheme.

In particular, using Theorem~\ref{thm:dualsummary}, we obtain:
\begin{thm}
Let $E$ be a ring spectrum and $X$ a space such that $X$ and every $\underline E_n$ are filtered colimits of finite CW-complexes $K$ with $E_*(K)$ free and such that $E_*$ is Pr\"ufer. Denote by $M \in \fBimod{E_*}{E_*}'$ the formal module scheme represented by the pro-finitely generated free $E_*$-module $\hat E^*(X)$. Then
\[
\G(E_*X) \cong \predual{\left(\Cof(\hat M) \circ_{E_*} \hat P_E\right)}.
\]
\end{thm}
\begin{proof}
If $E_*$ is a Pr\"ufer domain then $E_* \otimes_\Z E_*$ is flat over $E_*$. Indeed, it suffices to see that $F \otimes_\Z E_* = E_*$ for any cyclic subgroup $F$ of $E_*$. This is clear for $F=\Z$. For $F=\Z/n\Z$, the fact that $na=0$ for some $a \in E_*$ means that $n=0$ in $E_*$ since $E_*$ is a domain, hence $F \otimes E_* = E_*$. 

By Prop.~\ref{kappaprop},
\[
\Cof(\hat M) \circ_{E_*} \hat P_E \cong \Spf{\hat E^*(\Omega^\infty(E \sm X)))} 
\]
and hence by Thm.~\ref{thm:dualsummary},
\[
\predual{(\Cof(\hat M) \circ_{E_*} \hat P_E)} \cong E_*(\Omega^\infty(E \sm X)) = \G(E_*X).
\]
\end{proof}

Thus the standard unstable Adams resolution is given by repeatedly applying the comonad $\hat M \mapsto \Cof(\hat M) \circ_{E_*} \hat P_E$ to $M=\hat E^*(X)$. There is another, smaller, resolution of $X_E$ when considered as a formal scheme, to wit, the comonad resolution $X_E \mapsto X_E \circ_{E_*} \hat P_E$, and it is natural to wonder whether this can be used to replace the standard resolution. However, $X_E A \circ_{E_*} \hat P_E$ is not injective in general. The following corollary records a situation where it is.

\begin{corollary}\label{cor:anss}
Assume that $X$ and $E$ satisfy the properties of the previous theorem, and that in addition, $E^*(X)$ and all $E^*(\underline E_n)$ are free graded commutative algebras. Write as before $\hat P_E = \Spf{\hat E^*(\underline E_*)}$ for the formal plethory arising from $E$ and $X_E = \Spf{\hat E^*(X)}$ for the coalgebra over it arising from $X$. Then the $E_2$-term of the unstable Adams-Novikov spectral sequence is given by
\[
E_2^{s,t} = \Ext^s_{\hat P_E}((\S^t)_E,X_E).
\]
Here $\Ext^s_{\hat P_E}$ denotes $s$th cohomotopy group of the cosimplicial group (just a set for $t=0$)
\[
\Hom_{\hat P_E}((\S^t)_E,X_E \circ_{E_*} \hat P_E \rightrightarrows X_E \circ_{E_*} \hat P_E \circ_{E_*} \hat P_E \;\;\cdots).
\]
\end{corollary} 
\begin{proof}
The cosimplicial resolution $X_E \to X_E \circ_{E_*} \hat P_E \rightrightarrows \cdots$ is acylic because the counit $\hat P_E \to J_{E_*}$ provides a contracting extra degeneracy. By \cite[Thm.~6.2]{bousfield:cosimplicial-resolutions}, it suffices to show that $X_E \circ_{E_*} (\hat P_E)^{\circ_{E_*} n}$ is injective with respect to the injective models $\Cof(\hat M) \circ_{E_*} \hat P_E$ for each $n \geq 1$. For $n=1$, this follows from the assumption that $E^*(X)$ is free, i.e. $X_E = \Cof(\Indec{X_E})$. Lemma~\ref{Symhatodotmonoidal} guarantees that $\Cof(\hat M) \circ_{E_*} \Cof(\hat N) \cong \Cof(\hat M \circ_{E_*} \hat N)$ and in particular,
\[
\Cof(\hat M) \circ_{E_*} \hat P_E \cong \Cof(\hat M) \circ_{E_*} \Cof\Indec{\hat P_E} \cong \Cof(\hat M \circ_{E_*} \Indec{\hat P_E}).
\]
This provides the inductive step to show that every $X_E \circ_{E_*} (\hat P_E)^{\circ_{E_*} n}$ is in fact one of the injective models.
\end{proof}

\begin{example}
The conditions of Cor.~\ref{cor:anss} are satisfied for mod-$p$ homology or integral $K$-theory and $X = \S^{2n+1}$ . The description of the $E_2$-term in the corollary coincides in the latter case with the description of \cite[Thm.~4.9]{bendersky-thompson} in terms of unstable comodules.
\end{example}

However, the condition on freeness of $E^*(\underline E_*)$ of Cor.~\ref{cor:anss} is not satisfied for $E=K(n)$, the Morava $K$-theories.


\begin{appendices}

\section{Pro-categories and lattices}\label{colimitsapp}

In this appendix, I will review some background results about pro- and ind-categories.

Throughout, I will make use of ends and coends. We denote the end of a functor $F\colon I^\op \otimes I \to \C$ by
\[
\int\limits_i F(i,i) = \eq\left(\prod_{i \in I} F(i,i) \rightrightarrows \prod_{i \to j \in I} F(i,j)\right)
\]
and the coend of a functor $G\colon I \otimes I^\op \to \C$ by
\[
\int\limits^i F(i,i) = \coeq\left( \coprod_{i \to j \in I} F(i,j)\rightrightarrows \coprod_{i \in I} F(i,i) \right).
\]

Recall that a category $I$ is cofiltered if  each finite diagram $X\colon F \to I$ has a cone, i.e. an object $i$ together with a natural transformation $\const_i \to X$ in the category of functors from $F$ to $I$. If $\C$ is any category, the categories $\Pro-\C$ (or $\Ind-\C$) have as objects pairs $(I,X)$ where $I$ is a small cofiltered (filtered in the case of $\Ind-\C$) category and $X\colon I \to \C$ is a diagram; morphisms are defined by
\begin{align*}
\Pro-\C((I,X),(J,Y)) &= \lim_{j \in J} \colim_{i \in I} \C(X(i),Y(j))
\intertext{and}
\Ind-\C((I,X),(J,Y)) &= \lim_{i \in I} \colim_{j \in J} \C(X(i),Y(j)).
\end{align*}
It is easy to show (cf. for example \cite[Thm. 2.1.6]{edwards-hastings:cech-steenrod}) that $\Pro-\C$ is equivalent to the subcategory of objects indexed by cofiltered posets in which every ascending chain is finite. The dual of this property is usually called ``cofinite'', but I will refrain from calling this property ``cocofinite'' or ``finite'' and use the term ``noetherian'' and the dual property ``artinian.''

If $I$ is a meet-semilattice (a poset with all finite limits), then $I$ is in particular a cofiltered poset -- the meet $\lim F$ of a finite set $F \subseteq X$ is a cone -- but the converse is not true. Let $\Lat$ be the category of all noetherian meet-semilattices. It will be technically convenient to work with the full subcategory $\Lat(\C)$ of  $\Pro-\C$ generated by objects indexed by posets in $\Lat$. The following lemma shows that this will usually not be a loss of generality.

\begin{lemma} \label{latrepresents}
Let $\C$ be a category closed under finite limits. Then the natural inclusion $\Lat(\C) \hookrightarrow \Pro-\C$ is an equivalence.
\end{lemma}

\begin{proof}
We already know that the inclusion $\Lat(\C) \to \Pro-\C$ is full and faithful, and it remains to show that every object $X\colon I \to \C$ in $\Pro-\C$ is isomorphic to one in $\Lat(\C)$. We may assume that $I$ is a noetherian cofiltered poset. We define a relation $\leq$ on the set of finite subsets of $I$ by $F \leq F'$ iff $F$ is cofinal in $F'$, i.e. for each $x \in F'$ there is a $y \in F$ such that $y \leq x$. This relation is transitive and reflexive and thus induces a partial order on the set $\Fin(I) = \{ F \subseteq I \mid F \text{ finite}\}/\sim$, where $F \sim F' \iff F \leq F'$ and $F' \leq F$. The poset $\Fin(I)$ is noetherian if $I$ is. Furthermore, $\Fin(I)$ is closed under meets: $F \wedge F' = F \cup F'$.

There is a canonical inclusion functor $\iota\colon I \to \Fin(I)$ given by singletons. 

Now let $X \colon I \to \C$ be a diagram representing an object in $\Pro-\C$. Consider $\RKan_{\iota} X$ given by
\[
(\RKan_{\iota} X)(F) =\lim_F X,
\]
which exists because $\C$ was assumed to have finite limits. I claim that $X$ and $\RKan_{\iota} X$ are isomorphic in $\Pro-\C$. There is a canonical map $\RKan_{\iota} X \to X$ given by
\[
X(\{i\}) = X(i).
\]
To see this is an isomorphism in $\Pro-\C$, we need to construct for each $F \in \Fin(I)$ an object $i' \in I$ and a map $X(i') \to \lim_F X$ such that suitable diagrams commute. But since $I$ is cofiltered, there is a cone $i' \to F$ and compatible maps $X(i') \to X(i)$ for each $i \in F$, thus a map $X(i') \to \lim X \circ F$.
\end{proof}

We need to simplify even further. If $I$, $J$ are two noetherian semilattices, denote by $I^J$ the set of all monotonic functions from $J$ to $I$. There is a partial order on $I^J$ given by $f \leq f'$ iff $f(j) \leq f(j')$ for all $j \in J$. In fact, $I^J$ is again a semilattice: $(f \wedge g)(i) = f(i) \wedge g(i)$. However, $I^J$ is usually not noetherian.

Let $\Lat'(\C)$ be the category with the same objects as $\Lat(\C)$, but where the morphisms are defined as follows: given two objects $X\colon I \to \C$, $Y\colon J \to \C$, 
\[
\Lat'(\C)(X,Y) = \colim_{f \in I^J} \int\limits_{j \in J} \C(X(f(j)),Y(j)).
\]
\begin{lemma} \label{functorialize}
The canonical functor $\phi\colon \Lat'(\C) \to \Lat(\C)$, given by the identity on objects and on morphisms, for $f \in I^J$ and $j_0 \in J$, by the map
\[
\begin{array}{ccc}
 \int\limits_{j} \C(X(f(j)),Y(j)) &\to&  \colim_i \C(X(i),Y(j_0)) \\
\left\{\alpha_j\colon X(f(j))\to Y(j)\right\}_j & \mapsto&  \alpha_{j_0}
\end{array}
\]
is an equivalence of categories.
\end{lemma}
\begin{proof}
As described for instance in \cite{deleanu-hilton:borsuk-shape}, a morphism in $\Lat(\C)$ is given by a not necessarily monotonic function $f\colon J \to I$ and compatible maps $X(f(j)) \to Y(j)$ subject to the condition that for each arrow $j \leq j'$ in $J$, there is an object $i \in I$ such that $i \leq f(j)$ and $i \leq f(j')$. Given such a nonmonotonic function, we can produce a monotonic function $\tilde f\colon J \to I$ by
\[
\tilde f(j) = \lim_{j' \geq j} f(j'),
\]
using the noetherian condition on $J$ (therefore the set $\{j' \mid j' \geq j\}$ is finite) and the lattice condition on $I$ (finite limits exist). We have that $\tilde f \leq f$ and thus we get an induced compatible set of maps $\tilde \alpha_j\colon X(\tilde f(j)) \to Y(j)$ representing the same map as $\alpha$ in $\Lat(\C)$. Thus $\phi$ is full. For faithfulness, let $f,\; g\colon J \to I$ be two monotonic maps and $\alpha_j\colon X(f(j)) \to Y(j)$ and $\beta_j\colon X(g(j)) \to Y(j)$ two maps such that $\phi(\alpha) = \phi(\beta)$. This means that for every $j$ there exists an element $h(j) \leq f(j) \wedge g(j)$ such that $\alpha_j|_{X(h(j))} = \beta_j|_{X(h(j))} =: \gamma_j$. Again, by possibly choosing smaller $h(j)$, we can assume that $h$ is monotonic. Thus $h \leq f$ and $h \leq g$, and $\gamma \leq \alpha$, $\gamma \leq \beta$. Therefore, $\alpha$ and $\beta$ represent the same class of maps in $\colim\limits_{f \in I^J} \int\limits_{j \in J} \C(X(f(j)),Y(j)).$
\end{proof}

\section{Pro- and ind-categories and their enrichments} \label{indproapp}

Let $(\V,\otimes)$ be a $2$-ring and $\C$ a $2$-bimodule over $\V$ (cf. Def.~\ref{def:tworingtwomod}). In this section, we will study the structure of the categories $\Ind-\V$ and $\Pro-\C$ with respect to monoidality and enrichments. The examples we have in mind are:
\begin{enumerate}
\item For a  graded commutative ring $k$, $\V=\Mod_\Z$ (graded abelian groups) and $\C=\Mod_k$;
\item $\V=\Set^\Z$, the category of graded sets, and $\C=\Alg_k$ for a graded commutative ring $k$.
\end{enumerate}

\begin{lemma}
Let $\V$ be a $2$-ring. Then so is the category $\Ind-\V$, and the inclusion $\V \to \Ind-\V$ and colimit $\Ind-\V\to \V$ are strict monoidal functors. 
\end{lemma}
\begin{proof}
The fact that $\Ind-\V$ is cocomplete follows from the fact that $\Ind-\V$ always has filtered colimits, and that finite coproducts and coequalizers can be computed levelwise. By \cite{isaksen:limits-colimits}, $\Ind-\V$ is also complete. (The appendix of \cite{artin-mazur:etale} is often cited for this fact, but it only proves it for small categories $\V$.)

Let $I_i \xrightarrow{X_i} \V$ represent objects of $\Ind-\V$ ($i=1,2$), where $I_i$ are filtered categories. Then the symmetric monoidal structure on $\Ind-\V$ is defined by
\[
X_1 \otimes X_2\colon I_1 \times I_2 \to \V, \quad (X_1 \otimes X_2)(i_1,i_2) = X_1(i_1) \otimes X_2(i_2).
\]
Since $\otimes$ is closed in $\V$, it commutes with all colimits in $\V$, and thus $\colim\colon \Ind-\V \to \V$ is strict monoidal.

\medskip

We define an internal hom object $\underline{\Ind}(X_1,X_2)$ as follows: we may assume the indexing categories $I_1$, $I_2$ are artinian join-semilattices (by the dual of Lemma~\ref{latrepresents}; we again use the notation $\Lat(\V)$ for the full subcategory of $\Ind-\C$ of objects indexed by such semilattices). Then the poset of monotonic maps from $I_1$ to $I_2$ is also a lattice. This will be the indexing set of $\underline{\Ind}(X_1,X_2)$. For such an $\alpha \in {I_2}^{I_1}$, $\underline{\Ind}(X_1,X_2)$ is given by
\[
\underline\Ind(X_1,X_2)(\alpha) = \int\limits_{i_1} \C(X_1(i_1),X_2(\alpha(i_1))).
\]

To see that $\otimes$ and $\underline\Ind$ are adjoint, we compute
\begin{align*}
\Ind-\V(X_1,\underline{\Ind}(X_2,X_3)) = & \lim_{i_1} \colim_{\alpha \in {I_3}^{I_2}} \V\biggl(X_1(i_1),\int\limits_{i_2} \C(X_2(i_2),X_3(\alpha(i_2)))\biggr)\\
=&  \lim_{i_1} \colim_{\alpha \in {I_3}^{I_2}} \int\limits_{i_2} \V\biggl(X_1(i_1), \C(X_2(i_2),X_3(\alpha(i_2)))\biggr)\\
\underset{\text{Lemma~\ref{functorialize}}}= &  \lim_{i_1} \lim_{i_2} \colim_{i_3} \V\biggl(X_1(i_1), \C(X_2(i_2),X_3(i_3))\biggr)\\
=& \lim_{i_1,i_2} \colim_{i_3} \V(X_1(i_1) \otimes X_2(i_2), X_3(i_3))\\
 =& \Ind-\V(X_1 \otimes X_2,X_3).
\end{align*}
\end{proof}

This construction leads us out of the category $\Lat(\C)$ because ${I_2}^{I_1}$ is not artinian, but we can always apply the equivalence $\Ind-\C \to \Lat(\C)$ to get back an isomorphic internal hom object in $\Lat(\C)$.

\bigskip

Unfortunately, the category $\Pro-\C$ is not a $2$-ring even if $\C$ is. The analogous definition of a symmetric monoidal structure $(X_1 \otimes X_2)(i_1,i_2) = X_1(i_1) \otimes X_2(i_2)$ on $\Pro-\C$ is unproblematic, but this structure is not closed. However, $\Pro-\C$ is a $2$-bimodule over $\Ind-\V$. We will first define the structure and then prove that it gives rise to an enrichment.

\begin{defn}$\;$
\begin{enumerate}
\item
Define a functor $\indprotensor\colon \Ind-\V \times \Pro-\C \to \Pro-\C$ as follows.
Let  $K \in \Ind-\V$ be indexed by an artinian join-semilattice $J$ and $X \in \Pro-\C$ be indexed by a noetherian meet-semilattice $I$. Then $K \indprotensor X$ is indexed by the meet-semilattice $I^{J^\op}$, and
\[
(K \indprotensor X)(\alpha) = \int\limits^{j \in J} K(j) \indprotensor X(\alpha(j))) \quad \text{for } \alpha\colon J^\op \to I.
\]
\item Define a functor $\underline\Ind\colon \Pro-\C \times \Pro-\C \to \Ind-\V$ as follows. Let $X$, $Y \in \Pro-\C$ be indexed by noetherian meet-semilattice $I$ and $J$, respectively. The object $\underline{\Ind}(X,Y) \in \Ind-\V$ is indexed by the opposite lattice of the meet-semilattice $I^J$ and is given by
\[
\underline\Ind(X,Y)(\alpha) = \int\limits_j \C(X(\alpha(j)),Y(j))  \quad \text{for } \alpha\colon J \to I.
\]
\item Define a functor $\hom\colon \Ind-\V \times \Pro-\C \to \Pro-\C$ by assigning to $K\colon I \to \C$ and $X\colon J \to \C$, where $I$ is cofiltered and $J$ is filtered, the object indexed by $I^\op \times J$ and given by
\[
\hom(K,X)(i,j) = \hom(K(i),X(j)).
\]
\end{enumerate}
\end{defn}

\begin{lemma} \label{indproenrichment}
With the structure given above, the category $\Pro-\C$ is a $2$-bimodule over $\Ind-\V$.
\end{lemma}
\begin{proof}
Let $X$, $Y \in \Pro-\C$ and $K$, $L \in \Ind-\V$. We need to see:
\begin{enumerate}
\item $(L \otimes K) \indprotensor X \cong L \indprotensor (K \indprotensor X)$; \label{indprotensor}
\item $\Pro-\C(K \indprotensor X,Y) \cong \Ind-\V(K,\underline{\Ind}(X,Y))$; \label{indproenrich}
\item $\Pro-\C(K \indprotensor X,Y) \cong \Pro-\C(X,\hom(K,Y))$. \label{indprocotensor}
\end{enumerate}

\eqref{indprotensor}:
Let $I \xrightarrow{X} \C$, $J_1 \xrightarrow{L} \V$, and $J_2 \xrightarrow{K} \V$ be representations. Then we have for $\alpha\colon J_1^\op \times J_2^\op \to I$:
\begin{align*}
\left((L \otimes K) \indprotensor X \right)(\alpha) = & \int\limits^{\substack{j_1 \in J_1\\j_2 \in J_2}} L(j_1) \otimes K(j_2) \indprotensor X(\alpha(j_1,j_2))\\
= & \int\limits^{j_2} \Bigl(L(j_1) \indprotensor \int\limits^{j_1} K(j_2) \indprotensor X(\alpha(j_1,j_2))\Bigr)\\
= & \left(L \indprotensor (K \indprotensor X)\right)(\alpha^\#),
\end{align*}
where $\alpha^\#\colon J_2^\op \to I^{J_1}$ is the adjoint of $\alpha$.

\noindent\eqref{indproenrich}
As before, pick representatives $I_1 \xrightarrow{X} \C$, $I_2 \xrightarrow{Y} \C$, $J \xrightarrow{K} \V$.
We compute
\begin{align*}
\Pro-\C(K \indprotensor X,Y) = & \lim_{i_2} \colim_{\alpha\colon J^\op \to I_1} \C\left(\int^j K(j) \indprotensor X(\alpha(j)),Y(i_2)\right)\\
\cong & \lim_{i_2} \colim_{\alpha\colon J^\op \to I_1} \int_j \C\left(K(j) \indprotensor X(\alpha(j)),Y(i_2)\right)\\
= & \lim_{i_2} \lim_{j} \colim_{i_1} \C(K(j) \indprotensor X(i_1),Y(i_2))\tag{*}\\
= &  \lim_{i_2,j} \colim_{i_1} \V\left(K(j),\C\left(X\left(i_1\right),Y(i_2\right)\right)\\
= & \lim_j \colim_{\beta\colon I_2 \to I_1} \int_{i_2} \V\biggl(K(j),\C\Bigl(X\bigl(\beta\left(i_2\right)\bigr),Y(i_2)\Bigr)\biggr)\\
= & \lim_j \colim_{\beta\colon I_2 \to I_1}\V(K(j),\underline{\Ind}(X,Y)(\beta))\\
= &  \Ind-\V(K,\underline{\Ind}(X,Y)).
\end{align*}

\noindent\eqref{indprocotensor}
Pick representatives as in \eqref{indproenrich}. We pick up the computation at (*):
\begin{align*}
\lim_{i_2} \lim_{j} \colim_{i_1} \C(K(j) \indprotensor X(i_1),Y(i_2)) = & \lim_{i_2,j} \colim_{i_1} \C\left(X(i_1),\hom\left(K\left(j\right),Y(i_2\right)\right)\\
= & \Pro-\C(X,\hom(K,Y)).
\end{align*}
\end{proof}

Now assume that $(\C,\boxtimes,I_\C)$ is a $2$-algebra over $\V$. We define a symmetric monoidal structure $\boxtimes$ on $\Pro-\C$ by $(X \boxtimes Y)(i,j) = X(i) \boxtimes Y(j)$. The unit is the constant object $I_\C$. Furthermore, if $\C$ is closed (with internal homomorphism object $\map_\C$) then for representations $X\colon I \to \C$ and $Y \colon J \to \C$, there is a homomorphism object $\underline \Pro(X,Y)$ indexed by $J$ and given by
\[
\underline \Pro(X,Y)(j) = \colim_i \map_\C(X(i),Y(j)).
\]

\begin{lemma}\label{protensornattrf}
\begin{enumerate}
\item The category $\Pro-\C$ with the structure above is a $2$-algebra over $\Ind-\V$. More precisely,
Let $X$, $Y$, $Z \in \Pro-\C$ and $K \in \Ind-\V$. Then there is a natural map
\[
\alpha\colon K \indprotensor (X \boxtimes Y) \to (K \indprotensor X) \boxtimes Y
\]
which is an isomorphism if $Y$ is pro-constant. \label{tensortensortrf}
\item If $\C$ is also closed with internal homomorphism object $\map_\C$ then there is a natural morphism
\[
\Pro-\C(X \boxtimes Y,Z) \to \Pro-\C(X,\underline{\Pro}(Y,Z))
\]
which is an isomorphism if $X$ consists of small objects, i.e. if $\C(X(i),-)$ commutes with directed colimits for all $i$, or if $Y$ is pro-constant. \label{proadjoin}
\end{enumerate}
\end{lemma}
\begin{proof}
\eqref{tensortensortrf}:
Pick representatives $I_1 \xrightarrow{X} \C$, $I_2 \xrightarrow{Y} \C$, $J \xrightarrow{K} \V$. For $\alpha\colon J^\op \to I_1$, $i_2 \in I_2$, we have
\begin{align*}
\left( (K \indprotensor X) \boxtimes Y\right)(\alpha,i_2) = & \biggl(\int\limits^{j \in J} K(j) \indprotensor X(\alpha(j))\biggr) \boxtimes Y(i_2)\\
= & \int\limits^{j \in J} \left(K(j) \indprotensor X(\alpha(j)) \boxtimes Y(i_2)\right)\\
= & \left(K \indprotensor (X \boxtimes Y)\right)(\alpha \times \const_{i_2}).
\end{align*}
Since $K \indprotensor (X \boxtimes Y)$ is indexed by the larger category of all functors $J^\op \to I_1 \times I_2$, this only defines a natural pro-map $K \indprotensor (X \boxtimes Y) \to (K \indprotensor X) \boxtimes Y$, which is an isomorphism if $I_2 = \{i_2\}$.

\noindent\eqref{proadjoin}: The natural map
\begin{align*}
\Pro-\C(X \boxtimes Y,Z) = & \lim_{k} \colim_{i,j} \C(X(i) \boxtimes Y(j),Z(k))\\
=& \lim_k \colim_{i,j} \C(X(i),\map_\C(Y(j),Z(k))\\
\to & \lim_k \colim_i \C(X(i),\colim_j \map_\C(Y(j),Z(k))\\
= & \Pro-\C(X,\underline{\Pro}(Y,Z)).
\end{align*}
is an isomorphism if $X(i)$ is small or $J = \{j_0\}$.
\end{proof}

\end{appendices}

\bibliographystyle{alpha}
\bibliography{bibliography}

\end{document}